\documentclass[pdf, 11pt, reqno]{amsart}

\usepackage{amsthm, amsmath, amssymb}
\usepackage[OT1]{fontenc}
\usepackage{comment}
\usepackage{bbm}
\usepackage{enumitem}

 \usepackage[top=1.23in, bottom=1.23in, left=1.23in, right=1.23in]{geometry}
\allowdisplaybreaks

\numberwithin{equation}{section}
\newtheorem{thm}{Theorem}[section]
\newtheorem{cor}[thm]{Corollary}

\newtheorem{prop}[thm]{Proposition}
\newtheorem{defn}[thm]{Definition}

\newtheorem{lem}[thm]{Lemma}

\theoremstyle{definition}
\newtheorem{rem}[thm]{Remark}
\newtheorem{ex}[thm]{Example}

\def\R{{\mathbb{R}}}
\def\T{{\mathbb{T}}}
\def\C{{\mathbb{C}}}
\def\N{{\mathbb{N}}}
\def\Z{{\mathbb{Z}}}
\def\S{{\mathbb{S}}}
\def\Nh{\mathbb{N}_h}

\renewcommand{\atop}[2]{\substack{{#1}\\{#2}}}
\newcommand{\Kl}{K_{\lambda}}

\newcommand{\supp}{{\rm supp}\,}
\newcommand{\sgn}{{\rm sgn}\,}
\renewcommand{\a}{\alpha}
\renewcommand{\b}{\beta}
\newcommand{\g}{\gamma}

\renewcommand{\d}{\delta}

\newcommand{\la}{\lambda}

\newcommand{\eps}{\varepsilon}
\newcommand{\e}{\varepsilon}

\renewcommand{\t}{\tau}

\newcommand{\te}{\theta}
\newcommand{\s}{\sigma}
\newcommand{\vp}{\varphi}

\newcommand{\8}{\infty}
\newcommand{\dpi}{2\pi i}

\newcommand{\vt}{\vartheta}

\newcommand{\ind}{\mathbbm{1}}
\renewcommand{\r}{\right}

\begin{document}
%\date{\today}

\begin{abstract}
We establish an asymptotic formula for the number of
lattice points in the sets 
\[
\mathbf S_{h_1, h_2, h_3}(\lambda):
=\{x\in\mathbb Z_+^3:\lfloor h_1(x_1)\rfloor+\lfloor h_2(x_2)\rfloor+\lfloor h_3(x_3)\rfloor=\lambda\}
\quad \text{with}\quad \lambda\in\mathbb Z_+;
\]
where functions  $h_1, h_2, h_3$ are constant multiples of  regularly
varying functions of the form $h(x):=x^c\ell_h(x)$, where the exponent $c>1$ (but
close to $1$) and a function $\ell_h(x)$ is taken from a certain wide class of slowly
varying functions. Taking $h_1(x)=h_2(x)=h_3(x)=x^c$ we will also derive 
an asymptotic formula for the number of
lattice points in the sets
\[
\mathbf S_{c}^3(\lambda)
:=
\{x \in \mathbb Z^3 : \lfloor |x_1|^c \rfloor + \lfloor |x_2|^c \rfloor + \lfloor |x_3|^c \rfloor= \lambda \}
\quad \text{with}\quad \lambda\in\mathbb Z_+;
\]
which can be thought of as a perturbation of the classical Waring problem in three variables. 

We will use the latter asymptotic formula to study, the main results of this paper,  norm and
pointwise convergence  of  the ergodic averages
\[
\frac{1}{\#\mathbf S_{c}^3(\lambda)}\sum_{n\in \mathbf S_{c}^3(\lambda)}f(T_1^{n_1}T_2^{n_2}T_3^{n_3}x) 
\quad \text{as}\quad \lambda\to\infty;
\]
where $T_1, T_2, T_3:X\to X$ are commuting invertible and
measure-preserving transformations of a $\sigma$-finite measure space
$(X, \nu)$ for any function  $f\in L^p(X)$ with $p>\frac{11-4c}{11-7c}$. Finally, we will study the equidistribution problem corresponding to the 
spheres $\mathbf S_{c}^3(\lambda)$.

\end{abstract}

\title[]{Lattice points problem, equidistribution
and ergodic theorems for certain  arithmetic spheres}

\author[A. Iosevich]{Alex Iosevich}
\address{Alex Iosevich,     \newline
			University of Rochester,
		Department of Mathematics, \newline
      RC Box 270138, Rochester, 
      NY 14627, USA
      }
\email{iosevich@math.rochester.edu}

\author[B. Langowski]{Bartosz Langowski}
\address{Bartosz Langowski, \newline
Indiana University,
Department of Mathematics, \newline 
831 East 3rd St., Bloomington, 
IN 47405, USA \newline
\indent and \newline			
Wroc\l{}aw University of Science and Technology,
Faculty of Pure and Applied Mathematics, \newline
      Wyb{.} Wyspia\'nskiego 27,
      50--370 Wroc\l{}aw, Poland }
 \email{balango@iu.edu}

\author[M. Mirek]{Mariusz Mirek }
\address{Mariusz Mirek, \newline
Department of Mathematics,
Rutgers University, \newline
Piscataway, NJ 08854-8019, USA \newline
\indent and \newline
University of Wroc\l aw, 
Mathematical Institute, \newline
Plac Grunwaldzki 2/4,
 50--384 Wroc{\l}aw, Poland
}
\email{mariusz.mirek@rutgers.edu}

\author[T.Z. Szarek]{Tomasz Z. Szarek}
\address{Tomasz Z. Szarek,     \newline
BCAM - Basque Center for Applied Mathematics, \newline
48009 Bilbao,
Spain \newline
\indent and \newline
University of Wroc\l aw, 
Mathematical Institute, \newline
Plac Grunwaldzki 2/4,
 50--384 Wroc{\l}aw, Poland
}
\email{szarektomaszz@gmail.com}

\maketitle

\footnotetext{

\indent The first listed author was supported by the National Science Foundation
under grant no. HDR TRIPODS - 1934962. The second and the fourth listed authors were supported by the
National Science Centre of Poland within the research project OPUS
2017/27/B/ST1/01623.  Mariusz Mirek was partially supported by
Department of Mathematics at Rutgers University, and by the National
Science Centre in Poland, grant Opus 2018/31/B/ST1/00204.  The fourth
listed author was supported also by the Foundation for Polish Science via the
START Scholarship.  }

\section{Introduction}

Let $\lambda\in\Z_+$ be a positive integer and define the set
$\mathbf S_{2}^3(\lambda)$ of all lattice points on a two-dimensional
sphere of radius $\lambda^{1/2}$ by
\[
\mathbf S_{2}^3(\lambda)
:=
\{x \in \mathbb Z^3 : x_1^2 +x_2^2 +x_3^2 = \lambda \}.
\]
The study of the behavior of $\mathbf S_{2}^3(\lambda)$ as
$\lambda\to\infty$ is a central problem in number theory, which has
gone through a period of considerable change and development in the
past three decades. One of many interesting features of this problem is that 
the sets $\mathbf S_{2}^3(\lambda)$ might
be empty for some choices of $\lambda\in\Z_+$. A celebrated result of
Legendre, whose complete proof was given by Gauss \cite{Gauss}, states
that $\mathbf S_{2}^3(\lambda)\neq\emptyset$ if and only if
$\lambda\neq 4^m(8n+7)$ for $m, n\in\N$ (see also
 to \cite{gros}). Hence, for $\lambda\neq 4^m(8n+7)$ it makes
sense to study the counting function
\[
r_2(\lambda):=\# \mathbf S_{2}^3(\lambda).
\]
The behavior of $r_2(\lambda)$ as $\lambda\to\infty$ is very
delicate. On the one hand, $r_2(4^m)=6$ for $m\in\Z_+$. It is well known, (see e.g. the paper of Bateman \cite{Bat}),  that
\begin{align}
\label{eq:40}
r_2(\lambda)=\frac{\pi^{3/2}}{\Gamma(3/2)}\lambda^{1/2}\mathfrak S_3(\lambda)=
\frac{2^3\Gamma(1+1/2)^{3}}{\Gamma(3/2)}\lambda^{3/2-1}\mathfrak S_3(\lambda),
\end{align}
where $\Gamma$ is the standard Gamma function and the factor
$\mathfrak S_3(\lambda)$ is called the singular series, see \cite{Bat}
for more details. Formula \eqref{eq:40} can be also obtained using the
spectral theory of automorphic forms \cite{Duke, Iw}. We also refer to
\cite[Theorem 20.15, p. 478]{IK} for a more extensive treatment of
the formula for $r_2(\lambda)$ and its relations with the Kloosterman
circle method. The formula \eqref{eq:40} is a genuine asymptotic only
if the singular series $\mathfrak S_3(\lambda)$ does not
vanish. However, this is a very subtle question.  The upper bound $r_2(\lambda)\lesssim \lambda^{\frac{1}{2}+o(1)}$ can be obtained by
analyzing the singular series $\mathfrak S_3(\lambda)$ as
$\lambda\to\infty$. Moreover, if $\lambda\neq 0, 4, 7\mod 8$, then
there is also a lower bound
$r_2(\lambda)\gtrsim \lambda^{\frac{1}{2}-o(1)}$, which also follows
from \eqref{eq:40} and Siegel's bounds \cite{Sieg}, which ensure that
$|\mathfrak S_3(\lambda)|\gtrsim\lambda^{-o(1)}$; see also \cite[Remark
below formula (20.130), p. 479]{IK}, where more details are given.

Having many lattice points in $\mathbf S_{2}^3(\lambda)$ as
$\lambda\to\infty$ and $\lambda\neq 0, 4, 7\mod 8$, it is natural
to understand distribution of their projections 
\[
\mathbf P_2^3(\lambda):=\{\lambda^{-1/2}x: x\in\mathbf S_2^3(\lambda)\}\subset \mathbb S^2
\]
on the unit sphere $\mathbb S^2\subset\R^3$.  This line of
investigations were initiated by Linnik \cite{Lin}, who proved under
the Generalized Riemann Hypothesis, that the projected lattice points
$\mathbf P_2^3(\lambda)$ become equidistributed on the unit sphere
$\mathbb S^2$ as $\lambda\to\infty$ and
$\lambda\neq 0, 4, 7\mod 8$. More precisely, if
$\Omega\subseteq \mathbb S^2$ is a ``nice'' set then
\[
\frac{\# (\mathbf P_2^3(\lambda)\cap \Omega)}{r_2(\lambda)}\sim\nu_2(\Omega)
\]
as $\lambda\to\infty$ and $\lambda\neq 0, 4, 7\mod 8$, where $\nu_2$
is a normalized area measure on $\mathbb S^2$. Linnik's result was
proved unconditionally (without (GRH)) by Duke \cite{Duke} and
Golubeva and Fomenko \cite{GF} following a breakthrough paper by
Iwaniec \cite{Iw}. Linnik's ergodic method and related topics were
recently carefully discussed by Ellenberg, Michel and Venkantesh in
\cite{EMV}. We also refer to the recent paper of Bourgain, Rudnick and
Sarnak \cite{BRS}, where the spatial distribution of point sets on the
sphere $\mathbb S^2$ obtained from the representation of a large
integer as a sum of three squares were
investigated. The authors gave a strong evidence to the thesis that
the solutions behave randomly, which
stands in sharp contrast to what happens with sums of two or four or
more squares.

In this paper we  consider perturbations of 
the discrete spheres $\mathbf S_2^3(\lambda)$ and we will be mainly concerned
with three-dimensional variants  of the following sets
\begin{align}
\label{eq:39}
\mathbf S_{c}^d(\lambda)
:=
\{x \in \mathbb Z^d : \lfloor |x_1|^c \rfloor +\ldots +\lfloor |x_d|^c \rfloor= \lambda \}
\end{align}
for any $\lambda\in\Z_+$, where $d\in\Z_+$ and $c>1$. The sets
$\mathbf S_{c}^d(\lambda)$ will be called the arithmetic spheres or
arithmetic $c$-spheres.  We see that $\mathbf S_{c}^d(\lambda)$ with
$c=2$ coincides with the discrete $d$-dimensional Euclidean spheres
$\mathbf S_{2}^d(\lambda)$. We will be mainly  focused on the case $d=3$, and
our main aim is to show that the sets $\mathbf S_{c}^3(\lambda)$ can
be used as much simpler models than $\mathbf S_{2}^3(\lambda)$ or even
toy models (especially for $c>1$, which is close to $1$) to study
fundamental problems in number theory and ergodic theory as discussed
above. The arithmetic spheres also provide a good source of examples
exhibiting some seemingly counterintuitive phenomena. These aspects will be
discussed in detail later in the paper.

\vskip.125in 

We now briefly highlight the main results of this article.

\vskip.125in 

\begin{enumerate}[label*={\arabic*}.] 
\item {\bf Lattice point problems in $\Z^3$.} Using a variant of the
circle method we establish a precise asymptotic formula (in the spirit
of the classical Waring problem) for the number of lattice points   in $\mathbf S_{c}^3(\lambda)$ for $c\in(1, 9/8)$, i.e. 
\[
r_c(\lambda)=\#\mathbf S_{c}^3\sim \frac{2^3\Gamma(1+1/c)^3 }{\Gamma(3/c)}
\lambda^{3/c - 1}
\quad \text{ as } \quad \lambda\to\infty,
\]
 see
Corollary \ref{cor:asym}, and compare with the formula for
$r_2(\lambda)$ from \eqref{eq:40}.  This contrasts sharply with the
situation for $\mathbf S_{2}^3(\lambda)$, where the asymptotic formula
for the number of lattice points is still unknown.  The argument that
we present will also cover a more general situation; generalized
arithmetic spheres \eqref{eq:34} which are induced by certain
regularly varying functions, see Definition \ref{defn} and Theorem
\ref{thm:asymr}, which is the main result of this subsection.

\vskip.125in 

\item {\bf Ergodic theorems and corresponding maximal estimates in
$\Z^3$ and $\R^3$.} The asymptotic formula for the number of lattice
points in $\mathbf S_{c}^3(\lambda)$ allows us to study the main
results of this paper; norm and pointwise convergence for ergodic
averages \eqref{id:1} over the arithmetic spheres
$\mathbf S_{c}^3(\lambda)$. Let $(X,\nu, T)$ be a measure-preserving system  equipped with a family
$T = (T_1, T_2, T_3)$ of commuting invertible and measure-preserving
transformations $T_1, T_2, T_3:X\to X$. In Theorem \ref{thm:erg} we show that for every $p>\frac{11-4c}{11-7c}$ and every $f\in L^p(X)$  the ergodic averages
\[
\frac{1}{\#\mathbf S_{c}^3(\lambda)}\sum_{n\in \mathbf S_{c}^3(\lambda)}f(T_1^{n_1}T_2^{n_2}T_3^{n_3}x) 
\]
converge in $L^p(X)$ norm and almost-everywhere on $X$ as
$\lambda\to\infty$.  Similar problems where considered for the
discrete spheres induced by the Euclidean norm but only in dimensions
$d\ge5$. Here we show that $\mathbf S_{c}^3(\lambda)$ may be used to
illustrate these kind of phenomena in dimension $d=3$. We will obtain
these results by establishing maximal ergodic theorems for the ergodic
averages $A_{\lambda}^c$ from \eqref{id:1}, see Theorem \ref{thm:erg}.
Pointwise convergence will be established by studying $r$-variational
estimates. We will also prove sharp lacunary maximal estimates  for averaging operators over
$\mathbf S_{c}^3(\lambda)$, which are in marked contrast to the
behavior of averaging operators over the discrete Euclidean spheres in
$\Z^d$, see Theorem \ref{thm:L2}. Corresponding maximal and
variational estimates for continuous averaging operators \eqref{def:A}
over the spheres $\S_c^2\subset \R^3$ (see Section \ref{sec:not} for a
definition of $\S_c^2$) will be also discussed, see Theorem
\ref{thm:contsph}.

\vskip.125in 

\item {\bf Equidistribution  problems.} Finally we will discuss a variant of equidistribution problem for the
arithmetic spheres $\mathbf S_{c}^3(\lambda)$. More precisely,
we will study the following projections 
\begin{align}
\label{eq:45}
\mathbf P_c^3(\lambda):=\{\lambda^{-1/c}x: x\in\mathbf S_c^3(\lambda)\}
\end{align}
on a neighborhood of the unit sphere $\mathbb S^2_c\subset\R^3$, where
$$\S_c^2:=\{x\in\R^3:|x|_{c}=1\}$$ (see
Section \ref{sec:not} for a definition of the norm $|\cdot|_{c}$).
Even
though $\mathbf P_c^3(\lambda)\not\subseteq \mathbb S^2_c$ we will
show that the points from $\mathbf P_c^3(\lambda)$ can be interpreted
as equidistributed on the unit sphere $\mathbb S^2_c\subset\R^3$ as
$\lambda\to\infty$. Namely for some nice functions $\phi$ one has
\begin{align*}
\frac{1}{r_c(\lambda)}
\sum_{x \in \mathbf P_c^3(\lambda)} \phi (x)
\xrightarrow[\lambda \to\infty]{} 
\int_{\S_c^2} \phi (x) \, d\nu_{c} (x),
\end{align*}
where $\nu_{c}$ is a probability measure on $\S_c^2$ obtained by
normalization of the measue $\mu_c$, see Theorem \ref{thm:equi} and \eqref{polar}.
We will achieve this by upgrading the circle
method that led us to the asymptotic formula for the number of lattice
points in $\mathbf S_{c}^3(\lambda)$. We will also study the
discrepancy function corresponding to the sets $\mathbf P_c^3(\lambda)$.
\end{enumerate}

In the next three subsections, which will correspond to the
topics described above, we will formulate the main results of this
paper and present a broader context of undertaken problems and give  a brief overview of the proofs. 

\subsection{Lattice point problems in $\Z^3$}
As in \cite{M1, M2, Piat} we begin with introducing a class of
regularly varying or $c$-regularly varying functions $\mathcal{F}_c$,
which will be of our interest.
\begin{defn}\label{defn}
Let $c\in[1, 2)$ and $\mathcal{F}_c$ be the family of all functions
$h:[x_0, \8)\to [1, \8)$ for some $x_0\ge1$, satisfying the
following properties
\begin{enumerate}
\item[(i)] $h\in C^3([x_0, \8))$ and
$$h'(x)>0,\ \ \ \ h''(x)>0, \ \ \mbox{for every \  $x\ge x_0$.}$$
\item[(ii)] There exists a real valued function $\vartheta\in C^2([x_0, \8))$ and a constant $C_h>0$ such that
\begin{align}\label{eq1}
  h(x)=C_hx^c\ell_h(x), \ \ \mbox{where}\ \ \ell_h(x)=\exp\bigg(\int_{x_0}^x\frac{\vt(t)}{t}dt\bigg), \ \ \mbox{for every \  $x\ge x_0$,}
\end{align}
and if $c>1$, then
\begin{align}\label{eq2}
  \lim_{x\to\8}\vartheta(x)=0,\ \ \lim_{x\to\8}x\vartheta'(x)=0,\ \ \lim_{x\to\8}x^2\vartheta''(x)=0.
\end{align}
  \item[(iii)] If $c=1$, then $\vt(x)$ is positive, decreasing and  for every $\varepsilon>0$ we have
  \begin{align}\label{eq3}
    \frac{1}{\vt(x)}\lesssim_{\varepsilon}x^{\varepsilon}, \ \ \mbox{and} \ \ \lim_{x\to\8}\frac{x}{h(x)}=0.
  \end{align}
  Furthermore,
  \begin{align}\label{eq4}
  \lim_{x\to\8}\vartheta(x)=0,\ \ \lim_{x\to\8}\frac{x\vartheta'(x)}{\vt(x)}=0,\ \ \lim_{x\to\8}\frac{x^2\vartheta''(x)}{\vt(x)}=0.
\end{align}
\end{enumerate}
\end{defn}

Definition \ref{defn} will be essential to formulate the main result
of this subsection for the arithmetic spheres as well as generalized
arithmetic spheres.

For a fixed $c\in[1, 2)$ and a fixed function $h\in\mathcal F_c$ 
we define the sets
\begin{align}
\label{eq:35}
\mathbb N_h := \{\lfloor h(m) \rfloor : m \ge N_0\},
\end{align}
where  $N_0\in\Z_+$ is a sufficiently large absolute constant depending only on the function $h$.

For each $k\in[3]$ we fix $c_k \in [1,2)$ and a function
$h_k \in \mathcal{F}_{c_k}$. Since $h'_k(x) \ge 1$ for each
$k\in[3]$, see \eqref{condu1}, we may assume without loss of generality that there exists
an absolute constant $N_0\in\Z_+$ such that for $x\ge N_{0}$ the
functions $x \mapsto \lfloor h_k(x) \rfloor$ are well defined and
injective. From now on we assume that for each $k\in[3]$ the set
$\N_{h_k}$ is defined with this choice of $N_0$.
We now define
generalized arithmetic spheres with radius $\lambda\in\mathbb Z_+$ by
\begin{align}
\label{eq:34}
\mathbf S_{h_1, h_2, h_3}(\lambda):
=\{x\in(\Z_+\setminus[N_{0} - 1])^3: \lfloor h_1(x_1)\rfloor+ \lfloor h_2(x_2)\rfloor+\lfloor h_3(x_3)\rfloor=\lambda\}.
\end{align}
Let us define the corresponding counting function by
\begin{align}
\label{eq:36}
r_{h_1, h_2, h_3}(\lambda):=\#\mathbf S_{h_1, h_2, h_3}(\lambda).
\end{align}
Then it is not difficult to see that 
\begin{align*}
r_{h_1, h_2, h_3}(\lambda)=
\#\{ (n_1, n_2, n_3) \in \N_{h_1} \times \N_{h_2} \times \N_{h_3} : n_1 + n_2 + n_3 = \la\}.
\end{align*}
We now formulate the first result of this paper establishing 
the asymptotic formula for \eqref{eq:36}.

\begin{thm}\label{thm:asymr}
For $k\in[3]$ let $c_k \in [1,4/3)$ and $\g_k := 1/c_{k}$ be such that
\begin{align} \nonumber
4(1 - \g_1) + 5(1-\g_2)/2 + 5(1-\g_3)/2 &< 1, \\ \label{assum}
5(1 - \g_1)/2 + 4(1-\g_2) + 5(1-\g_3)/2 &< 1, \\ \nonumber
5(1 - \g_1)/2 + 5(1-\g_2)/2 + 4(1-\g_3) &< 1.
\end{align} 
Assume that $h_k \in \mathcal{F}_{c_k}$ for each $k\in[3]$, and let $\vp_k$ be its inverse. Then
\begin{align*}
r_{h_1, h_2, h_3}(\lambda) = 
\frac{\Gamma(\g_1) \Gamma(\g_2) \Gamma(\g_3) }{\Gamma(\g_1 + \g_2 +\g_3)}
\la^2 \vp_1'(\la) \vp_2'(\la) \vp_3'(\la)
+ o \left(  \la^2 \vp_1'(\la) \vp_2'(\la) \vp_3'(\la) \right),
\end{align*}
where $\Gamma$ denotes the standard Gamma function. 
\end{thm}

In other words Theorem \ref{thm:asymr} says that every sufficiently
large integer $\lambda\in\Z_+$ can be represented as a sum of three
integers $n_1\in\N_{h_1}$ $n_2\in\N_{h_2}$ and $n_3\in\N_{h_3}$ for
some fixed regularly varying functions $h_1, h_2$ and $h_3$. Theorem
\ref{thm:asymr} can be thought of as a low-dimensional counterpart of
the classical Waring problem asserting that every sufficiently large
positive integer is the sum of a bounded number of $k$-th powers of
positive integers for a fixed integer $k\ge2$. We refer to \cite{IK}
and \cite{Nat} for more details about the classical Waring
problem. Our interest of the behavior of the counting function
$r_{h_1, h_2, h_3}(\lambda)$ was rekindled by the fact that
$\mathbf S_{2}^3(\lambda)=\emptyset$ if and only if
$\lambda=4^m(8n+7)$ for any $m, n\in\N$. Then we started
to investigate whether it is possible to perturb the squares in the
spheres $\mathbf S_{2}^3(\lambda)$ by regularly varying functions
$h_1\in \mathcal F_{c_1}, h_2\in \mathcal F_{c_2}$ and
$h_3 \in \mathcal F_{c_3}$, which are close (in terms of the order of
magnitude of their exponents $c_1, c_2$ and $c_3$) to the identity
${\rm id}(x):=x$ function, to make sure that the set
$\mathbf S_{h_1, h_2, h_3}(\lambda)\neq\emptyset$ for all sufficiently
large integers $\lambda\in\Z_+$.

The analogy of Theorem \ref{thm:asymr} to the Waring problem for the $k$-th
powers will be more apparent if we assume that
$h_1(x)=h_2(x)=h_3(x)=x^c$, then 
the asymptotic formula for the number of lattice points
\begin{align}
\label{eq:38}
r_c(\lambda):=\# \mathbf S_{c}^3(\lambda)
\end{align}
in the arithmetic spheres $\mathbf S_{c}^3(\lambda)$, see \eqref{eq:39},
is very much in the spirit of the Hardy and Littlewood asymptotic
formula from the solution of the Waring problem for the $k$-th powers.
We refer to
\cite[Theorem 20.2, p. 456]{IK} and \cite[Theorem 5.7, p. 146]{Nat}
for detailed expositions of the Hardy and Littlewood theorem as well
as to the Kloosterman circle method \cite[Theorem 20.9, p. 472]{IK},
which handles positive definite quadratic forms with integer
coefficients in $d\ge 4$ variables.
We now see that the arithmetic spheres $\mathbf S_{c}^3(\lambda)$ can be thought of
as a perturbation of the classical spheres $\mathbf S_{2}^3(\lambda)$
by a noninteger exponent $c>1$, which is close to $1$, in place of
the squares. 
  Theorem \ref{thm:asymr}  can be easily applied to
deduce the asymptotic formula for $r_c(\lambda)$.
\begin{cor}\label{cor:asym}
If $c \in (1,9/8)$, then for every $\eps > 0$ we have
\begin{align*}  
r_c(\lambda) = \frac{2^3\Gamma(1+1/c)^3 }{\Gamma(3/c)}
\lambda^{3/c - 1}
+ O_{\eps} \big(  \lambda^{3/c - 1 - (9 - 8c)/(5c) + \eps} \big).
\end{align*} 
\end{cor}

We note that the formula for $r_c(\lambda)$ gives us a genuine
information about the asymptotic behavior of the number of lattice
points in $\mathbf S_{c}^3(\lambda)$.  Comparing the formula for
$r_c(\lambda)$ with the formula $r_2(\lambda)$ from \eqref{eq:40} we
see that the leading terms are of the same kinds. The only difference
is the lack of the singular series in the formula for $r_c(\lambda)$.
This contrasts sharply with the Euclidean situation for
$\mathbf S_{2}^3(\lambda)$, where the asymptotic formula for
$r_2(\lambda)$ greatly depends on the behavior of the underlying
singular sequence, which is very irregular, see \cite[Theorem 20.15,
p. 478]{IK} as well as \cite[Remark below formula (20.130),
p. 479]{IK}. The lack of the singular series in our situation can be
easily explained. We work with non-polynomial functions, which is much
simpler than the polynomial situation, and the only major arc that is
expected to be significant in analysis of $r_c(\lambda)$ is the one
centered at the origin. Therefore the singular series does not arise
in the case of arithmetic $c$-spheres or can be thought of as a sum
containing the only one term. This is the reason why the arithmetic
spheres \eqref{eq:39} may be used as toy models to examine various
phenomena in number theory and ergodic theory as we shall see later in
the paper.

Finally, let us mention  that Deshouillers \cite{Des} and later Arkhipov and Zhitkov
\cite{AZ} studied variants of Waring's problem with noninteger
exponents. Specifically, it is known from \cite{AZ} that every sufficiently large
integer $\lambda\in\Z_+$ can be written in the form
\begin{align}
\label{eq:32}
\lfloor x_1^c \rfloor + \ldots + \lfloor x_d^c \rfloor= \lambda 
\end{align}
for some positive integers $x_1,\ldots, x_d\in \Z_+$, whenever $c>12$
and $d> 22c^2(\log c+4)$. Moreover, the number of all $d$-tuples $(x_1,\ldots, x_d)\in\Z_+^d$ that solve \eqref{eq:32} behaves like
\begin{align}
\label{eq:33}
\frac{\Gamma(1+1/c)^d }{\Gamma(d/c)}
\lambda^{d/c - 1}(1+o(1)) \quad \text{as}\quad \lambda\to\infty.
\end{align}
If we were allowed to take $c\in(1, 9/8)$ and $d=3$ in \eqref{eq:33}
we would see that the asymptotic formula for $r_c(\lambda)$ coincides
with \eqref{eq:33} up to the factor $2^3$ that arises from considering
all integer triples in $\mathbf S_{c}^3(\lambda)$ instead of only
positive triples from $\mathbf S_{c}^3(\lambda)\cap \Z_+^3$. Neither
the method from \cite{Des} nor from \cite{AZ} does seem to work in the
case of small exponents $c\in(1, 9/8)$. Here we propose a different
approach to prove Theorem \ref{thm:asymr}, our strategy will be
briefly described later in the paper.  The proofs of Theorem
\ref{thm:asymr} and Corollary \ref{cor:asym} are given in Section
\ref{sec:asym}.

The asymptotic formula from Corollary \ref{cor:asym} was the starting
point to study the main results of this article concerning norm and
pointwise convergence for the ergodic averages over the arithmetic
spheres.

\subsection{Ergodic theorems and corresponding maximal estimates in
$\Z^3$ and $\R^3$}
Let $(X,\nu, T)$ be a measure-preserving system, where $(X,\nu)$ is a
$\sigma$-finite measure space equipped with a family
$T = (T_1, T_2, T_3)$ of commuting invertible and measure-preserving
transformations $T_1, T_2, T_3:X\to X$. We will also write
$T^n:=T_1^{n_1}T_2^{n_2}T_3^{n_3}$ for any $n\in\Z^3$; and for each
$k\in[3]$ we define $T_k^0:={\rm Id}$ to be the identity map on $X$.

We now
fix $c \in (1, 11/10)$, and let $\lambda(c)\in\Z_+$ be the smallest
integer such that $\mathbf{S}^3_c(\lambda)\neq\emptyset$ for all
$\lambda\ge\lambda(c)$. Such an integer exists in view of the
asymptotic formula for $r_c(\lambda)$ from Corollary \ref{cor:asym}.
Now for every $f\in L^0(X)$ (see Section
\ref{sec:not} for a definition of this space) and every integer 
$\lambda\ge \lambda(c)$ we define the ergodic averages over the arithmetic
spheres $\mathbf S_{c}^3(\lambda)$ by
\begin{align}  \label{id:1}
A_{\lambda}^c f(x): =
A_{\lambda}^{c,3} f(x): =
\frac{1}{\#\mathbf{S}^3_c(\lambda)} \sum_{n \in \mathbf{S}^3_c(\lambda)} f (T^n x), 
\qquad x \in X.
\end{align}
We see that the average \eqref{id:1} is well defined for all
$\lambda\ge \lambda(c)$.  This was also our motivation to study the
asymptotic formula for $r_c(\lambda)=\#\mathbf{S}^3_c(\lambda)$ in Corollary
\ref{cor:asym}. Without this asymptotic it is even not clear whether
$\mathbf{S}^3_c(\lambda)\neq\emptyset$. From now on when we consider
averages \eqref{id:1} we always assume that
$\lambda\ge\lambda(c)$. For the further reference, let
$A_{\lambda}^{c, d}$ denote the ergodic average, which is  an analogue of average \eqref{id:1}, over
the $d$-dimensional $c$-sphere $\mathbf{S}^d_c(\lambda)$ as in \eqref{eq:39}.

\begin{ex}[Integer shift system on $\Z^3$]
The integer shift system $(\Z^3,\nu_{\Z^3},T_{\Z^3})$ is the
three-dimensional integer grid $\Z^3$ equipped with counting measure
$\nu_{\Z^3}$ and the family of shifts
$T_{\Z^3}=(T_{\Z^3, 1}, T_{\Z^3, 2}, T_{\Z^3, 3})$, with
$T_{\Z^3, k}(x) := x-e_k$ for $k\in[3]$, where $e_k$ is the $k$-th
basis vector from the standard basis on $\R^3$. In view of the
Calder\'on transference principle, see \cite{B3}, the integer shift
system is thought of as a ``universal'' system for all other
measure-preserving systems.  It will be convenient as we shall see in
a moment to work with this system due to the extensive
Fourier-analytic structure available on $\Z^3$.  The averages
$A_{\lambda}^c$ on $(\Z^3,\nu_{\Z^3},T_{\Z^3})$ will be denoted by
$M_{\lambda}^c$ and we have for any function $f:\Z^3\to \C$ that
\begin{align}
\label{eq:41}
M_{\lambda}^cf(x):=
M_{\lambda}^{c,3}f(x):=
\frac{1}{\#\mathbf{S}^3_c(\lambda)} \sum_{n \in \mathbf{S}^3_c(\lambda)} f (x-n), 
\qquad x \in \Z^3, \qquad \lambda\ge\lambda(c).
\end{align}
Taking $\sigma_{\lambda} (x) := \ind_{\mathbf{S}^3_c(\lambda)} (x)$
for $x\in\Z^3$ we immediately see that the discrete average
$M_{\lambda}^c$ is a convolution operator and we have
\begin{align*}
M_{\lambda}^cf(x)=\frac{1}{r_c(\lambda)}\sigma_{\lambda}*f(x), \qquad x \in \Z^3, \qquad \lambda\ge\lambda(c).
\end{align*}
For the further reference, let
$M_{\lambda}^{c, d}$ denote the discrete average, which is  an analogue of average \eqref{eq:41}, over
the $d$-dimensional $c$-sphere $\mathbf{S}^d_c(\lambda)$ as in \eqref{eq:39}.
\end{ex}

We say that a set $\mathbb D=\{\lambda_n:n\in\Z_+\}\subset (0, \infty)$ is $\lambda_0$-lacunary if 
\begin{align*}
\lambda_0:=\inf_{n\in\Z_+}\frac{\lambda_{n+1}}{\lambda_{n}}>1.
\end{align*}

We shall be concerned with norm and pointwise convergence for the
averages $A_{\lambda}^c$ defined for $\lambda\in\Z_+$ as well as for
$\lambda \in \mathbb D$, where
$\mathbb D=\{\lambda_n:n\in\Z_+\}\subset \Z_+$ is a lacunary set. The
lacunary cases will exhibit different phenomena than the averages
$A_{\lambda}^c$ defined for $\lambda\in\Z_+$.

The main result of the paper is the following ergodic theorem for averages \eqref{id:1}.

\begin{thm} \label{thm:erg} Let $(X,\nu, T)$ be a measure-preserving
system.  Fix $c \in (1, 11/10)$ and let $A_{\lambda}^c$ be the average
from \eqref{id:1} defined for all integers $\lambda\ge \lambda(c)$.
Then for every $f \in L^p(X)$, where $(11-4c)/(11-7c) < p <\8$ we have
that:
\begin{itemize}
\item[(i)] (Mean ergodic theorem). The averages $A_{\lambda}^c f$ converge in $L^p(X)$ norm as $\lambda\to\infty$.
\item[(ii)] (Pointwise ergodic theorem). The averages $A_{\lambda}^c f$ converge pointwise almost everywhere as $\lambda\to\infty$.
\item[(iii)] (Maximal ergodic theorem). The maximal inequality holds
\begin{align}
\label{eq:15}
\big\|\sup_{\lambda\in\Z_+}|A_{\lambda}^c f|\big\|_{L^p(X)}\lesssim_{c, p} \|f\|_{L^p(X)}.
\end{align}
\end{itemize}
Inequality \eqref{eq:15} also holds for $p=\infty$.  Moreover, if
$\lambda\ge \lambda(c)$ in the average $A_{\lambda}^c$ is restricted
to a $\lambda_0$-lacunary set $\mathbb D\subset \Z_+$, then (i), (ii) and
(iii) remain valid for any $f\in L^p(X)$ and any
$p\in(1, \infty)$. However, the implicit constant in \eqref{eq:15}
additionally depends on $\lambda_0>1$.
\end{thm}

We immediately see that (i) is a simple consequence of (ii), (iii) and
the dominated convergence theorem. Thus it suffices to establish
pointwise ergodic theorem (ii) and maximal ergodic theorem (iii). A
detailed proof of Theorem \ref{thm:erg} will be given in Section
\ref{sec:erg}. We first prove maximal inequality \eqref{eq:15} and
then use it to deduce pointwise convergence for $A^c_{\lambda}$.  The
maximal ergodic theorem (iii), in fact, will follow from Theorem \ref{thm:L2} below by
appealing to the Calder{\'o}n transference principle as in
\cite{B3}. We prove Theorem \ref{thm:L2} in Section \ref{sec:el2}.

The main maximal result of this article is the boundedness for the discrete averages $M_{\lambda}^c$. 
\begin{thm} \label{thm:L2}
Fix $c \in (1, 11/10)$ and let $M_{\lambda}^c$ be the discrete  average
from \eqref{eq:41} defined for all integers $\lambda\ge \lambda(c)$. Then the following two maximal inequalities hold.
\begin{itemize}
\item[(i)] (Maximal inequality for full averages). If $(11-4c)/(11-7c) < p \le\8$ then one has
\begin{align}
\label{eq:43}
\big\|\sup_{\lambda\in\Z_+}|M_{\lambda}^c f|\big\|_{\ell^p(\Z^3)}\lesssim_{c, p} \|f\|_{\ell^p(\Z^3)},
\qquad f \in \ell^p(\Z^3).
\end{align} 

\item[(ii)] (Maximal inequality for lacunary averages). If $1 < p \le\8$ then one has
\begin{align}
\label{eq:44}
\big\|\sup_{\lambda\in\mathbb D}|M_{\lambda}^c f|\big\|_{\ell^p(\Z^3)}\lesssim_{c, p, \lambda_0} \|f\|_{\ell^p(\Z^3)},
\qquad f \in \ell^p(\Z^3);
\end{align} 
whenever $\mathbb D\subset \Z_+$ is $\lambda_0$-lacunary set.
\end{itemize}
\end{thm}

As we mentioned above, inequality \eqref{eq:43} implies \eqref{eq:15}
upon appealing to the Calder{\'o}n transference principle as described
in \cite{B3}. By the same argument we conclude that \eqref{eq:44}
implies a corresponding lacunary variant of inequality
\eqref{eq:15}. Therefore, the conclusion of Theorem
\ref{thm:erg} (iii) is entirely reduced to the conclusion of Theorem
\ref{thm:L2}, which will be proved using Fourier methods in Section
\ref{sec:el2}. An essential part of our approach is the fact that the
averages $M_{\lambda}^c$ from \eqref{eq:41} may be thought of as
discrete analogues of spherical averaging operators in $\R^3$, which
are dilates of the sphere $\S_c^2$.
Let $\mu_{c}$ be a measure on $\S_c^2$ arising in a natural way from
the polar decomposition with respect to the norm $|\cdot|_c$, see
\eqref{polar} for a precise description of $\mu_{c}$.
Then continuous spherical averaging operators are defined by
\begin{align} \label{def:A}  
\mathcal A_t^c f (x) := \int_{\S_c^2} f(x - t \theta) \, d\mu_c(\theta),
\qquad x \in \R^3, \quad t>0, \quad f \in C_c^{\8} (\R^3),
\end{align} 
As before, for the further reference, let
$\mathcal A_t^{c, d}$ denote the average, which is  an analogue of average \eqref{def:A}, over
the sphere $\S^{d-1}_c \subset \R^d$ (see
Section \ref{sec:not} for a definition of this set). 
Our next result subsumes maximal and $r$-variational estimates for the operators $\mathcal A_t^c$, which will be proved in Section \ref{sec:osc}, (see Section \ref{sec:not} for a definition of $r$-variational seminorm $V^r$ and its simple properties).
\begin{thm} \label{thm:contsph} Let $c \in (1, 2)$ be fixed, and let
$\mathcal A_t^c$ for $t>0$ be the spherical averaging operator defined in \eqref{def:A}. Then
for every $p\in (3/2, 4)$ and $r\in(2, \infty)$ we have
\begin{align}
\label{eq:24}
\| V^r(\mathcal A_t^cf: t>0)\|_{L^p(\R^3)}\lesssim_{c, p, r}\| f \|_{L^p(\R^3)}, \qquad f \in L^p(\R^3),
\end{align}
as well as, for all $p\in (3/2, \infty]$, the maximal estimate 
\begin{align}
\label{est17}
\| \sup_{t>0}|\mathcal A_t^cf|\|_{L^p(\R^3)}
\lesssim_{c, p}
\| f \|_{L^p(\R^3)}, \qquad f \in L^p(\R^3).
\end{align}
If we consider only lacunary times the range of $p$ in \eqref{eq:24} and \eqref{est17} can be
extended. Namely, for every $p\in (1, \infty)$ and
$r\in(2, \infty)$ we have
\begin{align}
\label{eq:25}
\| V^r(\mathcal A_{t}^cf: t\in\mathbb D)\|_{L^p(\R^3)}\lesssim_{c, p, r, \lambda_0}\| f \|_{L^p(\R^3)}, \qquad f \in L^p(\R^3),
\end{align}
and for all $p\in (1, \infty]$, the maximal estimate holds
\begin{align}
\label{eq:26}
\| \sup_{t\in\mathbb D}|\mathcal A_{t}^cf|\|_{L^p(\R^3)}
\lesssim_{c, p, \lambda_0}
\| f \|_{L^p(\R^3)}, \qquad f \in L^p(\R^3),
\end{align}
whenever $\mathbb D\subset(0, \infty)$ is $\lambda_0$-lacunary set.
Moreover, the range of $p$ in \eqref{eq:24} and in \eqref{est17} is sharp from below in the sense that
for every $p\in[1, 3/2]$ the estimates  \eqref{eq:24} and  \eqref{est17} do not hold.
\end{thm}

Averaging operators over the Euclidean spheres ($\S_c^{d-1}$ with
$c=2$) were extensively studied over the years.  It is very well known
from the results of Stein \cite{[42]} for all $d\ge 3$, and Bourgain
\cite{[7]} for $d=2$ that the maximal spherical function
$\sup_{t>0}|\mathcal A_t^{2, d}f|$ is bounded on $L^p(\R^d)$ if and
only if $\frac{d}{d-1}<p\le \infty$. If $t$ is
restricted to a lacunary set $\mathbb D\subset(0, \infty)$ then the
lacunary spherical maximal function
$\sup_{t\in\mathbb D}|\mathcal A_t^{2, d}f|$ is bounded on
$L^p(\R^d)$ for all $1<p\le \infty$ as it was shown by Calder{\'o}n
\cite{Cal} and independently by Coifman and Weiss \cite{CW}. 

Consider the $c$-sphere $\S_c^{d-1}$ as defined in \eqref{id:307}. If $c$ is an even integer, $>2$, it follows from (\cite{IS}, 
see also \cite{ISS}) that the corresponding maximal operator is bounded for 
$p>\max \left\{\frac{d}{d-1}, \frac{c}{d-1} \right\}$ (sharp). Similar methods yield the same result for any $c \ge 2$. The 
case $1<c<2$ is a bit different due to the lack of smoothness and the fact just as in the case $c=2$, the 
Gaussian curvature does not vanish. We are able to handle this case in Theorem  \ref{thm:contsph} above using a 
suitable Fourier decay estimate (see inequality
\eqref{eq:fourier}) that allows to follow the general outline of Stein's original argument for the sphere in
dimensions three and higher. While Theorem \ref{thm:contsph} is only stated and proved in three dimensions, the maximal 
estimate easily extends to all $d \ge 3$ with $p>\frac{3}{2}$ replaced by $p>\frac{d}{d-1}$. 

Variational estimates for the spherical averages $\mathcal A_t^{2, d}$
were studied by Jones, Seeger and Wright \cite{JSW} who showed that
$V^r(\mathcal A_t^{2, d}f: t>0)$ is bounded on $L^p(\R^d)$ for all
$r>2$ if $\frac{d}{d-1}<p\le 2d$, where the conditions for $r$ and $p$
are both sharp.  If $p>2d$ then $V^r(\mathcal A_t^{2, d}f: t>0)$ is
bounded on $L^p(\R^d)$ as long as $r>p/d$ and this is not true when
$r<p/d$. Recently, also Beltran, Oberlin, Roncal, Seeger and Stovall
\cite{BORSS} showed the endpoint estimate, which states that if $d\ge3$ and $p>2d$ then
$V^{p/d}(\mathcal A_t^{2, d}f: t>0)$ is of restricted weak type $(p,p)$.

Theorem \ref{thm:contsph} gives sharp estimates in \eqref{est17},
\eqref{eq:25} and \eqref{eq:26}. However, the range of $p$ and $r$ in
\eqref{eq:24} is not sharp. The exponent $c\in(1, 2)$
does not affect the maximal estimates \eqref{est17} and \eqref{eq:26},
which coincide with their Euclidean counterparts (with $c=2$ and
$d=3$). This phenomenon (as mentioned above) can be easily explained by examining Fourier
transform estimates of the spherical measure $\mu_c$ on $\S_c^{2}$,
which for all $c\in(1, 2)$ does have precisely the same bounds as the
Fourier transform of the spherical measure on $\S_2^{2}$, see
\eqref{eq:fourier}. We have not found this result in the existing
literature, and since it may be of independent interest we included
the proof in Section \ref{sec:osc}. Moreover, the Fourier
transform estimates \eqref{eq:fourier} of the measure $\mu_c$ will be
essential in the proofs of Theorem \ref{thm:erg} and Theorem
\ref{thm:L2}. 

Discrete analogues $M_{\lambda}^{2, d}$ of the Euclidean spherical
averages $\mathcal A_t^{2, d}$ were investigated by Magyar \cite{Ma},
and Magyar, Stein and Wainger \cite{MSW}. In the latter work a
complete result was proved, which asserts that the discrete maximal
spherical function $\sup_{\lambda\in\Z_+}|M_{\lambda}^{2, d}f|$ is
bounded on $\ell^p(\Z^d)$ if and only if $d\ge 5$ and
$\frac{d}{d-2}<p\le \infty$. It is a remarkable result that
illustrates a difference between discrete and continuous operators.
The restricted weak-type endpoint result was also proved for
$\sup_{\lambda\in\Z_+}|M_{\lambda}^{2, d}f|$ by Ionescu \cite{I1}.

Magyar, Stein and Wainger paper \cite{MSW} was the starting point of
three different interesting lines of investigations in the discrete
harmonic analysis and related areas.
\medskip
\paragraph{\bf 1.}
The first line, initiated by Magyar \cite{Ma2}, extended the result of
Magyar--Stein--Wainger \cite{MSW} to discrete averages defined over a
certain wide class of positive definite hypersurfaces in the spirit of
Birch \cite{Bir} and Davenport \cite{Dav}. Magyar \cite{Ma2} also
initiated investigations of norm and pointwise ergodic theorems for
these kind of averages as well as equidistribution problems \cite{Ma3}
and \cite{Mag} in the context of discrepancy function, see also the
next subsection for more details.  Magyar's results were extended by
Hughes \cite{H3, H2} to $k$-spheres $\mathbf{S}^d_k(\lambda)$ for any
integer $k\ge 2$ and sufficiently large dimensions $d\in\Z_+$. Hughes'
results were recently improved by Anderson, Cook, Hughes and Kumchev 
in \cite{ACHK}. In the latter paper the authors also continued their investigations  of
the so-called ergodic Waring--Goldbach problems \cite{ACHK0}, where averages are taken over $k$-spheres in primes
$\mathbf{S}^d_k(\lambda)\cap\mathbb P^d$ for any integer $k\ge 2$ and
sufficiently large dimensions $d\in\Z_+$.  Discrete maximal averages over Birch forms were recently studied by Cook
\cite{C1}.

\medskip

\paragraph{\bf 2.} The second line of investigation, initiated by Hughes \cite{H1}, on the
one hand, asks about $\ell^p(\Z^d)$ bounds for discrete spherical
maximal operators obtained from restricting the supremum to lacunary
sets. On the other hand, Hughes observed (even though the
Magyar--Stein--Wainger theorem is sharp) that it also makes sense to study
$\sup_{\lambda\in2\N+1}|M_{\lambda}^{2, d}f|$ for $d=4$ upon
restricting the radii $\lambda$ to odd integers. Specifically, the
author \cite{H1} constructed a sophisticated lacunary set of radii
$\mathbb D$ such that
$\sup_{\lambda\in \mathbb D}|M_{\lambda}^{2, d}f|$ is bounded on
$\ell^p(\Z^d)$ for every $\frac{d}{d-2}\le p\le \infty$ and
$d\ge 4$. Hughes' result was recently extended by Kesler, Lacey and
Mena \cite{KLM}, where it was proved that
$\sup_{\lambda\in \mathbb D}|M_{\lambda}^{2, d}f|$ for any lacunary
sequence $\mathbb D\subset \Z_+$ is bounded on $\ell^p(\Z^d)$ for
every $\frac{d-2}{d-3}< p\le \infty$ and $d\ge5$. The case $d=4$ was
recently established by Anderson and Madrid \cite{AM} for
$\frac{d+1}{d-1}< p\le \infty$ and all lacunary sequences
$\mathbb D\subset \Z_+\setminus 4\Z_+$.
Cook and
Hughes \cite{CH} considered similar questions in the context of Birch
forms, and specifically recovered the main result from \cite{KLM}.
They also showed in \cite{CH}  that in
contrast to the continuous case, no such inequalities can hold close
to $\ell^1(\Z^d)$. More precisely, for any $1<p< \frac{d}{d-1}$ there
exists a set of lacunary radii $\mathbb D\subset\Z_+$ such that
$\sup_{\lambda\in \mathbb D}|M_{\lambda}^{2, d}f(x)|$ is unbounded on
$\ell^p(\Z^d)$. This negative result is also true  \cite{CH} for averages over
more general forms  in the spirit of Birch. This is another remarkable
phenomenon that exhibit some peculiar features in the discrete world.
Finally, let us emphasize that the negative result from \cite{CH} does not
exclude positive results for $1<p< \frac{d}{d-1}$. Namely, Cook showed
$\ell^p(\Z^d)$ maximal bounds for all $1<p\le \infty$ by considering
$M_{\lambda}^{2, d}$ for $d\ge 5$ with a very sparse sequence of radii \cite{C3} as
well as by studying averages associated to a certain class of homogeneous
algebraic hypersurfaces in \cite{C1}.

\medskip
\paragraph{\bf 3.} The third line, initiated by Kesler, Lacey and Mena
\cite{KLM}, \cite{KLM2} who proposed to study sparse estimates in the
context of discrete spherical averages. It is a very successful line
of research, which significantly enhanced the field of discrete
harmonic analysis. In \cite{KLM2} Kesler, Lacey and Mena proved
conjecturally sharp sparse bounds for
$\sup_{\lambda\in\Z_+}|M_{\lambda}^{2, d}f|$. Using these bounds the
authors recovered in a fairly unified way the main results from
Magyar--Stein--Wainger paper \cite{MSW} as well as the endpoint result
of Ionescu \cite{I1}. Sparse estimates \cite{K1, K2, KL} turned out to
be very efficient in $\ell^p(\Z^d)$-improving estimates for the
Magyar--Stein--Wainger averages $M_{\lambda}^{2, d}$, which were
initially studied by Hughes \cite{H4}. We also refer to Anderson's
paper \cite{And} which obtained $\ell^p(\Z^d)$-improving bounds for
spherical averages along the primes.

\medskip

Theorem \ref{thm:erg} and Theorem \ref{thm:L2} contribute to the first
and the second line of research. Firstly, it is the first paper that
investigates arithmetic spheres $\mathbf{S}^3_c(\lambda)$ from
\eqref{eq:39} with noninteger exponents $c\not\in\Z_+$ in the context
of norm and pointwise ergodic theorems, see Theorem
\ref{thm:erg}. Secondly, Theorem \ref{thm:L2} establishes full (all
radii $\lambda\in\Z_+$) maximal estimates \eqref{eq:43}; as well as
sharp (for all exponents $p\in(1, \infty]$) lacunary maximal estimates
\eqref{eq:44} for the discrete averages $M_{\lambda}^{c}$ over the
arithmetic spheres $\mathbf{S}^3_c(\lambda)$ in dimension $d=3$. To
the best of our knowledge it is the first result in the discrete
setup that handles three-dimensional case.  Thirdly, since
$\mathbf{S}^3_c(\lambda)$ is induced by non-polynomial functions the
circle method that lies behind of the maximal estimates in Theorem
\ref{thm:erg} and Theorem \ref{thm:L2} is much simpler. The only major
arc that is expected to be significant in analysis of Fourier
multipliers corresponding to the averages $M_{\lambda}^{c}$ is the one
centered at the origin and the arithmeticity that was apparent in all
papers mentioned above does not enter into our proofs at all. 
Therefore, the proofs are conceptually closer
to the situation that arises in the continuous setup, this allows us
to think that averages $A_{\lambda}^{c}$ may be thought of as toy
models to study the  questions of this type in ergodic theory and
discrete harmonic analysis.  This gives strong motivation to
understand the situation more thoroughly.

 \subsection{Equidistribution  problems}
We will discuss equidistribution problems corresponding to
the arithmetic spheres $\mathbf{S}^3_c(\lambda)$.  The last goal of
the article is to study the following projections $\mathbf P_c^3(\lambda)$, see \eqref{eq:45}.
In other words this is the set of all projections of lattice points from
$\mathbf{S}^3_c(\lambda)$ via the dilatation
$x \mapsto \lambda^{-1/c} x$ on a neighborhood of the unit sphere
$\mathbb S^2_c\subset\R^3$. Even though
$\mathbf P_c^3(\lambda)\not\subseteq \mathbb S^2_c$ we can show that
the points from $\mathbf P_c^3(\lambda)$ are asymptotically close to
$\mathbb S^2_c$ since
\begin{align*} 
\sup_{x \in \mathbf P_c^3(\lambda)} ||x|_c^c - 1| \le 3/\lambda  
\xrightarrow[\lambda \to\infty]{} 0,
\end{align*}
and in fact 
the points from $\mathbf P_c^3(\lambda)$ can be interpreted as
equidistributed on the unit sphere $\mathbb S^2_c\subset\R^3$ as
$\lambda\to\infty$, in the sense of the following result.
\begin{thm}
\label{thm:equi}
Let $c \in (1,9/8)$ be fixed. Then for every $\phi \in C^{\8}(\R^3)$ we have
\begin{align}
\label{eq:47}
\frac{1}{r_c(\lambda)}
\sum_{x \in \mathbf P_c^3(\lambda)} \phi (x)
\xrightarrow[\lambda \to\infty]{} 
\int_{\S_c^2} \phi (x) \, d\nu_{c} (x),
\end{align}
where $\nu_{c}$ is a probability measure on $\S_c^2$ obtained by normalization of the measue  $\mu_c$, i.e.
\[
\nu_{c} :=\frac{\mu_c}{\mu_c(\S_c^2)}, \quad \text{ where } \quad \mu_c(\S_c^2)=\frac{8 \Gamma(1/c)^3}{c^2\Gamma(3/c) }.
\]
\end{thm}

Theorem \ref{thm:equi} can be thought of as a variant of Linnik's
problem \cite{Lin} for the arithmetic spheres
$\mathbf{S}^3_c(\lambda)$, see also \cite{Duke} and \cite{GF} for
unconditional variants of Linnik's result.

In fact, being motivated by Magyar \cite{Ma3} and \cite{Mag}, we obtain
a much stronger result by investigating the rate of equidistribution
of the sets $\mathbf P_c^3(\lambda)$ on the neighborhood of the sphere
$\S_c^2$ as $\lambda \to \infty$. For this purpose we shall
investigate the discrepancy of the sets $\mathbf P_c^3(\lambda)$ with respect
to the following spherical caps
\begin{align*} 
C_{a,\xi} := \{ x \in \S_c^2 : x \cdot \xi \ge a \}, 
\qquad \xi \in \S^2, \quad a > 0.
\end{align*}
Here and later on $\S^2 = \S^2_2$, see \eqref{id:307}.
Due to the technical reason caused by the fact that $\mathbf P_c^3(\lambda) \nsubseteq \S_c^2$, we also need to consider 
\begin{align*} 
\mathbf C_{a,\xi} := \{ x \in \R^3 : 100 \ge x \cdot \xi \ge a \}, 
\qquad \xi \in \S^2, \quad a > 0.
\end{align*}
Observe that $C_{a,\xi} \subseteq \mathbf C_{a,\xi}$ because $|x\cdot \xi| \le 1$ for $x \in \S_c^2$ and $\xi \in \S^2$. 
Then the discrepancy function of the set $\mathbf P_c^3(\lambda)$ with respect to the caps  $C_{a,\xi}$ and $\mathbf C_{a,\xi}$ associated with a given direction $\xi \in \S^2$ is defined by
\begin{align*} 
D_c(\lambda,\xi)  := \sup_{a>0} |D_c(\lambda,\xi,a)|, 
\qquad \lambda \in \Z_+, \quad \xi \in \S^2,
\end{align*}
where
\begin{align*} 
D_c(\lambda,\xi,a)  := \#(\mathbf P_c^3(\lambda) \cap \mathbf C_{a,\xi}) 
- r_c(\lambda) \nu_{c} (C_{a,\xi}),
\qquad \lambda \in \Z_+, \quad \xi \in \S^2, \quad a > 0.
\end{align*}
Using the ideas from \cite{Mag} we are able to prove the following estimate.

\begin{thm} \label{thm:disc}
Let $c \in (1,9/8)$ be fixed. Then for every $\eps > 0$ we have
\begin{align*} 
D_c(\lambda,\xi) \lesssim_{\eps} \lambda^{3/c - 1 - (9-8c)/(5c) + \eps},
\qquad \lambda \in \Z_+, \quad \xi \in \S^2.
\end{align*}
The implicit constant is uniform in $\lambda$ and $\xi$.
\end{thm}

Theorem \ref{thm:disc} is a three-dimensional analogue of Magyar's
result \cite{Ma3} for the Euclidean spheres (as well as some
hypersurfaces corresponding to homogeneous polynomials with integer
coefficients) for all dimensions $d\ge 4$, which establishes bounds
for the discrepancy function with respect to spherical caps  along
diophantine directions $\xi\in\S^{d-1}$, (see \cite{Ma3} and
\cite{Mag} for a definition of diophantine directions).  Our result
controls the discrepancy function $D_c(\lambda,\xi)$ uniformly in
$\xi \in \S^2$ in contrast to Magyar's results \cite{Ma3, Mag}, which
exclude some sets of directions that is of Lebesgue measure zero in
$\R^{d-1}$. The reason is exactly the same as in the previous two
problems undertaken in this paper. The only major arc that is expected
to be significant in analysis of the discrepancy function
$D_c(\lambda,\xi)$ is the one centered at the origin and we do not
need to exclude non-diophantine directions from the picture in our case. We show a
detailed proof of Theorem \ref{thm:disc} in Section
\ref{sec:disc}. Then we easily deduce, proceeding in a similar way as in
the proof of Theorem~\ref{thm:disc}, the convergence in Theorem \ref{thm:equi}.

%%%%%%%%%%%%%%%%%%%%%%%%%%%

\subsection{Structure of the paper}
The paper is organized as follows. Theorem \ref{thm:asymr} is proved
in Section \ref{sec:asym}. In our proof we combine the ideas of
Heath--Brown \cite{HB} with Vinogradov's ideas from the ternary
Goldbach problem (see \cite[Section 8, p. 211]{Nat} and also
\cite[Chapter 13, p. 336]{IK}) as it was done in the Balog and
Friedlander paper \cite{BF}, where the ternary Goldbach problem is
solved in the Piatetski--Shapiro primes $\N_{h}\cap\mathbb P$ with
$h(x)=x^c$ for some $c>1$.  The important part of the argument relies
on the estimates contained in Proposition \ref{prop:J3} (an asymptotic
formula giving the leading term for $r_{h_1, h_2,h_3}$ in Theorem
\ref{thm:asymr}) and Proposition \ref{prop:FG} (exponential sum
estimates giving the error term in Theorem \ref{thm:asymr}).

In Section~\ref{sec:el2} we establish maximal inequalities
\eqref{eq:43} and \eqref{eq:44} from Theorem \ref{thm:L2}, which is
the first main result of this article.  Maximal inequalities
\eqref{eq:43} and \eqref{eq:44} corresponding to the averages
$M_{\lambda}^c$ and the Magyar--Stein--Wainger maximal inequality
\cite{MSW} corresponding to the averages $M_{\lambda}^{2,d}$ have the same
starting point --- the Hardy--Littlewood circle method --- however
completely different proofs. In our situation, there is only one major
arc (centered at the origin) and respectively the only one minor
arc. This phenomenon is caused by the non-integer nature of functions
$x^c$, $c \in (1,2)$, which is in marked contrast with the classical
polynomial situation arising in \cite{MSW}, where major arcs consist
of a collection of small neighborhoods of rationals with sufficiently
small denominators. From this point of view our model is much simpler,
but we cannot directly follow the ideas from \cite{MSW} or even the
subsequent papers discussed above. Maximal estimates corresponding to
the averages $M_{\lambda}^c$ on the minor arc (see Theorem
\ref{thm:minor}) are reduced upon analyzing various Fourier expansions
to certain exponential sum estimates.  The latter exponential sums,
following the ideas of Heath--Brown \cite{HB}, can be estimated using
the van der Corput second derivative test for exponential sums in
place of the usual Weyl's inequality like in the polynomial situation.
Maximal estimates corresponding to the averages $M_{\lambda}^c$ on the
major arc (see Theorem \ref{thm:major}) can be controlled by maximal
functions associated to continuous convolution operators with the
kernels $K_{\lambda}$ as in \eqref{eq:5} upon applying a comparison
principle from Lemma \ref{norm}. Boundedness of the latter operators
are obtained by standard arguments as in \cite{MSZ1}, which in turn
are reduced to the Fourier transform estimates of the surface measure
$\mu_c$ on the hypersurface $\S_c^2$, see \eqref{eq:fourier}. Since we
work with noninteger $c\in(1, 2)$ it is easy to see that $\S_c^2$ is
locally nonsmooth and the Fourier transform estimates of $\mu_c$ in
\eqref{eq:fourier} are much more delicate than the Euclidean one.  We
have not found these kind of estimates in the existing literature
therefore the details (which are interesting in their own right) are
provided in Section \ref{sec:osc}, see Lemma \ref{lem:FTest}. The
proof of Theorem \ref{thm:contsph} is also provided in Section
\ref{sec:osc} as a consequence of the Fourier transform estimates
\eqref{eq:fourier} and the techniques from \cite{MSZ1}.

Theorem \ref{thm:erg}, the second main result of this paper, is proved
in Section \ref{sec:erg}. In fact, as mentioned above, we only need to
establish pointwise convergence for the averages $A_{\lambda}^c$. This
is achieved by splitting the average $A_{\lambda}^c$ into two parts
corresponding respectively to major and minor arc. The minor arc part of the
operator $A_{\lambda}^c$  converges to
zero almost-everywhere. To understand the major arc part we study various
$r$-variational estimates, see \eqref{eq:20} and \eqref{eq:21}.  The
proof is also to a large extent based on the Fourier transform
estimates \eqref{eq:fourier} for $\mu_c$ and the techniques from
\cite{MSZ1}. Our approach allows us to handle also $\sigma$-finite
measure-preserving systems in contrast to probability systems usually
studied in this context.

Finally, in Section \ref{sec:disc} we prove Theorem \ref{thm:disc} by
adapting Magyar's approach \cite{Ma3, Mag} to our setup. The key
ingredients in this process are tools that we developed to prove Theorem
\ref{thm:asymr} as well as Theorem \ref{thm:L2}, see especially the proof of
Lemma \ref{lem:1}. The technique from the proof of Lemma \ref{lem:1} is
also used to establish Theorem \ref{thm:equi}.

\subsection{Open problems} We close this section with a discussion of some open problems that arise naturally out of this project. 

\vskip.125in  

\begin{enumerate}[label*={\arabic*}.]
\item In this paper we are only concerned with dimension $3$, however
it makes sense to consider all these problems in higher dimensional
setting. Therefore, do Theorem \ref{thm:asymr}, Theorem \ref{thm:erg}, Theorem
\ref{thm:L2}, Theorem \ref{thm:contsph}, Theorem \ref{thm:equi} and
Theorem \ref{thm:disc} remain valid in higher dimensional setup? We
expect that these theorems should have higher dimensional analogues.
However, then the relations between the exponent $c$ and the dimension
$d$ must be understood and should enter somehow into play.

\vskip.125in 

\item Does Theorem \ref{thm:asymr} remain valid for all $c\in(1, 2)$
or even $c>1$?  In our proof it was important that $c>1$ but very
close to $1$. Our method can be refined to obtain a larger range of
$c$ in Theorem \ref{thm:asymr} though still very far from
$c\in(1, 2)$.

\vskip.125in 

\item In all theorems except Theorem \ref{thm:asymr} we were concerned
with the arithmetic spheres $\mathbf{S}^3_c(\lambda)$, however it is
interesting to know whether they remain true with the generalized
spheres $\mathbf{S}_{h_1, h_2, h_3}(\lambda)$ in place of
$\mathbf{S}^3_c(\lambda)$.  An important ingredient in the proofs is
the polar decomposition \eqref{polar} for $\S_c^2$, which is rather
not available for generalized spheres from \eqref{eq:34}. This was the
obstacle that we could not overcome.

\vskip.125in 

\item In Theorem \ref{thm:erg} and Theorem \ref{thm:L2} we obtained
sharp results it terms of the range $p\in(1, \infty)$ for the averages
taken over $\lambda_0$-lacunary sets. It is interesting to know what
is the sharp range of exponents in Theorem \ref{thm:erg} and Theorem
\ref{thm:L2} when averages are taken over the full set of integers
$\Z_+$. So far we have $(11-4c)/(11-7c) < p <\8$ and $c\in(1, 11/10)$.
We expect that these results should hold for all $p\in (3/2, \infty)$
at least for any $c\in(1, 11/10)$. The question about all $c>1$ may be
very hard, especially taking into account the case of $c\in\Z_+$,
which as we have seen \cite{AM, CH, H1, KLM} is very different. This
question is also very interesting in higher dimensions when averages $A_{\lambda}^{c, d}$
are taken over  $\mathbf{S}^d_c(\lambda)$ for any $d\ge 3$. Then the
range of $p$ should depend on $d$ and also on $c$ especially when $c>2$.
It is most likely that the sharp range should be for the exponents
$p>\frac{d}{d-c}$ if $c>2$ and $d\ge3$. This would correspond to the
conjecture formulated by Cook and Hughes \cite[Conjecture 1]{CH}. 

\vskip.125in 

\item In Theorem \ref{thm:erg} we established pointwise convergence by
considering  $r$-variational estimates corresponding to various
pieces of the underlying averages. However, it would be nice to have
$r$-variational estimates for the averages $A_{\lambda}^c$ themselves.

\vskip.125in 

\item A similar question concerns the continuous averages
$\mathcal A_{\lambda}^{c, d}$. In view of Jones, Seeger and Wright
$r$-variational result \cite{JSW} for the classical spherical averages
and the recent endpoint result from \cite{BORSS} we ask about sharp
$L^p(\R^d)$ estimates of $r$-variations for
$\mathcal A_{\lambda}^{c, d}$ in terms of parameters $r$, $p$, $c$ and
$d$. It is a very intriguing question due to unclear relations of
the underlying parameters $r$, $p$, $c$ and $d$. These relations
are expected to be different for $c\in(1, 2)$ and for $c\ge 2$.  Here one will have
to extend the Fourier transform estimates from \eqref{eq:fourier} to
higher dimensions and understand their interactions with the exponent
$c$ and the dimension $d$.

\end{enumerate}

\section{Notation}\label{sec:not}
We now set up notation  that will be used throughout the paper.
\subsection{Basic notation} The sets $\Z$, $\R$, $\C$ and $\T:=\R/\Z$ have standard
meaning. The set of positive integers and nonnegative integers will be
denoted  respectively by $ \Z_+:=\{1, 2, \ldots\}$ and
$\N:=\{0,1,2,\ldots\}$. The set $\mathbb P\subset\Z_+$ denotes the set of the prime numbers.  For any real number $N>0$ we define
\[
[N]:=\Z_+\cap(0, N].
\]

For any $x\in\R$ we will use the floor function
$\lfloor x \rfloor: = \max\{ n \in \Z : n \le x \}$, the fractional
part $\{x\}:=x-\lfloor x\rfloor$ and the distance to the nearest
integer $\|x\|:={\rm dist}(x, \Z)$. For $x, y\in\R$ we shall also
write $x \vee y := \max\{x,y\}$ and $x \wedge y := \min\{x,y\}$.

For two nonnegative quantities
$A, B$ we write $A \lesssim_{\delta} B$ ($A \gtrsim_{\delta} B$) if
there is an absolute constant $C_{\delta}>0$ (which possibly depends
on $\delta>0$) such that $A\le C_{\delta}B$ ($A\ge C_{\delta}B$).  We
will write $A \simeq_{\delta} B$ when $A \lesssim_{\delta} B$ and
$A\gtrsim_{\delta} B$ hold simultaneously. We will omit the subscript
$\delta$ if irrelevant. For a function $f:X\to \C$ and positive-valued
function $g:X\to (0, \infty)$, write $f = O(g)$ if there exists a
constant $C>0$ such that $|f(x)| \le C g(x)$ for all $x\in X$. We will
also write $f = O_{\delta}(g)$ if the implicit constant depends on
$\delta$.  For two functions $f, g:X\to \C$ such that $g(x)\neq0$ for
all $x\in X$ we write $f = o(g)$ if $\lim_{x\to\infty}f(x)/g(x)=0$.

We use $\ind_A$ to denote the indicator function of a set $A$. If $S$ 
is a statement we write $\ind_S$ to denote its indicator, equal to $1$
if $S$ is true and $0$ if $S$ is false. For instance $\ind_A(x)=\ind_{x\in A}$.

\subsection{Euclidean spaces}

For every $d\in\Z_+$ the set
$\{e_i\in\R^d:i\in[d]\}$ will always denote the standard basis in
$\R^d$. The standard inner product on $\R^d$ is denoted by 
\begin{align*}
x\cdot\xi:=\sum_{k\in [d]}x_k\xi_k 
\end{align*}
for every $x=(x_1,\ldots, x_d)$ and
$\xi=(\xi_1, \ldots, \xi_d)\in\R^d$. The inner product induces the  Euclidean norm 
$|x|_2:=\sqrt{x\cdot x}$, which will be abbreviated to $|x|$. We will also consider $\R^d$ with the following norms
\[
|x|_p:=
\begin{cases}
\big(\sum_{i\in[d]}|x_i|^p\big)^{1/p}& \text{ if } \quad p\in[1, \infty),\\
\ \max_{k\in[d]}|x_k| & \text{ if } \quad p=\infty.
\end{cases}
\]

The $d$-dimensional  torus $\T^d:=\R^d/\Z^d$ is a priori endowed with the   periodic norm
\[
\|\xi\|:=\Big(\sum_{k=1}^d \|\xi_k\|^2\Big)^{1/2}
\qquad \text{for}\qquad
\xi=(\xi_1,\ldots,\xi_d)\in\T^d,
\]
where $\|\xi_k\|={\rm dist}(\xi_k, \Z)$ for all $\xi_k\in\T$ and
$k\in[d]$.  However, identifying $\T^d$ with $[-1/2, 1/2)^d$, we
see that the norm $\|\cdot\|$ coincides with the Euclidean norm
$|\cdot|$ restricted to $[-1/2, 1/2)^d$.

The unit sphere in $\R^d$ induced by the norm $|\cdot|_p$ is defined by
\begin{align}\label{id:307}
\S_p^{d-1}:=\{x\in\R^d: |x|_p=1\}.
\end{align}
If $p=2$ we shall abbreviate $\S_p^{d-1}$ to $\S^{d-1}$, which is the standard Euclidean sphere.

\subsection{The $L^p$ spaces} The triple $(X, \mathcal B(X), \mu)$
is a measure space $X$ with $\sigma$-algebra $\mathcal B(X)$ and
$\sigma$-finite measure $\mu$.  The space of all $\mu$-measurable
complex-valued functions defined on $X$ will be denoted by $L^0(X)$.
The space of all functions in $L^0(X)$ whose modulus is integrable
with $p$-th power is denoted by $L^p(X)$ for $p\in(0, \infty)$,
whereas $L^{\infty}(X)$ denotes the space of all essentially bounded
functions in $L^0(X)$. In our case we will usually have $X=\R^d$ or
$X=\T^d$ equipped with the Lebesgue measure, and $X=\Z^d$ endowed with the
counting measure. If $X$ is endowed with counting measure we will
abbreviate $L^p(X)$ to $\ell^p(X)$.

Let $X$ be a locally compact Hausdorff space. Then $C(X)$ denotes the
space of all continuous functions on $X$. $C_c(X)$ denotes the space
of all continuous and compactly supported functions on $X$. $C_0(X)$
denotes the space of all continuous functions on $X$ that vanish at
infinity. Finally, let $U\subseteq \R^d$ be open, for any $n\in\Z_+$
let $C^{n}(U)$ denote the space of all functions $f$ on $U$ whose
partial derivatives of order $\le n$ all exist and are continuous. We
also set $C^{\infty}(U):=\bigcap_{n\in\Z_+}C^{n}(U)$ and
$C^{\infty}_c(\R^d)$ denotes the set of all compactly supported smooth
functions on $\R^d$.  If $U\subseteq \R^d$ it also make sense to
consider the spaces $C^{n}(U)$. In this case we say that
$f\in C^{n}(U)$ if $f\in C^{n}(V)$ for some open $V\supseteq U$.

The partial derivative of a function $f:\R^d\to\C$ with respect to the
$j$-th variable $x_j$ will be denoted by
$\partial_{x_j}f=\partial_j f$, while the $m$-th partial derivative
with respect to the $j$-th variable will be denoted by
$\partial_{x_j}^mf=\partial_j^m f$. The gradient of a function
$f:\R^d\to\C$ is the following vector
$\nabla f := (\partial_{x_1}f,\ldots, \partial_{x_d}f):= (\partial_1f,\ldots, \partial_df)$.
If $\alpha\in\N^d$ is a multi-index $|\alpha|:=\alpha+\ldots+\alpha_d$
denotes its size, and it will be always clear from the context that
$|\alpha|$ is not the Euclidean norm of $\alpha\in\N^d$. We shall also
write $\alpha!:=\alpha_1!\ldots\alpha_d!$ and
$\binom{\alpha}{\beta}:= \frac{\alpha!}{\beta!(\alpha-\beta)!}$ for
all multi-indices $\alpha=(\alpha_1,\ldots, \alpha_d)\in\N^d$ and
$\beta=(\beta_1,\ldots, \beta_d)\in\N^d$ such that $\alpha\ge\beta$,
where the last relation means that $\alpha_j\ge \beta_j$ for all
$j\in[d]$. For any $\alpha\in\N^d$ let $\partial^{\alpha}f$ denote the
derivative
$\partial^{\alpha_1}_{x_1}\ldots \partial^{\alpha_d}_{x_d}f=\partial^{\alpha_1}_1 \ldots \partial^{\alpha_d}_df$
operator of total order $|\alpha|$. If
$f:\R^{d_1}\times \R^{d_2}\to\C$ for some $d_1, d_2\in\Z_+$, and let
$(x, y)\in \R^{d_1}\times \R^{d_2}$ then $\nabla_{\!x} f(x, y)$ denotes
the gradient of $f$ with respect to the variable $x\in\R^{d_1}$ and we
write
$\nabla_{\!x} f(x, y):=(\partial_{x_1}f(x, y),\ldots, \partial_{x_{d_1}}f(x, y))$
and $\partial_{x}^{\alpha}f(x, y)$ denote the partial differential
operator of order $\alpha\in\N^{d_1}$ with respect to the variable
$x\in\R^{d_1}$ and we write
$\partial_{x}^{\alpha}f(x, y):=\partial^{\alpha_1}_{x_1} \ldots \partial^{\alpha_d}_{x_d}f(x, y)$. We
have analogous definitions for the variable $y\in\R^{d_2}$.

\subsection{$r$-variations}
For any family
 $(\mathfrak a_t: t\in\mathbb I)$ of elements of $\C$ indexed by a totally
 ordered set $\mathbb I$, and any exponent $1 \leq r < \infty$, the
 $r$-variation seminorm is defined to be
 \begin{align*}
 V^r( a_t: t\in\mathbb I):=
 \sup_{J\in\Z_+} \sup_{\substack{t_{0}<\dotsb<t_{J}\\ t_{j}\in\mathbb I}}
\Big(\sum_{j=0}^{J-1}  |a(t_{j+1})-a(t_{j})|^{r} \Big)^{1/r},
 \end{align*}
 where the  supremum is taken over all finite increasing sequences in $\mathbb I$.

 It is easy to see that for every $t_0\in \mathbb I$ one has
 \[
\sup_{t\in \mathbb I}|a_t|\le |a_{t_0}|+  V^r( a_t: t\in\mathbb I).
\]

 We will usually use $r$-variations with $\mathbb I=\Z_+$ or
 $\mathbb I=(0, \infty)$ or $\mathbb I=\mathbb D$, where $\mathbb D$
 is $\lambda_0$-lacunary subset of $(0, \infty)$. Then if
 $ V^r( a_t: t\in\mathbb I)<\infty$ we see that the underlying
 sequence $(a_t: t\in\mathbb I)$ is a Cauchy sequence and consequently in
each case the limit $\lim_{t\to\infty}a_t$ exists, and in
 the second case when $\mathbb I=(0, \infty)$ the limit
 $\lim_{t\to0}a_t$ also exists.
 
\subsection{Fourier transform}  
We will write $e(z):=e^{2\pi i z}$ for every $z\in\C$.
The Fourier transform and the inverse Fourier transform of $f\in L^1(\R^d)$ will be denoted respectively by
\begin{align*}
\mathcal F_{\R^d} f(\xi) &:= \int_{\R^d} f(x) e(-x\cdot \xi)\ dx,\qquad \xi\in\R^d,\\
 \mathcal F_{\R^d}^{-1} f(x) &:= \int_{\R^d} f(\xi) e(x\cdot \xi)\ d\xi,\qquad x\in\R^d.
\end{align*}
The Fourier coefficient of $f\in L^1(\T^d)$,   and the Fourier series of $g\in \ell^1(\Z^d)$ will be denoted respectively by
\begin{align*}
\mathcal F_{\Z^d} f(n) &:= \int_{\T^d} f(x) e(-n\cdot \xi)\ d\xi,\qquad n\in\Z^d,\\
 \mathcal F_{\Z^d}^{-1} g(\xi) &:= \sum_{n\in\Z^d} g(n) e(n\cdot \xi),\qquad \xi\in\T^d.
\end{align*}

%%%%%%%%%%%%%%%%%%%%%%%%%%%%%%%%%%%%%%%%%%%%%%%%%%%%%%%%%%%%%%%%%%%%%%%%%%%%%%%%%%%
\section{Proof of Theorem \ref{thm:asymr}. Asymptotic formula for $r_{h_1, h_2, h_3}(\la)$}\label{sec:asym}
%%%%%%%%%%%%%%%%%%%%%%%%%%%%%%%%%%%%%%%%%%%%%%%%%%%%%%%%%%%%%%%%%%%%%%%%%%%%%%%%%%%

We begin with the observation that the asymptotic behaviour of
$r_{h_1, h_2, h_3}(\la)$ is not affected by removing the restriction that all
coordinates of a given triple are larger than $N_0$.

\begin{rem}\label{rem:1}
We note that the values of functions $h_1, h_2, h_3$ for $n\in[N_0-1]$ do not change
the asymptotic formula of $r_{h_1, h_2, h_3}(\la)$ obtained in Theorem \ref{thm:asymr}. More
precisely, assume for any $k\in[3]$ that $h_k \in \mathcal{F}_{c_k}$, and
they take arbitrary values for $n\in[N_0-1]$. We define
\begin{align*}
\widetilde{r}_{h_1, h_2, h_3}(\la)
:= \# \{ (m_1, m_2, m_3) \in \Z_+^3 : \lfloor h_1(m_1) \rfloor + \lfloor h_2(m_2) \rfloor + \lfloor h_3(m_3) \rfloor = \la \}.
\end{align*} 
Then one can easily show that
\[
0 \le \widetilde{r}_{h_1, h_2, h_3}(\la) - r_{h_1, h_2, h_3}(\la) = o \left(  \la^2 \vp_1'(\la) \vp_2'(\la) \vp_3'(\la) \right). 
\]
\end{rem}

%%%%%%%%%%%%%%%%%%%%%%%%%%%%%%%%%%%%%%%%%%%%%%%%%%%%%%%%%%%%%%%%%%%%%%%%%%%%%%%%%%%
\subsection{Properties of $h$ and $\varphi$}\label{sec:properties}
%%%%%%%%%%%%%%%%%%%%%%%%%%%%%%%%%%%%%%%%%%%%%%%%%%%%%%%%%%%%%%%%%%%%%%%%%%%%%%%%%%%

In this subsection we gather some properties of functions
$h\in\mathcal{F}_c$ and their inverses, which will be used throughout
the paper. We pick a large number  $N_0\in\Z_+$ such that for every $x \ge N_0$ we have
\begin{align}\label{condu1}
h'(x) \ge 1.
\end{align}
Then for every $n \ge h(N_0)$, as in \cite[Lemma 2.12, p. 624]{M1}, one has 
\begin{align}\label{condu2}
\lfloor - \vp (n) \rfloor - \lfloor - \vp (n+1) \rfloor 
= \ind_{ \Nh } (n),
\end{align}
where $\varphi$ is the inverse of $h$. The next lemma shows that the function $\vp$ is $1/c$-regular.

\begin{lem}\label{formfunlem}
Assume that $c\in[1, 2)$, $h\in\mathcal{F}_c$, $\g=1/c$ and let $\vp:[h(x_0), \8)\to[x_0, \8)$ be the inverse of h. Then
\begin{align}\label{funfi}
  \vp(x):=x^{\g}\ell_{\vp}(x),\qquad \mbox{where}\qquad \ell_{\vp}(x):=\exp\bigg(\int_{h(x_0)}^x\frac{\te(t)}{t}dt+D\bigg),
\end{align}
  for every $x\ge h(x_0)$, where $D=\log\big(x_0/h(x_0)^{\g}\big)$ and
  \begin{align*}
  \te(x):=\frac{1}{(c+\vartheta(\vp(x)))}-\g
  =-\frac{\vartheta(\vp(x))}{c(c+\vartheta(\vp(x)))},
\end{align*}
with $\vartheta$ as in Definition \ref{defn}(ii), and
\begin{equation}
\label{eq:limteta}
\lim_{x\to\8}\te(x)=0.
\end{equation}

  If $L(x)=\ell_h(x)$ or $L(x)=\ell_{\vp}(x)$, then for every $\e>0$ we have
  \begin{align}\label{slowhfi}
    \lim_{x\to\8}x^{-\e}L(x)=0,\ \ \ \mbox{and}\ \ \ \lim_{x\to\8}x^{\e}L(x)=\8,
  \end{align}
  and consequently, for every $\e>0$
  \begin{align}\label{ratefi}
    x^{\g-\e}\lesssim_{\e}\vp(x)\ \ \ \mbox{and}\ \ \ \lim_{x\to\8}\frac{\vp(x)}{x}=0.
  \end{align}
  Furthermore, $x\mapsto x\vp(x)^{-\d}$ is increasing for every $\d< c$, (if $c=1$, even $\d\le1$ is allowed) and for every $x\ge h(x_0)$ we have
  \begin{align}\label{compfi}
    \vp(x)\simeq\vp(2x)\ \ \mbox{and}\ \ \vp'(x)\simeq\vp'(2x).
  \end{align}
\end{lem}

\begin{proof}
The proof of Lemma \ref{formfunlem} can be found in \cite[Lemma 2.6, p. 623]{M1}, we also refer to \cite{M2} as well as to \cite{Piat}. 
\end{proof}

The properties of the derivatives of function $\vp$ are gathered in the lemma below. 

\begin{lem}\label{funlemfi}
Assume that $c\in[1, 2)$, $h\in\mathcal{F}_c$, $\g=1/c$ and let $\vp:[h(x_0), \8)\to[x_0, \8)$ be its inverse. Then
 for every $n\in[3]$ there exists  a function $\theta_n:[h(x_0), \8)\to\R$ such that $\lim_{x\to\8}\te_n(x)=0$ and
 \begin{align}\label{fiequat}
   x\vp^{(n)}(x)=\vp^{(n-1)}(x)(\b_n+\theta_n(x)), \ \ \mbox{for every \  $x\ge h(x_0)$,}
 \end{align}
  where $\b_n=\g-n+1$, $\te_1(x)=\te(x)$ and
	\begin{equation*}
\theta_i(x)=\theta_{i-1}(x)+\frac{x\theta_{i-1}'(x)}{\beta_{i-1}+\theta_{i-1}(x)},
\qquad i = 2, 3.
\end{equation*}

If $c=1$, then there exist a positive function
$\s:[h(x_0), \8)\to(0, \8)$ and a function $\t:[h(x_0), \8)\to \R$
such that \eqref{fiequat} with $n=2$ reduces to
\begin{align}\label{fiequat1}
  x\vp''(x)=\vp'(x)\s(x)\t(x),\ \ \mbox{for every \  $x\ge h(x_0)$.}
\end{align}
Moreover, $\s(x)$  is decreasing, $\lim_{x\to\8}\s(x)=0,$ $\s(2x)\simeq\s(x),$ and  $\s(x)^{-1}\lesssim_{\varepsilon}x^{\varepsilon},$
for every $\varepsilon>0$. Finally, there are constants $0<c_3\le c_4$ such that  and $c_3\le-\t(x)\le c_4$ for every $x\ge h(x_0)$.
\end{lem}
\begin{proof}
The proof of Lemma \ref{funlemfi} can be found in \cite[Lemma 2.14,
p. 625]{M1}, we also refer to \cite{M2} as well as to \cite{Piat}.
\end{proof}

We will also need the following result.

\begin{lem}\label{lem:lim} 
For $k\in[2]$, let $c_k \in [1,2)$, $\g_k=1/c_k$,
$h_k \in \mathcal{F}_{c_k}$, and let $\vp_{k}$ be their inverses,
respectively.  Further, let $L_{\vp}(x) := \vp'(x)/x^{\g-1}$. Then
\begin{align*}
\lim_{\la \to\8}\int_{N_0/\la}^{1-N_0/\la} x^{\g_1-1}(1-x)^{\g_2-1}
\frac{L_{\vp_1}(\la x)}{L_{\vp_1}(\la)}\frac{L_{\vp_2}(\la(1-x))}{L_{\vp_2}(\la)} \, dx
= \frac{\Gamma(\g_1)\Gamma(\g_2)}{\Gamma(\g_1+\g_2)}.
\end{align*}
\end{lem}

\begin{proof}
By the definition of the Beta function it suffices to show
\begin{align*}
\lim_{\la\to\8}\int_{N_0/\la}^{1-N_0/\la} x^{\g_1-1}(1-x)^{\g_2-1}
\frac{L_{\vp_1}(\la x)}{L_{\vp_1}(\la)}
\frac{L_{\vp_2}(\la (1-x))}{L_{\vp_2}(\la)} \, dx
= 
\int_0^1 x^{\g_1-1}(1-x)^{\g_2-1} \, dx.
\end{align*}
For this purpose we will use the dominated convergence theorem, which
reduces the matter to proving, with $\vp = \vp_1$ or $\vp = \vp_2$,
that
\begin{align}\label{lim1}
\lim_{\la\to\8} \frac{L_\vp(\la x)}{L_\vp(\la)} = 1,
\qquad x \in (0,1),
\end{align}
and 
\begin{equation}\label{est2}
\frac{L_{\vp}(\la x)}{L_{\vp}(\la)} \lesssim x^{-\g/100}, \qquad x \ge N_0/\la.
\end{equation}

We first deal with \eqref{lim1}. Using \eqref{fiequat} and \eqref{funfi} we conclude
\begin{equation*}
\vp'(x)=\frac{\vp(x)}{x}(\g+\te_1(x))=x^{\g-1}\ell_{\vp}(x)(\g+\te(x)),
\end{equation*}
which yields $L_{\vp}(x) = \ell_{\vp}(x)(\g+\te(x))$. Applying \eqref{funfi} once again we infer that
\begin{align}\label{iden1}
\frac{L_\vp(\la x)}{L_\vp(\la)} 
=
\exp\left(-\int_{\la x}^\la\frac{\te(t)}{t} \, dt \right) 
\frac{\g + \te(\la x)}{\g + \te(\la)},
\end{align}
for sufficiently large $\la\in\Z_+$. 
By \eqref{eq:limteta} the last fraction  converges to $1$ as 
$\la \to \8$.
Therefore, to verify \eqref{lim1} we observe
 \begin{equation}\label{est1}
\Big|\int_{\la x}^\la\frac{\te(t)}{t}dt\Big|
\le 
\int_{\la x}^\la\frac{|\te(t)|}{t}dt
\le 
\max_{t\in[\la x,\la]}|\te(t)|\log(1/x)\xrightarrow[\la\to\8]{} 0. 
\end{equation}
 
We now prove \eqref{est2}. Using again \eqref{eq:limteta} (notice that $\la \ge \la x \ge N_0$) we see
$$
\frac{\g+\te(\la x)}{\g+\te(\la)}<\frac{101\g/100}{99\g/100}=101/99, 
\qquad x \ge N_0/\la.
$$
Applying \eqref{est1} and  \eqref{eq:limteta} we obtain
\begin{align*}
\exp\left(-\int_{\la x}^\la \frac{\te(t)}{t} \, dt \right)
\le x^{-\g/100}, \qquad x \ge N_0/\la.
\end{align*}
These estimates together with \eqref{iden1} lead to \eqref{est2}.
The proof of Lemma \ref{lem:lim} is finished.
\end{proof}

%%%%%%%%%%%%%%%%%%%%%%%%%%%%%%%%%%%%%%%%%%%%%%%%%%%%%%%%
%%%%%%%%%%%%%%%%%%%%%%%%%%%%%%%%%%%%%%%%%%%%%%%%%%%
\subsection{Some asymptotic formulas}\label{sec:J3}
%%%%%%%%%%%%%%%%%%%%%%%%%%%%%%%%%%%%%%%%%%%%%%%%%%%%%%%%
%%%%%%%%%%%%%%%%%%%%%%%%%%%%%%%%%%%%%%%%%%%%%%%%%%%

For $k\in[3]$, let $c_k \in [1,2)$, $h_k \in \mathcal{F}_{c_k}$ and $\vp_{k}$ be its inverse. 
We will describe an asymptotic behavior of the following function 
$$
J_{\vp_1', \vp_2', \vp_3' }(\la):=
\sum_{\genfrac{}{}{0pt}{}{n_1+n_2+n_3=\la}{n_1, n_2, n_3 \ge N_0}}
\vp_1'(n_1) \vp_2'(n_2) \vp_3'(n_3), \qquad \la \in \Z_+,
$$
We  begin with a simpler object
$$
J_{\vp_1', \vp_2'}(\la):=\sum_{m=N_0}^{\la-N_0} \vp_1'(m) \vp_2'(\la-m), \qquad \la \in \Z_+,
$$
and we prove the following result.

\begin{lem}\label{lem:J2}
For $k\in[2]$, let $c_k \in [1,2)$, $\g_k=1/c_k$,
$h_k \in \mathcal{F}_{c_k}$ and $\vp_{k}$ be its inverse.  Then
\begin{align*}
  J_{\vp_1', \vp_2'}(\la)=\frac{\Gamma(\g_1)\Gamma(\g_2)}{\Gamma(\g_1+\g_2)}\la\vp_1'(\la)\vp_2'(\la)
  +o\left(\la\vp_1'(\la)\vp_2'(\la)\right).
\end{align*}
\end{lem}

\begin{proof}
We first show for sufficiently large $\la\in\Z_+$ that 
\begin{align} \label{est13}
\Big| \sum_{m=N_0}^{\la-N_0}\vp_1'(m)\vp_2'(\la-m)
-
\int_{N_0}^{\la - N_0}\vp_1'(x)\vp_2'(\la-x) \, dx \Big|
\lesssim
\la^{3/4} \vp_1'(\la)\vp_2'(\la).
\end{align}

Using \eqref{fiequat} and \eqref{ratefi} we may write
$\vp_1' (\la-N_0) \vp_2' (N_0) \lesssim \vp_1' (\la) \lesssim \la^{3/4} \vp_1'(\la) \vp_2'(\la)$.
Thus  \eqref{est13} will follow if we show that
\begin{align} \label{est14}
\sum_{m=N_0}^{\la-N_0-1} \int_{m}^{m + 1} 
\big| \vp_1'(m) \vp_2'(\la-m)
-
\vp_1'(x) \vp_2'(\la-x) \big|  \, dx 
\lesssim
\la^{3/4} \vp_1'(\la)\vp_2'(\la).
\end{align}
Let $\psi_\la (x) = \vp_1'(x) \vp_2'(\la-x)$, $N_0 \le x \le \la - N_0$. 
Using \eqref{fiequat} with $n=2$ we have
\begin{align*}
\psi'_\la (x) 
= \vp_1'(x) \vp_2'(\la-x)
\Big( 
\frac{\g_1 - 1 + \theta_{1,2} (x)}{x} - \frac{\g_2 - 1 + \theta_{2,2}(\la - x)}{\la - x}
\Big),
\end{align*}
where $\theta_{k,2}$ is a $\theta_{2}$ function from Lemma~\ref{funlemfi} corresponding to $\vp_k$ for $k\in[2]$.
This together with \eqref{compfi} shows that
\begin{align*}
|\psi_\la (m) 
-
\psi_\la (x)|  
\lesssim
\vp_1'(x) \vp_2'(\la-x) \Big( 
\frac{1}{x} + \frac{1}{\la - x}
\Big), \qquad x \in (m,m+1).
\end{align*}
Using  \eqref{compfi} again, we see that the left-hand side of \eqref{est14} is controlled by
\begin{align*}
\int_{N_0}^{\la - N_0}
\vp_1'(x) &\vp_2'(\la-x) 
\Big( 
\frac{1}{x} + \frac{1}{\la - x}
\Big) \, dx\\
& \lesssim
\int_{N_0}^{\la/2}
\frac{\vp_1'(N_0) \vp_2'(\la) }{x} \, dx
+
\int_{\la/2}^{\la - N_0} 
\frac{\vp_1'(\la) \vp_2'(N_0)}{\la - x} \, dx \\
& \lesssim
\vp_2'(\la) \log \la
+
\vp_1'(\la) \log \la \\
& \lesssim
\la^{3/4} \vp_1'(\la)\vp_2'(\la),
\end{align*}
where in the last estimate we have used Lemma~\ref{funlemfi} and
\eqref{ratefi}.  We are reduced to estimate the integral instead of
the sum.  Changing the variables, using the identity
$\vp_k'(x) = x^{\g_k-1} L_{\vp_k}(x)$ for $k\in[2]$, and applying
Lemma \ref{lem:lim} we conclude
\begin{align*}
\int_{N_0}^{\la-N_0}\vp_1'(x)&\vp_2'(\la-x)dx\\
&=\la\int_{N_0/\la}^{1-N_0/\la}\vp_1'(\la x)\vp_2'(\la(1-x))dx\\
  &=\la^{\g_1+\g_2-1}
\int_{N_0/\la}^{1-N_0/\la} x^{\g_1-1}(1-x)^{\g_2-1}L_{\vp_1}(\la x)L_{\vp_2}(\la(1-x))dx\\
  &=\la \vp_1'(\la) \vp_2'(\la)
\int_{N_0/\la}^{1-N_0/\la} x^{\g_1-1}(1-x)^{\g_2-1} \frac{L_{\vp_1}(\la x)}{L_{\vp_1}(\la)}
  \frac{L_{\vp_2}(\la(1-x))}{L_{\vp_2}(\la)} \, dx\\
  &=\frac{\Gamma(\g_1)\Gamma(\g_2)}{\Gamma(\g_1+\g_2)}\la\vp_1'(\la)\vp_2'(\la)
  +o\left(\la \vp_1'(\la)\vp_2'(\la) \right).
\end{align*}
The proof of Lemma \ref{lem:J2} is finished.
\end{proof}

\begin{rem}
Lemma \ref{lem:J2} is optimal in a sense that the error term
$o\left(\la\vp_1'(\la)\vp_2'(\la)\right)$ cannot be replaced in the
statement of this result by
$O\left(\la^{1 - \eps} \vp_1'(\la) \vp_2'(\la)\right)$ for any
$\eps > 0$. Indeed, a careful inspection of the proof shows that it is
equivalent to a stronger version of Lemma \ref{lem:lim} below. More
precisely, this would imply that there exists $\eps > 0$ such that
\begin{align} \label{est4}
\bigg|
\int_{N_0/\la}^{1-N_0/\la} x^{\g_1-1}(1-x)^{\g_2-1}
\frac{L_{\vp_1}(\la x)}{L_{\vp_1}(\la)}
  \frac{L_{\vp_2}(\la(1-x))}{L_{\vp_2}(\la)} \, dx
- \frac{\Gamma(\g_1)\Gamma(\g_2)}{\Gamma(\g_1+\g_2)}
\bigg|
\lesssim \la^{-\eps}
\end{align}
for large $\la \in \Z_+$. However, taking $\vp_2 (x) = x^{\g_2}$,
$\g_2 = 1/c_{2} \in (1/2,1)$, and $\vp_1 (x) = x^{\g_1} \log x$,
$\g_1 = 1/c_{1} \in (1/2,1)$, we can easily show that the left-hand
side of \eqref{est4} is comparable with $(\log \la)^{-1}$ for
sufficiently large $\la \in \Z_+$, which is in contradiction with
\eqref{est4}.
\end{rem}

\begin{cor} \label{cor:J2}
For $k\in[2]$, let $c_k  \in (1,2)$,  $\g_k=1/c_k$ and $\vp_k(x) = x^{\g_k}$, then we have 
\begin{align*}
  J_{\vp_1', \vp_2'}(\la)= \frac{\Gamma(\g_1) \Gamma(\g_2)}{\Gamma(\g_1+\g_2)} \g_1 \g_2 \la^{\g_1+\g_2 - 1}
  + O\left( \la^{\g_1+\g_2 - 1 - \g_1 \wedge \g_2} \right).
\end{align*}
\end{cor}

\begin{proof}
Repeating the arguments used in the proof of Lemma~\ref{lem:J2} the claim follows. We omit the details. 
\end{proof}

Using Lemma \ref{lem:J2} we will describe an asymptotic behavior of 
$J_{\vp_1', \vp_2', \vp_3' }(\la)$.

\begin{prop}\label{prop:J3}
For $k\in[3]$, let $c_k \in [1,2)$, $\g_k=1/c_k$,
$h_k \in \mathcal{F}_{c_k}$ and $\vp_{k}$ be its inverse.  Then
\begin{align*}
  J_{\vp_1', \vp_2', \vp_3'}(\la)
	=
	\frac{\Gamma(\g_1) \Gamma(\g_2) \Gamma(\g_3) }{\Gamma(\g_1 + \g_2 + \g_3 )}
	\la^2 \vp_1'(\la) \vp_2'(\la) \vp_3'(\la)
  +o\left( \la^2 \vp_1'(\la) \vp_2'(\la) \vp_3'(\la) \right).
\end{align*}
\end{prop}

\begin{proof}
The proof will be in some sense inductive. Observe that 
\begin{align*}
  J_{\vp_1', \vp_2', \vp_3'}(\la)
	& =
\sum_{n_1 = N_0}^{\la-N_0} \vp_1'(n_1)
\sum_{\genfrac{}{}{0pt}{}{n_2+n_3 = \la - n_1}{n_2, n_3 \ge N_0}}
\vp_2'(n_2) \vp_3'(n_3)
 = 
\sum_{n_1 = N_0}^{\la-N_0} \vp_1'(n_1)
J_{\vp_2', \vp_3'}(\la- n_1).
\end{align*}
Let $\psi(x) := x \vp_2'(x) \vp_3'(x)$. Using Lemma \ref{lem:J2} we see that
\begin{align}\label{iden2}
  J_{\vp_1', \vp_2', \vp_3'}(\la)
=
\frac{ \Gamma(\g_2) \Gamma(\g_3) }{\Gamma(\g_2 + \g_3 )}
\sum_{n = N_0}^{\la-N_0} \vp_1'(n) \psi (\la-n)
+
\sum_{n = N_0}^{\la-N_0} \vp_1'(n) h (\la-n),
\end{align}
where $h$ is a function satisfying $h(\la)=o(\psi(\la))$.

We claim that 
\begin{align}\label{red1}
 \sum_{n = N_0}^{\la-N_0} \vp_1'(n) \psi (\la-n)
=
\frac{ \Gamma(\g_1) \Gamma(\g_2 + \g_3) }{\Gamma(\g_1 + \g_2 + \g_3 )}
\la \vp_1'(\la) \psi(\la)
+ 
o( \la \vp_1'(\la) \psi(\la) ).
\end{align}
Assume for a moment that we can verify \eqref{red1}. Then
applying this to the first term on the right-hand side of
\eqref{iden2} it suffices to show that
\begin{align}\label{red2}
 \sum_{n = N_0}^{\la-N_0} \vp_1'(\la - n) h (n)
=
o( \la \vp_1'(\la) \psi(\la) ).
\end{align}
To prove \eqref{red2} we split the summation in \eqref{red2} into two
parts $ \sum_{n = N_0}^{N_1}$ and $\sum_{n = N_1+1}^{\la-N_0}$, where
$N_1\in\Z_+$ is a number such that for all $n> N_1$ the asymptotic
$h(n)=o(\psi(n))$ will get efficient. For $n\in [N_0, N_1]$ we apply
\eqref{fiequat} (for the first derivative) together with \eqref{funfi}
and \eqref{slowhfi} to conclude
$\vp_1'(\la - n) h (n) = o( \la \vp_1'(\la) \psi(\la) )$, since
$h(n)=O(1)$ for $n \in [N_0, N_1]$. Further, for the remaining $n>N_1$ we use the fact that
$h(\la)=o(\psi(\la))$ and once again \eqref{red1} to obtain
\eqref{red2}.
The proof of Proposition
\ref{prop:J3} will be completed if we verify \eqref{red1}.

To prove \eqref{red1} we shall proceed much the same way as in the
proof of Lemma \ref{lem:J2}. For large $\la\in\Z_+$ we show that
\begin{align*}
\Big| \sum_{n = N_0}^{\la-N_0} \vp_1'(n) \psi (\la - n)
-
\int_{N_0}^{\la-N_0} \vp_1'(x) \psi (\la - x) \, dx \Big|
\lesssim
\la^{3/4} \vp_1'(\la) \psi(\la).
\end{align*}

Observe that $\vp_1'(\la - N_0) \psi (N_0) \lesssim \vp_1'(\la) \lesssim \la^{3/4} \vp_1'(\la) \psi(\la)$. It suffices to show that
\begin{align} \label{est15}
\sum_{n = N_0}^{\la-N_0 - 1} 
\int_{n}^{n+1} \big| g_\la (n) - g_\la(x) \big| \, dx 
\lesssim
\la^{3/4} \vp_1'(\la) \psi(\la),
\end{align}
where $g_\la (x) := \vp_1'(x)\psi(\la-x)$, $N_0 \le x \le \la - N_0$.
Taking Lemma~\ref{funlemfi} into account we have
\begin{align*}
g_\la'(x)
= g_\la(x) \left( \frac{\g_1 - 1 + \te_{1,2} (x) }{x} 
  -  \frac{\g_2 + \g_3 -1 + \te_{2,2} (\la-x) + \te_{3,2} (\la-x)}{\la-x}  \right),
\end{align*}
where $\te_{k,2}$ is a $\theta_{2}$ function from Lemma~\ref{funlemfi} corresponding  to $\vp_k$ for $k\in[3]$.
This shows that
\begin{align*}
| g_\la'(x) |
\lesssim
\frac{g_\la(x)}{x\wedge (\la-x)}, \qquad N_0 \le x \le \la - N_0.
\end{align*}
Therefore combining this with \eqref{compfi} we see that the left-hand side of \eqref{est15} is dominated by
\begin{align*}
\int_{N_0}^{\la - N_0}
\frac{g_\la(x)}{x\wedge (\la-x)} \, dx
& \lesssim
\int_{N_0}^{\la/2} \vp_1'(N_0) \psi (\la) \frac{dx}{x}
+
\int_{\la/2}^{\la - N_0} \vp_1'(\la) \psi (N_0) \frac{dx}{\la - x} \\
& \lesssim
\psi (\la) \log \la + \vp_1'(\la) \log \la \\
& \lesssim
\la^{3/4} \vp_1'(\la) \psi(\la),
\end{align*}
where in the last estimate we have used Lemma~\ref{funlemfi} and
\eqref{ratefi}.  To justify \eqref{red1}, it suffices to verify
\eqref{red1} with $\int_{N_0}^{\la-N_0} g_\la (x) \, dx$ in place of
$ \sum_{n = N_0}^{\la-N_0} \vp_1'(n) \psi (\la-n)$.

Changing the variable $x \mapsto \la x$ and using the identity $\vp_k'(x) = x^{\g_k-1} L_{\vp_k}(x)$ for $k\in[3]$, we infer that
\begin{align*}
\int_{N_0}^{\la-N_0} g_\la (x) \, dx
& =
\la^2 \vp_1'(\la) \vp_2'(\la) \vp_3'(\la) 
\int_{N_0/\la}^{1-N_0/\la} x^{\g_1 - 1}(1-x)^{\g_2 + \g_3 - 1} L_{\vp_1, \vp_2, \vp_3}(\la, x) \, dx,
\end{align*}
where
\[
L_{\vp_1, \vp_2, \vp_3}(\la, x):= \frac{L_{\vp_1}(\la x)}{L_{\vp_1}(\la)}
  \frac{L_{\vp_2}(\la(1-x))}{L_{\vp_2}(\la)}
\frac{L_{\vp_3}(\la(1-x))}{L_{\vp_3}(\la)}.
\]
Now it suffices to show that
\[
\lim_{\la \to \8} 
\int_{N_0/\la}^{1-N_0/\la} x^{\g_1 - 1}(1-x)^{\g_2 + \g_3 - 1} L_{\vp_1, \vp_2, \vp_3}(\la, x) \, dx
=
\frac{ \Gamma(\g_1) \Gamma(\g_2 + \g_3) }{\Gamma(\g_1 + \g_2 + \g_3 )}.
\]
This, however, follows from the dominated convergence theorem upon appealing to  \eqref{lim1} and \eqref{est2} as in Lemma \ref{lem:lim}.
The proof of Proposition \ref{prop:J3} is complete.
\end{proof}

\begin{cor} \label{cor:J3}
For $k\in[3]$, let $c_k \in (1,2)$, $\g_k=1/c_k$ and $\vp_k(x) = x^{\g_k}$, then we have 
\begin{align*}
  J_{\vp_1', \vp_2', \vp_3'}(\la)= 
	\frac{\Gamma(\g_1) \Gamma(\g_2) \Gamma(\g_3)}{\Gamma(\g_1+\g_2+\g_3)} \g_1 \g_2 \g_3 \la^{\g_1+\g_2+\g_3 - 1}
  + O\left( \la^{\g_1+\g_2+\g_3 - 1 - \g_1 \wedge \g_2 \wedge \g_3} \right).
\end{align*}
\end{cor}

\begin{proof}
It suffices to proceed as in the proof of Proposition~\ref{prop:J3} and use Corollary~\ref{cor:J2} instead of Lemma~\ref{lem:J2}. The details are left to the reader. 
\end{proof}

\subsection{Exponential sum estimates} \label{sec:FG}
To prove Theorem \ref{thm:asymr} we need some auxiliary results.
We will need the discrete version of the Van der Corput lemma.

\begin{lem}[Van der Corput, {\cite[Corollary 8.13]{IK}}]\label{lem:VdC}
Assume that $a,b \in \R$ are such that $b - a \ge 1$. Let
$F \in C^2([a,b])$ be a real valued function and let $I$ be a
subinterval of $[a,b]$ such that $|I| \ge 1$. If there exist
$\eta > 0$ and $r \ge 1$ such that
\[
\eta \lesssim |F''(x)| \lesssim r\eta, \qquad x \in I,
\]
then
\[
\Big| \sum_{k \in I} e(F(k))
\Big|
\lesssim
r |I| \eta^{1/2} + \eta^{-1/2}.
\]
\end{lem}

We shall also need the following variant of the summation by parts formula.

\begin{lem}[Summation by parts, {\cite[Theorem A.4, p.\,304]{Nat}}]\label{lem:SBP}
Let $u(n)$ and $g(n)$ be arithmetic functions and let $x,y \in \R$ be
such that $0 \le y < x$. Then, for $g \in C^1([y,x])$ we have
\[
\sum_{y < n \le x} u(n) g(n) = U(x) g(x) - \int_y^x U(z) g'(z) \, dz,
\] 
where $U(z) = \sum_{\lfloor y \rfloor + 1 \le n \le \lfloor z \rfloor} u(n)$.
\end{lem}

Let $c \in [1,2)$, $h \in \mathcal{F}_{c}$ and let $\vp$ be the inverse of $h$. Then, we define 
\begin{align} \label{def:F}
F_\la (t) & := \sum_{N_0 \le n \le \la} \vp'(n) e( nt), 
\qquad t \in \T, \quad \la \ge N_0, \\ \label{def:G}
G_\la (t) & := \sum_{n \in \Nh\cap[\lambda]} 
e(nt), 
\qquad t \in \T, \quad \la \ge N_0.
\end{align}

Proposition \ref{prop:FG} will be essential in the proof of Theorem \ref{thm:asymr}.

\begin{prop}\label{prop:FG}
Let $c \in [1,4/3)$, $\g = 1/c$, $h \in \mathcal{F}_{c}$ and $\vp$ be its inverse. Further, assume that $\chi > 0$ is such that 
$4(1-\g) + 5 \chi < 1$. Then,
\[
\| F_\la - G_\la \|_{L^{\infty}(\T)} \lesssim \vp(\la) \la^{-\chi}, \qquad \la \ge N_0.
\]
\end{prop}

Noting that $\chi < \g$ and using \eqref{slowhfi} we have
\begin{align}\label{est3}
1 \lesssim \vp(\la) \la^{-\chi}, \qquad \la \ge N_0,
\end{align}
hence the summation in
$F_N$ can be extended to all $n\in[\la]$ no matter how $\vp$ is defined
for $n\in[N_0-1]$ and Proposition \ref{prop:FG} remains valid.

For $x\in \R$ define a sawtooth function by $\Phi(x):= \{x\} - 1/2$. Expanding $\Phi$ in a Fourier series, see \cite[Section 2]{HB}, we obtain
\begin{align}\label{F2}
\Phi(x) =
\sum_{0 < |m| \le M} \frac{1}{\dpi m} e(m x) + O\left( \min\left\{ 1, \frac{1}{M\| x \|} \right\} \right),
\end{align}
for every $M \in \Z_+$, recall that
$\| x \| = \min\{ |x-n| : n \in \Z\}$ is the distance of $x \in \R$
to the nearest integer. Moreover, for any $M \in \Z_+$ we have
\begin{align}\label{F1}
\min\left\{ 1, \frac{1}{M\| x \|} \right\}
=
\sum_{m \in \Z} b_m e(mx), \qquad x \in \R,
\end{align}
where 
\begin{align}\label{F1co}
| b_m | \lesssim \min \left\{ \frac{1+\log M}{M}, \frac{M}{|m|^2} \right\}, \qquad M \in \Z_+, \quad m \in \Z.
\end{align}
Note that $b_m$ depends also on $M$, but we are not going to emphasize this fact in the sequel.

\begin{proof}[Proof of Proposition \ref{prop:FG}]
Using \eqref{condu2} and the identity $\lfloor x \rfloor = x - \{ x \}$ we may write
\begin{align*}
G_\la (t) =& O(1) + 
\sum_{h(N_0) \le n \le \la}
e(nt) \big( \lfloor - \vp (n) \rfloor - \lfloor - \vp (n+1) \rfloor  \big) \\
=&  O(1) +
\sum_{N_0 \le n \le \la} e(nt) \big( \vp (n + 1) - \vp (n)  \big)\\
&+
\sum_{N_0 \le n \le \la} e(nt) 
\big( \Phi ( - \vp (n + 1) ) - \Phi ( - \vp (n) )  \big).
\end{align*}
Further, we have 
\begin{align*}
G_\la (t) &= O(1) + F_\la (t) + I_1+ I_2,
\end{align*}
where
\begin{align*}
I_1:=&\sum_{N_0 \le n \le \la} e(nt) \big( \vp (n + 1) - \vp (n) - \vp' (n) \big),\\
I_2:=&\sum_{N_0 \le n \le \la} e(nt) 
\big( \Phi ( - \vp (n + 1) ) - \Phi ( - \vp (n) )  \big).
\end{align*}
Now, it suffices to show that
\begin{align*}
|I_1| + |I_2| \lesssim \vp(\la) \la^{-\chi}, \qquad \la \ge N_0, \quad t \in \T.
\end{align*}

We first deal with $I_1$. Using the fact that
\begin{align}\label{fibis}
-\vp'' (2x) \simeq -\vp'' (x), \qquad x \ge N_0,
\end{align}
which easily follows from \eqref{compfi} and Lemma \ref{funlemfi}, we obtain
\begin{align*}
\big| \vp (n + 1) - \vp (n) - \vp' (n) \big|
&=
\Big| \int_n^{n+1} \int_n^x \vp'' (y) \, dy \, dx \Big|\\
&\lesssim
- \int_n^{n+1} \vp'' (x) \, dx= \vp' (n) - \vp' (n + 1).
\end{align*}
Using this estimate we conclude
\[
|I_1| \lesssim 
\sum_{N_0 \le n \le \la} \big( \vp' (n) - \vp' (n + 1) \big)
= \vp' (N_0) - \vp' (\la + 1) \lesssim 1,
\]
which in view of \eqref{est3} gives us the desired estimate for $I_1$.

Next, we focus on $I_2$. Decomposing the sum defining $I_2$ into dyadic pieces we see that
\begin{align*}
I_2 = \sum_{0 \le l \le \log_2 \la} 
\sum_{\genfrac{}{}{0pt}{}{ 2^l \le n < 2^{l+1} }{ N_0 \le n \le \la }}
e(nt) \Big( \Phi \big( - \vp (n + 1) \big) - \Phi \big( - \vp (n) \big)  \Big).
\end{align*}
For $1 \le P < P' \le 2P$ and $t \in \T$, define the auxiliary functions
\[
S_{P,P'} (t)
:=
\Big|
\sum_{N_0 \vee P \le n < P'} 
e(nt) \big( \Phi ( - \vp (n + 1) ) - \Phi ( - \vp (n) )  \big)
\Big|.
\]
Now we fix $\chi' > \chi$ such that $4(1-\g) + 5 \chi' < 1$. Our aim is to show that
\begin{align}\label{estS}
S_{P,P'} (t) 
\lesssim
\vp (P) P^{-\chi'}, \qquad N_0 \le P < P' \le 2P, \quad t \in \T.
\end{align}
Let $\s$ be the function from \eqref{fiequat1} if $c=1$, or $\s \equiv 1$ if $c>1$. Further, let
\begin{align*}
M :=  \frac{P^{1 + \chi' + \eps}}{\vp (P) \s (P)},
\end{align*}
where $\eps > 0$ is fixed and such that $4(1-\g) + 5 \chi' + 7 \eps < 1$. 
Now applying \eqref{F2} with $M$ as above we infer that
\begin{align*}
S_{P,P'} (t) 
 \lesssim S_1 +S_2, \qquad N_0 \le P < P' \le 2P,
\end{align*}
where
\begin{align*}
S_1&:=\Big|
\sum_{ P \le n < P'} 
e(nt) 
\sum_{0 < |m| \le M} \frac{1}{\dpi m} 
\big( e(m \vp (n + 1)) - e(m \vp (n)) \big)
\Big|,\\
S_2&:=\sum_{ P \le n \le P'}
\min\left\{ 1, \frac{1}{M\| \vp (n) \|} \right\}.
\end{align*}

In order to show \eqref{estS} it is enough to prove that
\begin{align*}
S_1 + S_2
\lesssim
\vp (P) P^{-\chi'}, \qquad N_0 \le P < P' \le 2P, \quad t \in \T.
\end{align*}
We first analyze $S_2$. By \eqref{F1} we may write
\begin{align}
\label{eq:1}
S_2 = \sum_{ P \le n \le P'} \sum_{m \in \Z} b_m e(m \vp (n))
\le 
\sum_{m \in \Z} |b_m| 
\Big| \sum_{ P \le n \le P'} e(m \vp (n)) \Big|.
\end{align}
Let us denote
\begin{align*}
T_{P,P'} (t,m) := 
\Big| \sum_{ P \le n \le P'} e( nt +  m \vp (n) ) \Big|,
\qquad N_0 \le P \le P' \le 2P, \quad t \in \T, \quad m \in \Z,
\end{align*}
and observe that $T_{P,P'} (0,m)$ is exactly the inner exponential sum in \eqref{eq:1}. We shall prove 
\begin{align}\label{estT}
T_{P,P'} (t,m) \lesssim
\begin{cases}
	|m|^{1/2}  \frac{P}{(\vp (P) \s (P))^{1/2}}, & \quad m \ne 0,\\
	 P, & \quad m = 0,
\end{cases} 
\qquad N_0 \le P \le P' \le 2P, \quad t \in \T.
\end{align}
The case $m = 0$ is trivial so we focus on $m \ne 0$. In order to prove \eqref{estT} we use  Lemma \ref{lem:VdC}.
For $x \in [P,2P]$ let $F(x) := tx + m \vp (x)$.
Using \eqref{fibis} and Lemma \ref{funlemfi} we see that
\[
| F''(x) | = - |m| \vp'' (x) \simeq |m| \frac{\vp (P) \s (P)}{P^2},
\qquad x \in [P,2P].
\]
Therefore an application of Lemma \ref{lem:VdC} with
$I=[P,P'] \subseteq [P,2P]$, $\eta = |m| \frac{\vp (P) \s (P)}{P^2}$ and
$r=1$ leads us to
\begin{align*}
T_{P,P'} (t,m) 
\lesssim
P |m|^{1/2} \frac{ (\vp (P) \s (P))^{1/2} }{P}
+ 
|m|^{-1/2} \frac{ P }{(\vp (P) \s (P))^{1/2}}
\lesssim
|m|^{1/2} \frac{ P }{(\vp (P) \s (P))^{1/2}}.
\end{align*}
The last estimate follows from inequality
$\vp (x) \s(x) \lesssim x$ for $x \ge N_0$, which is a simple
consequence of \eqref{slowhfi} if $c > 1$ and \eqref{ratefi} if
$c = 1$. Therefore we have justified \eqref{estT}. Now combining this
estimate with \eqref{F1co} we obtain
\begin{align*}
S_2
& \lesssim
\frac{1+\log M}{M} P + 
\sum_{0 < |m| \le M} 
\frac{1+\log M}{M}  \frac{ |m|^{1/2}P }{(\vp (P) \s (P))^{1/2}}
+
\sum_{ |m| \ge M+1} \frac{M}{|m|^2}  \frac{ |m|^{1/2}P }{(\vp (P) \s (P))^{1/2}} \\
& \lesssim
\frac{1+\log M}{M} P + (1+\log M)  \frac{ M^{1/2}P }{(\vp (P) \s (P))^{1/2}}. 
\end{align*}
Since  $\s(x)^{-1} \lesssim_\delta x^\delta$ for every $\delta > 0$ (see Lemma \ref{funlemfi}), we have $\log M \lesssim \log P$ and consequently
\begin{align*}
S_2 & \lesssim
\frac{\log P + 1}{P^{\chi' + \eps}} \vp (P) \s (P) 
+ 
(\log P + 1)  
\frac{ P^{(3 + \chi' + \eps)/2} }{\vp (P) \s (P)} \\
& =
\vp (P) P^{ - \chi'} 
\Big(  \frac{\log P + 1}{P^{\eps}} \s (P) 
+ 
(\log P + 1)  
\frac{ P^{(3 + 3\chi' + \eps)/2} }{\vp (P)^2 \s (P)}
\Big).
\end{align*}
The required bound for $S_2$ follows because $\log P + 1 \lesssim P^\eps$, $\s(P) \lesssim 1$ and $(3 + 3\chi' + \eps)/2 < 2\g$.

Now we analyze $S_1$. We see that 
\begin{align} \label{estS1} 
S_1 \lesssim
\sum_{0 < |m| \le M} \frac{1}{|m|} 
\Big|
\sum_{ P \le n < P'} 
e(nt + m \vp (n) )
\big( e(m ( \vp (n + 1) - \vp (n) ))  - 1 \big)
\Big|.
\end{align} 
Let $u(n) = e(nt + m \vp (n) )$ and 
$g(z) = e(m ( \vp (z + 1) - \vp (z) ))  - 1$. Clearly $g \in C^1 ([P-1,2P])$ and using \eqref{compfi}, Lemma \ref{funlemfi} and \eqref{fibis} we obtain
\begin{align*} 
|g (z)| & \lesssim |m| \vp' (P) \simeq |m| \frac{\vp (P)}{P}, 
\qquad z \in [P-1,2P], \\
|g' (z)| & \lesssim |m| \big| \vp' (z+1) - \vp' (z) \big| 
\simeq - |m| \vp'' (P) \simeq |m| \frac{\vp (P) \s (P)}{P^2}, 
\qquad z \in [P-1,2P].
\end{align*} 
Further, using \eqref{estT} we have
\[
|U(z)|= \Big| \sum_{ P \le n \le \lfloor z \rfloor} u(n)  \Big| 
= T_{P, \lfloor z \rfloor } (t, m) 
\lesssim 
|m|^{1/2}  \frac{P}{(\vp (P) \s (P))^{1/2}},
\qquad z \in [P-1,2P].
\]
Therefore applying Lemma \ref{lem:SBP} to the inner summation in \eqref{estS1} and then using the above estimates we arrive at 
\begin{align*} 
S_1 & \lesssim
\sum_{0 < |m| \le M} \frac{1}{|m|} 
\Big(
|m|^{1/2}  \frac{P|m|}{(\vp (P) \s (P))^{1/2}} 
 \frac{\vp (P)}{P}
+
|m|^{1/2}  \frac{P|m|}{(\vp (P) \s (P))^{1/2}}
 \frac{\vp (P) \s (P)}{P^2} P
\Big) \\
& \lesssim
\sum_{0 < |m| \le M}  |m|^{1/2} \bigg( \frac{\vp (P)}{\s (P)} \bigg)^{1/2}
\lesssim M^{3/2} \bigg( \frac{\vp (P)}{\s (P)} \bigg)^{1/2}.
\end{align*} 
By $\s (P)^{-1} \lesssim P^\eps$, we further obtain
\begin{align*} 
S_1 \lesssim
 \frac{P^{(3 + 3\chi' + 3\eps)/2}}{\vp (P) \s (P)^2}
\lesssim 
\vp (P) P^{- \chi'} 
\frac{P^{(3 + 5\chi' + 7\eps)/2}}{\vp (P)^2}
\end{align*} 
and the desired estimate for $S_1$ follows since $(3 + 5\chi' + 7\eps)/2 < 2\g$. 

The proof of \eqref{estS} is completed, we will show that
$|I_2| \lesssim \vp(\la) \la^{-\chi}$. Let
$\delta < (\chi' - \chi)/2$. Using \eqref{estS} (with $P=2^l$,
$P'= 2^{l+1} \wedge (\la+1)$), \eqref{slowhfi} and \eqref{ratefi} we
conclude
\begin{align*} 
|I_2| & \lesssim
1 + \sum_{\log_2 N_0 \le l \le \log_2 \la} \vp(2^l) 2^{-l \chi'}
\lesssim
\sum_{0 \le l \le \log_2 \la} 2^{l ( \g - \chi' + \delta )}\\
&\lesssim 
\la^{\g - \chi' + \delta}
\lesssim 
\vp (\la) \la^{ - \chi' + 2\delta}  \le
\vp (\la) \la^{ - \chi }.
\end{align*} 
This yields the required estimate for $I_2$ and completes the proof of Proposition \ref{prop:FG}.
\end{proof}

%%%%%%%%%%%%%%%%%%%%%%%%%%%%%%%%%%%%%%%%%%%%%%%%%%%%%%%%%%%%%%%%%%%%%%%%%%%%%%%%%%%%%
\subsection{Proof of Theorem \ref{thm:asymr}}
%%%%%%%%%%%%%%%%%%%%%%%%%%%%%%%%%%%%%%%%%%%%%%%%%%%%%%%%%%%%%%%%%%%%%%%%%%%%%%%%%%%%%

Here we will proceed as in the proof of \cite[Theorem 1, p. 46]{BF}
or \cite[Theorem 1.6, p. 30]{Piat}.  For $k\in[3]$, let $F_\la^k$ and
$G_\la^k$ be the functions defined in \eqref{def:F} and \eqref{def:G}
respectively, that correspond to $h_k$.
We first observe that
\begin{align*} 
r_{h_1, h_2, h_3}(\la) 
= \sum_{\genfrac{}{}{0pt}{}{ n_1, n_2, n_3 \le \la }{ n_k \in \N_{h_k} }} 
\int_0^1 e(t (n_1 + n_2 + n_3 - \la)) \, dt 
=
\int_0^1 G_\la^1 (t) G_\la^2 (t) G_\la^3 (t) e(- t \la) \, dt.
\end{align*} 
Proceeding in a similar way we obtain 
\begin{align*} 
\int_0^1 F_\la^1 (t) F_\la^2 (t) F_\la^3 (t) e(- t \la) \, dt
= \sum_{\genfrac{}{}{0pt}{}{ n_1 + n_2 + n_3 = \la }{ n_1, n_2, n_3 \ge N_0 }}
\vp'_1 (n_1) \vp'_2 (n_2) \vp'_3 (n_3)
= J_{\vp'_1, \vp'_2, \vp'_3} (\la).
\end{align*} 
Therefore, using Proposition \ref{prop:J3} and \eqref{fiequat} we reduce our problem to checking that 
\begin{align} \label{red3}
\int_0^1 
\big| G_\la^1 (t) G_\la^2 (t) G_\la^3 (t) - F_\la^1 (t) F_\la^2 (t) F_\la^3 (t) \big| \, dt
=
o \left( \frac{ \vp_1(\la) \vp_2(\la) \vp_3(\la) }{\la} \right).
\end{align} 
Using  H\"older's inequality we see that the left-hand side of \eqref{red3} is controlled by 
\begin{align}
\label{eq:2}
\begin{gathered}
\big\| F_\la^1 - G_\la^1 
\big\|_{L^{\infty}(\T)} \big\| F_\la^2 \big\|_{L^{2}(\T)}
\big\| F_\la^3 \big\|_{L^{2}(\T)}  
+
\big\| G_\la^1 \big\|_{L^{2}(\T)}
\big\| F_\la^2 - G_N^2 \big\|_{L^{\infty}(\T)}
\big\| F_\la^3 \big\|_{L^{2}(\T)}\\
+
\big\| G_\la^1 \big\|_{L^{2}(\T)} 
\big\| G_\la^2 \big\|_{L^{2}(\T)} 
\big\| F_\la^3 - G_\la^3 \big\|_{L^{\infty}(\T)}.
\end{gathered}
\end{align}
By the Plancherel theorem and the fact that $\vp'_k (n)\lesssim 1$ and
$\vp'_k (n) \simeq \vp_k (n+1) - \vp_k (n)$ for $n \ge N_0$ and
$k\in[3]$ (see \eqref{compfi}), we obtain
\begin{align}\label{estF}
\big\| F_\la^k \big\|_{L^{2}(\T)}^2
=
\sum_{N_0 \le n \le \la} \vp'_k (n)^2 
\lesssim 
\vp_k (\la+1) - \vp_k (N_0) 
\lesssim
\vp_k (\la), \qquad \la \ge N_0,
\end{align}
and
\begin{align}\label{estG}
\big\| G_\la^k \big\|_{L^{2}(\T)}^2
=
\sum_{n \in \N_{h_k}\cap[\la] } 1 
\le \vp_k (\la+1)
\simeq
\vp_k (\la), \qquad \la \ge N_0.
\end{align}
Let 
\[
\chi_k: = 
\sum_{\genfrac{}{}{0pt}{}{ j=1 }{ j \ne k }}^3 (1-\g_j)/2 + \eps,
\qquad k \in [3].
\]
Using the assumption \eqref{assum} we can choose $\eps > 0$ such that 
$4(1-\g_k) + 5 \chi_k < 1$ for $k \in [3]$. By Proposition \ref{prop:FG} we have
\[
\big\| F_\la^k - G_\la^k \big\|_{L^{\infty}(\T)} \lesssim \vp_k (\la) \la^{-\chi_k}, \qquad \la \ge N_0.
\]
Combining this with \eqref{estF} and \eqref{estG} we see that \eqref{eq:2} is bounded by
\begin{align*}
& \vp_1 (\la) \la^{-\chi_1} (\vp_2 (\la) \vp_3 (\la))^{1/2}
+
\vp_2 (\la) \la^{-\chi_2} (\vp_1 (\la) \vp_3 (\la))^{1/2}
+
\vp_3 (\la) \la^{-\chi_3} (\vp_1 (\la) \vp_2 (\la))^{1/2} \\
& \qquad =
\frac{\vp_1 (\la) \vp_2 (\la) \vp_3 (\la)}{\la^{1 + \eps/2}}
\bigg(
\frac{\la^{1 - \chi_1 + \eps/2}}{(\vp_2 (\la) \vp_3 (\la))^{1/2}}
+
\frac{\la^{1 - \chi_2 + \eps/2}}{(\vp_1 (\la) \vp_3 (\la))^{1/2}}
+
\frac{\la^{1 - \chi_3 + \eps/2}}{(\vp_1 (\la) \vp_2 (\la))^{1/2}} 
\bigg) \\
& \qquad \lesssim
\frac{\vp_1 (\la) \vp_2 (\la) \vp_3 (\la)}{\la^{1 + \eps/2}}.
\end{align*}
The last estimate above is a straightforward consequence of \eqref{ratefi} and the fact that
\[
1 - \chi_k + \eps/2 < 
\sum_{\genfrac{}{}{0pt}{}{ j=1 }{ j \ne k }}^3 \g_j /2,
\qquad k \in [3].
\] 
The estimate \eqref{red3} follows and the proof of Theorem \ref{thm:asymr} is completed.
\qed

\medskip

We finally establish the asymptotic formula for the number of lattice points in the arithmetic spheres $\mathbf S^3_{c}(\lambda)$ from \eqref{eq:39}.

\begin{proof}[Proof of Corollary \ref{cor:asym}]
Let $r_{c, \Z_+}(\la):= \# \{x \in \Z_+^3 : \lfloor x_1^c \rfloor + \lfloor x_2^c \rfloor + \lfloor x_3^c \rfloor = \lambda \}$ for any $\lambda \in \Z_+$, and let $\g = 1/c$. 
Since  $r_c(\lambda) = 2^3 r_{c,\Z_+}(\la) + O(\lambda^{\g})$, it is enough to prove that for every $\eps > 0$ one has
\begin{align*}  
r_{c,\Z_+}(\la) = \frac{\g^{3} \Gamma(\g)^3 }{\Gamma(3\g)}
 \lambda^{3\g - 1}
+ O_{\eps} \big(  \lambda^{3\g - 1 - (9 - 8c)/(5c) + \eps} \big).
\end{align*}  
This, however, follows by a careful repetition of the arguments from the proof of Theorem~\ref{thm:asymr}.

We  provide some details. Let $F_N$ and $G_N$ be the functions defined in \eqref{def:F} and \eqref{def:G} respectively, that correspond to $h(x) = x^c$ and $N_0=1$.
Let $\vp(x) = x^\g$ and observe that  
\begin{align*}  
r_{c,\Z_+}(\la) = \int_0^1 G_\lambda(t)^3 e^{-\dpi t \lambda} \, dt, \qquad
J_{\vp',\vp',\vp'} (\lambda) = \int_0^1 F_\lambda(t)^3 e^{-\dpi t \lambda} \, dt, \qquad \lambda \in \Z_+.
\end{align*} 
We take  sufficiently small $\eps>0$ and apply Proposition~\ref{prop:FG} with $\chi = \frac{1-4(1-\g)}{5} - \eps > 0$ (notice that $4(1-\g) + 5 \chi < 1$) together with the estimate 
$\| F_\lambda \|_{L^{2}(\T)}^2 + \| G_\lambda \|_{L^{2}(\T)}^2 \lesssim \lambda^{\g}$ for $\lambda \in \Z_+$, and conclude that
\begin{align*}  
| r_{c,\Z_+}(\la) - J_{\vp',\vp',\vp'} (\lambda) | 
\lesssim_\eps \lambda^{3\g - 1 - (9\g - 8)/5 + \eps},  
\qquad \lambda \in \Z_+.
\end{align*} 
Now the claim immediately follows  from Corollary~\ref{cor:J3}.
\end{proof}

%%%%%%%%%%%%%%%%%%%%%%%%%%%%%%%%%%%%%%%%%%%%%%%%%%%%%%%%%%%%%%%%%%%%%%%%%%%%%%%%%%%
\section{Proof of Theorem \ref{thm:L2}. $\ell^p(\mathbb{Z}^3)$ bounds for maximal functions}
\label{sec:el2}
%%%%%%%%%%%%%%%%%%%%%%%%%%%%%%%%%%%%%%%%%%%%%%%%%%%%%%%%%%%%%%%%%%%%%%%%%%%%%%%%%%%

%%%%%%%%%%%%%%%%%%%%%%%%%%%%%%%%%%%%%%%%%%%%%%%
\subsection{Preliminary reductions} \label{sec:41}

Throughout this section $c \in (1, 11/10)$ is fixed and we define
\begin{align}
\label{eq:46}
\kappa:=\kappa_c:=\frac{3-4c}{4c}=\frac{3}{4c}-1.
\end{align}
We also fix a $\lambda_0$-lacunary set $\mathbb D:=\{\lambda_n: n\in\Z_+\}\subset \Z_+$ with $\lambda_0=\inf_{n\in\Z_+}\frac{\lambda_{n+1}}{\lambda_n}>1$.

Let
$\eta \in C_c^{\8}(\R^3)$ be such that $0\le \eta(x)\le 1$ for all $x\in\R^3$,  
$\supp \eta \subseteq [-10, 10]^3$
and $\eta(x) = 1$ whenever $|x|_c^c \le 4$.  We also assume that
$\eta$ is of product type, i.e.
$\eta(x) := \prod_{j=1}^3\eta_j (x_j)$ for some even functions
$\eta_j \in C_c^{\8}(\R)$ with $j \in [3]$. We employ the ideas from
the circle method and for every $\la\in\Z_+$  we split the unit interval $[-1/2, 1/2]$ into  major 
$\mathfrak M_{\lambda}$ and  minor  $\mathfrak m_{\lambda}$ arcs, which are defined respectively by
\begin{align*}  
\mathfrak M_{\lambda} & := 
\Big(-\frac{\lambda^{\kappa}}{2c} , \frac{\lambda^{\kappa}}{2c}  \Big), \\
\mathfrak m_{\lambda} & := [-1/2, 1/2] \setminus \mathfrak M_{\lambda}.
\end{align*} 
Take a smooth partition of unity $\psi (t) + \widetilde{\psi} (t) = 1$ for any $t \in \R$, where $\psi$ is supported in $(-1/2, 1/2)$ and $\psi (t) = 1$ for 
$t \in (-1/(2c), 1/(2c) )$. We consider
\begin{align*}  
\psi_\lambda (t) = \psi \Big( \frac{t}{\lambda^{\kappa}} \Big), \qquad
\widetilde{\psi}_\lambda (t) = \widetilde{\psi} \Big( \frac{t}{\lambda^{\kappa}} \Big), 
\qquad t \in \R, \quad \lambda \ge 1. 
\end{align*} 
It is easy to see that $\widetilde{\psi}_\lambda (t) \le \mathbbm{1}_{\mathfrak m_{\lambda}} (t)$ for any $t \in [-1/2, 1/2]$.
Simple computations show that
\begin{align}
\label{iden31}
\sigma_\lambda (x) =
\eta \Big( \frac{x}{\lambda^{1/c}} \Big) 
\int_{-1/2}^{1/2} e((Q (x) - \lambda) t) 
\big( \psi_\lambda (t) +  \widetilde{\psi}_\lambda (t) \big) \, dt 
=\sigma_{\lambda}^{\mathfrak M} (x) + \sigma_{\lambda}^{\mathfrak m} (x),
\end{align}
where 
\begin{align}
\label{eq:18}
\sigma_{\lambda}^{\mathfrak M} (x)&:=
\lambda^{\kappa} \eta \Big( \frac{x}{\lambda^{1/c}} \Big) 
\mathcal{F}^{-1}_{\R} \psi \big( \lambda^{\kappa} (Q (x) - \lambda) \big),\\
\label{eq:19}
\sigma_{\lambda}^{\mathfrak m} (x)&:=
\eta \Big( \frac{x}{\lambda^{1/c}} \Big)
\int_{-1/2}^{1/2} e((Q(x) - \lambda) t) 
 \widetilde{\psi}_\lambda (t)  \, dt
\end{align}
and
$$
Q(x)=\lfloor |x_1|^c \rfloor + \lfloor |x_2|^c \rfloor + \lfloor |x_3|^c \rfloor,\qquad x\in \R^3.
$$
Consequently we obtain
\begin{align}  \label{est41}
\sup_{\lambda \in\Z_+}\frac{1}{r_c(\lambda)}|\sigma_{\lambda} \ast f (x)| & \lesssim
\sup_{\lambda \in\Z_+} \frac{1}{\lambda^{3/c - 1}} | \sigma_{\lambda}^{\mathfrak M} \ast f (x)|
+ 
\sum_{n=0}^{\8} \sup_{2^n \le \lambda < 2^{n+1}} 
\frac{1}{\lambda^{3/c - 1}} |\sigma_{\lambda}^{\mathfrak m} \ast f (x)|,
\end{align}
and
\begin{align}
\label{est42}
\sup_{n \in\Z_+}\frac{1}{r_c(\lambda_n)}|\sigma_{\lambda_n} \ast f (x)|
& \lesssim
\sup_{n \in\Z_+}\frac{1}{\lambda_n^{3/c - 1}} | \sigma_{\lambda_n}^{\mathfrak M} \ast f (x)|
+ 
\sum_{n=0}^{\8} 
\frac{1}{\lambda_n^{3/c - 1}} |\sigma_{\lambda_n}^{\mathfrak m} \ast f (x)|.
\end{align}

%%%%%%%%%%%%%%%%%%%%%%%%%%%%%%%%%%%%%%%%%%%%%%%
\subsection{Minor arc estimate}

In this section we show the following result.
\begin{thm} \label{thm:minor}
Let $c \in (1, 11/10)$. Then for every  $N \ge 1$ and  $f \in \ell^2(\Z^3)$ we have
\begin{align*}
\big\| \sup_{N \le \lambda \le 2N } \lambda^{-3/c + 1}
|\sigma_{\lambda}^{\mathfrak m} \ast f| \big\|_{\ell^2(\Z^3)}
\lesssim
N^{-(11 - 10c)/(6c)} \log^2 (N + 1) 
\|f \|_{\ell^2(\Z^3)}.
\end{align*}
\end{thm}

To prove Theorem~\ref{thm:minor} we need some preparatory results.
We note that
\begin{align} \label{F4}
e(-x \{ y \}) 
= \sum_{|m| \le M} c_m(x) e(my) + O\Big(\min \Big\{1, \frac{1}{M\|y\|} \Big\} \Big).
\end{align}
This expansion is uniform in $0 < |x| \le 1/2$, $M \in \Z_+$ and $y \in \R$. Moreover,  one has
\begin{align*}
c_m (x) = 
\begin{cases}
	\frac{1-e(- x) }{\dpi (x+m)}, & \quad x+m \ne 0,\\
	 1, & \quad x+m = 0.
\end{cases} 
\end{align*}

\begin{rem}
Note that \eqref{F4} is not true uniformly for all $x \in \R$. To see this, take $y=0$, $x=M+1/2$  and let $M \to \8$. 
\end{rem}

For a fixed $c \in (1,2)$ let us define
\begin{align} \label{def:T}
U_{P,P'} (t,\xi) := \Big| \sum_{P \le n \le P'} e(n^c t + n\xi) \Big|,
\qquad 1 \le P \le P' \le 2P, \quad t \in \R, \quad \xi \in \T,
\end{align}
and
\begin{align} \label{def:S}
V_{P,P'} (M) := \sum_{P \le n \le P'} \min \Big\{1, \frac{1}{M\|n^c\|} \Big\},
\qquad M \in \Z_+, \quad 1 \le P \le P' \le 2P.
\end{align}

\begin{lem} \label{lem:F6}
Let $c \in (1,2)$ be fixed. Then for every $1 \le P \le P' \le 2P$, $t \ne 0$ and  $\xi \in \T$ one has
\begin{align*} 
U_{P,P'} (t,\xi)
\lesssim 
P^{c/2} |t|^{1/2} + P^{1-c/2} |t|^{-1/2}.
\end{align*}
\end{lem}

\begin{proof}
 It is enough to apply Lemma~\ref{lem:VdC} (with $F(x) = x^c t + x \xi$, $\eta = P^{c-2} |t|$, $r=1$).
\end{proof}

\begin{lem} \label{lem:F5}
Let $c \in (1,2)$ be fixed. Then for every $M \in \Z_+$ and  $1 \le P \le P' \le 2P$ one has
\begin{align*} 
 V_{P,P'} (M)
\lesssim 
(1+\log M )\big(M^{-1} P + P^{c/2} M^{1/2}\big).
\end{align*}

\end{lem}

\begin{proof}

Using \eqref{F1} we see that
\begin{align*} 
V_{P,P'} (M) 
= 
\sum_{P \le n \le P'} \sum_{m \in \Z} b_m e(m n^c)
\le 
\sum_{m \in \Z} |b_m|U_{P,P'}(m,0).
\end{align*}
By  Lemma~\ref{lem:VdC} (with $F(x) = m x^c$, $\eta = |m| P^{c-2}$, $r=1$) we deduce
\begin{align*} 
U_{P,P'}(m,0)
\lesssim
\begin{cases}
	P, & \quad m = 0,\\
	|m|^{1/2} P^{c/2}, & \quad m \ne 0.
\end{cases} 
\end{align*}
Consequently, combining this with \eqref{F1co} we obtain
\begin{align*} 
V_{P,P'} (M) 
& \lesssim
\frac{1+\log M}{M} \Big( P + \sum_{|m| \le M}  |m|^{1/2} P^{c/2} \Big)
+ \sum_{|m| \ge M + 1} M |m|^{-3/2} P^{c/2} \\
& \lesssim
\frac{1+\log M}{M} P + (1+\log M) P^{c/2} M^{1/2}.
\end{align*}
This finishes the proof of Lemma~\ref{lem:F5}.
\end{proof}

For $c \in (1,2)$ and $g \in C_c^{\8} (\R)$ we introduce
\begin{equation}\label{eq:Pts}  
\Pi_{t,s}^{g} (\xi) := \sum_{n \in \Z} e(\lfloor |n|^c \rfloor t + n \xi) 
g \Big( \frac{n}{s^{1/c}} \Big), 
\qquad s\geq 1, \quad  t\in\R, \quad \xi \in \T,
\end{equation}
and the quantity 
\[
N_c := (2c)^{-1} (2 N)^{\kappa}.
\]

\begin{lem} \label{lem:L2}
Let $c \in (1,2)$ and $g \in C_c^{\8} (\R)$ be fixed. Then for all $N \ge 1$ one has
\begin{align*} 
\sup_{1\le s \le 2 N} \sup_{\xi \in \T} 
\sup_{N_c \le |t| \le 1/2}
|\Pi_{t,s}^{g} (\xi) |
\lesssim
N^{1/3 + 1/(3c)} \log (N + 1).
\end{align*} 
\end{lem}

\begin{proof}
Observe that there exists $K_0 > 0$ such that
$\supp g \subset [- K_0, K_0]$ and consequently
$\supp g \big( \frac{\cdot}{s^{1/c}} \big) \subset [- K_0 s^{1/c}, K_0 s^{1/c}]$.
Thus the summation in \eqref{eq:Pts} can be restricted to the set $[-K_0 s^{1/c}, K_0 s^{1/c}]\cap\Z$.
Then summation by parts, see Lemma \ref{lem:SBP}, reduces the proof of Lemma \ref{lem:L2} to establishing
that for any fixed $K>0$ and for every $N \ge 1$ we have
\begin{align} \label{red4}
\sup_{1\le s \le  KN}\sup_{1\le Z\le K_0s^{1/c}} \sup_{\xi \in \T} 
\sup_{N_c \le |t| \le 1/2}
\Big| \sum_{ n \in [Z] } e(\lfloor n^c \rfloor t + n \xi)
 \Big|
\lesssim
N^{1/3 + 1/(3c)} \log (N + 1).
\end{align}
Appealing to \eqref{F4} with 
$M = \lfloor N^{2(1/c-1/2)/3} \rfloor \ge 1$ we may write
\begin{align*}
\sum_{n \in [Z] } e(\lfloor n^c \rfloor t + n \xi)
=
\sum_{n \in [Z] } e( n^c t + n \xi) e(-\{ n^c \} t)=I_1(t, \xi, Z)+I_2(t, \xi, Z),
\end{align*}
where
\begin{align*}
I_1(t, \xi, Z)&:=\sum_{|m| \le M} c_m (t)\sum_{n \in [Z]} e( n^c (t+m) + n \xi),\\
I_2(t, \xi, Z)&:=O\Big( \sum_{n \in [Z]} \min \Big\{1, \frac{1}{M\|n^c\|} \Big\} \Big).
\end{align*}
We first deal with $I_2(t, \xi, Z)$. Let $V_{P,P'} (M)$ be defined as in
\eqref{def:S}. By Lemma~\ref{lem:F5} we conclude
\begin{align*}
I_2(t, \xi, Z)
& \lesssim
\sum_{0 \le l \le \log_2 Z} \sum_{2^l \le n < 2^{l+1} } 
\min \Big\{1, \frac{1}{M\|n^c\|} \Big\} 
=
\sum_{0 \le l \le \log_2 Z} 
V_{2^l,2^{l+1} - 1} (M) \\
& \lesssim
\sum_{0 \le l \le \log_2 Z} (1+\log M)
\big( 2^l M^{-1}  +  2^{lc/2} M^{1/2}\big) \\
& \lesssim
(1+\log M)\big( N^{1/c}M^{-1} +  N^{1/2} M^{1/2}\big) \\
& \lesssim
N^{1/3 + 1/(3c)} \log (N + 1),
\end{align*}
as desired.

Now we focus on $I_1(t, \xi, Z)$. Notice that
\begin{align*}
\sup_{N_c \le |t| \le 1/2} | I_1(t, \xi, Z) |
& \le \sup_{N_c \le |t| \le 1/2}
\sum_{|m| \le M} | c_m (t) |
\Big|
\sum_{n \in [Z]} e( n^c (t+m) + n \xi)
\Big| \\
& \le 
\sup_{N_c \le |t| \le M + 1/2}
\Big|
\sum_{n \in [Z] } e( n^c t + n \xi)
\Big|
\sup_{ |t| \le 1/2}
\sum_{|m| \le M} | c_m (t) |.
\end{align*}
Since we have 
\begin{align*}
\sum_{|m| \le M} | c_m (t) |
\lesssim 
1 + \sum_{0 < |m| \le M} \frac{1}{|m|} 
\lesssim 
\log M, \qquad |t| \le 1/2,
\end{align*}
an application of Lemma~\ref{lem:F6} gives us
\begin{align*}
\sup_{N_c \le |t| \le 1/2} | I_1(t, \xi, Z) | 
& \lesssim
\log M \sup_{N_c \le |t| \le M + 1/2}
\sum_{0 \le l \le \log_2 Z} 
\Big|
\sum_{2^l \le n \le (2^{l+1} - 1) \wedge Z } e( n^c t + n \xi)
\Big| \\
& =
\log M \sup_{N_c \le |t| \le M + 1/2}
\sum_{0 \le l \le \log_2 Z} U_{2^l, (2^{l+1} - 1) \wedge Z} (t, \xi) \\
& \lesssim
\log M \sup_{N_c \le |t| \le M + 1/2}
\sum_{0 \le l \le \log_2 Z} \Big( 2^{lc/2} |t|^{1/2} + 2^{l(1-c/2)} |t|^{-1/2} \Big) \\
& \lesssim
\log M 
\Big( N^{1/2} M^{1/2} + N^{1/c - 1/2} N_c^{-1/2} \Big) \\
& \lesssim
N^{1/3 + 1/(3c)} \log (N + 1),
\end{align*}
as claimed. This finishes the proof of Lemma~\ref{lem:L2}.
\end{proof}

Now we are ready to prove Theorem~\ref{thm:minor}.

\begin{proof}[Proof of Theorem~\ref{thm:minor}]
Let $\Upsilon (x) = x \cdot \nabla \eta (x)$, $x \in \R^3$, and observe that
\begin{align*}
\eta \Big( \frac{x}{\lambda^{1/c}} \Big)
=
\eta \Big( \frac{x}{N^{1/c}} \Big)
+
\int_{N}^{\lambda} \frac{d}{ds} 
\Big( \eta \Big( \frac{x}{s^{1/c}} \Big) \Big) \, ds
=
\eta \Big( \frac{x}{N^{1/c}} \Big) 
- c^{-1} \int_{N}^{\lambda} \Upsilon \Big( \frac{x}{s^{1/c}} \Big) \frac{ds}{s}.
\end{align*}
Notice that $\Upsilon \in C_c^{\8} (\R^3)$ and that $\Upsilon$ is a sum of three functions of product type. Now for a fixed $g \in C_c^{\8} (\R^3)$ let us define the auxiliary function %$S_{t,s}^g $ as follows
\begin{align*}
\Delta_{t,s}^g (x)
=
e(Q (x) t) 
g \Big( \frac{x}{s^{1/c}} \Big), \qquad
t \in \R, \quad s \ge 1, \quad x \in \R^3.
\end{align*}
Elementary computations show
\begin{align*}  
\sigma_{\lambda}^{\mathfrak m} \ast f (x) 
& = 
\sum_{y \in \Z^3}
\eta \Big( \frac{y}{\lambda^{1/c}} \Big)
\int_{-1/2}^{1/2} e((Q (y) - \lambda)t) 
 \widetilde{\psi}_\lambda (t)  \, dt \, f (x-y) \\
& =
\int_{-1/2}^{1/2} \sum_{y \in \Z^3} e(Q (y) t)
\eta \Big( \frac{y}{N^{1/c}} \Big) f (x-y)
e(-\lambda t) 
 \widetilde{\psi}_\lambda (t)  \, dt \\
& \quad 
- c^{-1} \int_{-1/2}^{1/2} \sum_{y \in \Z^3} e(Q (y) t)
\int_{N}^{\lambda} \Upsilon \Big( \frac{y}{s^{1/c}} \Big) \frac{ds}{s}
f(x-y) e(-\lambda t) 
 \widetilde{\psi}_\lambda (t)  \, dt \\
& = 
\int_{-1/2}^{1/2} \Delta_{t,N}^{\eta} \ast f (x) e(-\lambda t) 
 \widetilde{\psi}_\lambda (t)  \, dt
- c^{-1} \int_{-1/2}^{1/2} \int_{N}^{\lambda} \Delta_{t,s}^{\Upsilon} \ast f (x) \frac{ds}{s} e(-\lambda t) 
 \widetilde{\psi}_\lambda (t)  \, dt.
\end{align*} 
Since $N \le \lambda \le 2N$, we have $(2c)^{-1} \lambda^{\kappa} \ge N_c$ and 
$\widetilde{\psi}_\lambda (t) \le \mathbbm{1}_{\mathfrak m_\lambda} (t) 
\le \mathbbm{1}_{N_c \le |t| \le 1/2} $ for $t \in [-1/2, 1/2]$. Consequently, we get 
\begin{align*}  
\sup_{ N \le \lambda \le 2N } |\sigma_{\lambda}^{\mathfrak m} \ast f (x) |
\le 
\int_{ N_c \le |t| \le 1/2 }
| \Delta_{t,N}^{\eta} \ast f (x) |  \, dt
+ \int_{ N_c \le |t| \le 1/2 } \int_{N}^{2 N} |\Delta_{t,s}^{\Upsilon} \ast f (x) | \frac{ds}{s} \, dt.
\end{align*} 
Now using Minkowski's and H\"older's inequalities we obtain
\begin{align}
\label{est12}
\begin{split}
\big\| \sup_{ N \le \lambda \le 2N } |\sigma_{\lambda}^{\mathfrak m} \ast f | \big\|_{\ell^2(\Z^3)}
& \le 
\int_{ N_c \le |t| \le 1/2 }
\| \Delta_{t,N}^{\eta} \ast f \|_{\ell^2(\Z^3)}  \, dt\\
&\quad + \int_{ N_c \le |t| \le 1/2 } \int_{N}^{2 N} 
\|\Delta_{t,s}^{\Upsilon} \ast f \|_{\ell^2(\Z^3)} \frac{ds}{s} \, dt \\ 
& \lesssim
\bigg(
\int_{ N_c \le |t| \le 1/2 }
\| \Delta_{t,N}^{\eta} \ast f \|^2_{\ell^2(\Z^3)}  \, dt
\bigg)^{1/2} \\
& \quad +
\bigg(
\int_{ N_c \le |t| \le 1/2 } \int_{N}^{2 N} 
\|\Delta_{t,s}^{\Upsilon} \ast f \|^2_{\ell^2(\Z^3)} \frac{ds}{s} \, dt
\bigg)^{1/2}.
\end{split}
\end{align}

Let $g \in C_c^{\8} (\R^3)$ be a fixed function of product type, i.e. $g(x) := \prod_{j=1}^3g_j(x_j)$.
Then
\begin{align*}  
\mathcal F_{\Z^3}^{-1}\Delta_{t,s}^{g} (\xi) = \prod_{j=1}^3 \Pi_{t,s}^{g_j} (\xi_j), 
\qquad s \ge 1, \quad t \in \R, \quad \xi \in \T^3,
\end{align*}
where $\Pi_{t,s}^{g_j} (\xi_j)$ is given as in \eqref{eq:Pts} for any $j\in[3]$ .
By Plancherel's theorem applied twice and Lemma~\ref{lem:L2} we obtain
\begin{align*}  
& \int_{ N_c \le |t| \le 1/2 } 
\|\Delta_{t,s}^{g} \ast f \|^2_{\ell^2(\Z^3)} \, dt =
\int_{ N_c \le |t| \le 1/2 } 
\int_{ \T^3 } |\mathcal F_{\Z^3}^{-1}\Delta_{t,s}^{g} (\xi)|^2 |\mathcal F_{\Z^3}^{-1}f (\xi)|^2 \, d\xi \, dt \\
& \qquad \le 
\sup_{N \le s \le 2 N} \sup_{\xi_1, \xi_2 \in \T} 
\sup_{N_c \le |t| \le 1/2}
\prod_{j=1}^2|\Pi_{t,s}^{g_j} (\xi_j)|^2
\int_{ \T^3 } \int_{ |t| \le 1/2 } |\Pi_{t,s}^{g_3} (\xi_3)|^2 \, dt
|\mathcal F_{\Z^3}^{-1}f (\xi)|^2 \, d\xi \\
& \qquad \lesssim
N^{4(1/3 + 1/(3c))} \log^4 (N + 1) N^{1/c} 
\| f \|^2_{\ell^2(\Z^3)}.
\end{align*}
Consequently, taking $g=\eta$ or $g=\Upsilon$ and plugging this estimate into \eqref{est12} we conclude
\begin{align*}  
\big\| \sup_{ N \le \lambda \le 2N } |\sigma_{\lambda}^{\mathfrak m} \ast f | \big\|_{\ell^2(\Z^3)}
\lesssim
N^{2/3 + 7/(6c) } \log^2 (N + 1) 
\| f \|_{\ell^2(\Z^3)}, \qquad N \ge 1.
\end{align*}
The proof of Theorem~\ref{thm:minor} is finished.
\end{proof}

\begin{cor} \label{cor:maxminor}
Let $c \in (1, 11/10)$ be fixed. Then for every $1 < p \le 2$ there exists $\delta > 0$ such that
\begin{align*}
\big\| \lambda^{-3/c + 1}
\sigma_{\lambda}^{\mathfrak m} \ast f \big\|_{\ell^p(\Z^3)}
\lesssim_p
\lambda^{-\delta} \|f \|_{\ell^p(\Z^3)},
\qquad \lambda \ge 1, \quad f \in \ell^p(\Z^3).
\end{align*}
Moreover, for every $(11-4c)/(11-7c) < p \le 2$ there exists $\delta > 0$ such that
\begin{align*}
\big\| \sup_{N \le \lambda \le 2N } \lambda^{-3/c + 1}
|\sigma_{\lambda}^{\mathfrak m} \ast f| \big\|_{\ell^p(\Z^3)}
\lesssim_p
N^{-\delta} \|f \|_{\ell^p(\Z^3)},
\qquad N \ge 1, \quad f \in \ell^p(\Z^3).
\end{align*}
\end{cor}

\begin{proof}
To prove Corollary~\ref{cor:maxminor} it is enough to interpolate
$\ell^2(\Z^3)$ estimate from Theorem~\ref{thm:minor} with the trivial
bounds
\begin{align*}
\big\|  \lambda^{-3/c + 1}
\sigma_{\lambda}^{\mathfrak m} \ast f \big\|_{\ell^1(\Z^3)}
\lesssim
\|f \|_{\ell^1(\Z^3)},
\qquad \lambda \ge 1, \quad f \in \ell^1(\Z^3), \\
\big\| \sup_{N \le \lambda \le 2N } \lambda^{-3/c + 1}
|\sigma_{\lambda}^{\mathfrak m} \ast f| \big\|_{\ell^1(\Z^3)}
\lesssim
N \|f \|_{\ell^1(\Z^3)},
\qquad N \ge 1, \quad f \in \ell^1(\Z^3),
\end{align*}
which follow from the fact that $\frac{1}{r_c(\lambda)}\sigma_{\lambda}*f$ is an averaging operator and
\begin{align}
\label{eq:3}
\big\|  \lambda^{-3/c + 1}
\sigma_{\lambda}^{\mathfrak M}*f \big\|_{\ell^1(\Z^3)}
\lesssim
\|f \|_{\ell^1(\Z^3)}, \qquad \lambda \ge 1.
\end{align}
In order to prove \eqref{eq:3} we invoke  \eqref{eq:8}, Lemma \ref{norm} and
\eqref{eq:9}, which are proved in the subsection below.
\end{proof}

%%%%%%%%%%%%%%%%%%%%%%%%%%%%%%%%%%%%%%%%%%%%%%%
\subsection{Major arc estimate} \label{subsec:maj}
We now estimate the maximal functions corresponding  to the major arc. Our main result of this subsection is the following maximal theorem. 
\begin{thm}
\label{thm:major}
Let $c\in(1, 2)$ be fixed. Then the following two maximal inequalities hold
\begin{align}
\label{eq:6}
\big\|\sup_{\lambda \in\Z_+}\lambda^{-3/c +1} | \sigma_{\lambda}^{\mathfrak M} \ast f|\big\|_{\ell^p(\Z^3)}
\lesssim \|f\|_{\ell^p(\Z^3)}, \qquad f\in\ell^p(\Z^3), \qquad 3/2<p\le \infty,
\end{align}
\begin{align}
\label{eq:7}
\big\|\sup_{n \in\Z_+}\lambda_n^{-3/c + 1} | \sigma_{\lambda_n}^{\mathfrak M} \ast f|\big\|_{\ell^p(\Z^3)}
\lesssim \|f\|_{\ell^p(\Z^3)}, \qquad f\in\ell^p(\Z^3), \qquad 1<p\le \infty.
\end{align}
\end{thm}
Since $\psi \in C_c^{\8} (\R)$ thus for any $x \in \R^3$ and $\lambda \ge 1$ we may write, see \eqref{eq:18},
\begin{align*}  
| \sigma_{\lambda}^{\mathfrak M} (x) | 
 \lesssim
\lambda^{\kappa}  
\frac{1}{1 + \big( \lambda^{\kappa} |Q(x) - \lambda| \big)^{10}} 
\lesssim \lambda^{\kappa} \omega_{\lambda}(x),
\end{align*}
where
\begin{align} \label{def:omega}
\omega_{\lambda}(x):=
\frac{1}{1 + \big( \lambda^{\kappa} | |x|_c^c - \lambda| \big)^{10}},
\qquad x \in \R^3, \quad \lambda \ge 1.
\end{align}

\begin{lem}\label{kernel}
Let $c \ge 1$ and $a>0$ be fixed. Then, 
\begin{align}\label{ker1}
    \omega_{\lambda}(t) \simeq \omega_{\lambda}(t+\g),
\end{align}
uniformly in $\lambda \ge 1$ and $t, \g \in \R^3$ satisfying $|\g|_{\infty}\le a$.
\end{lem}

\begin{proof}
By symmetry it suffices to justify only the estimate $(\lesssim)$. Equivalently we show that
\begin{align} \label{red5}
\Big( \lambda^{\kappa} | |t+\g|_c^c - \lambda| \Big)^{10}
\lesssim
1+\Big( \lambda^{\kappa} | |t|_c^c - \lambda| \Big)^{10},
\end{align} 
uniformly in $\lambda \ge 1$ and $t, \g \in \R^3$ satisfying $|\g|_{\infty}\le a$.
Since the case $|t|_\8 \le 2a$ is obvious, from now on we may assume that $|t|_\8>2a$. 
This implies that $|t + \g |_\8 \simeq |t|_\8$ and consequently there exists a constant $C_a\ge 1$ such that 
\[
C_a^{-1} |t|_c^c \le |t+\g|_c^c \le C_a |t|_c^c, \qquad |t|_\8>2a, \quad |\g|_{\infty}\le a.
\]
Moreover, for every $j\in[3]$, we have
$\big| |t_j+\g_j|^c-|t_j|^c \big| \lesssim |t_j|^{c-1}+1$,
which implies
\begin{align}
\label{eq:4}
\big| |t+\g|_c^c - |t|_c^c \big| \le \sum_{j=1}^3 \big| |t_j+\g_j|^c-|t_j|^c \big| 
\lesssim  
|t|_\8^{c-1} +1
\le  |t|_c^{c - 1}+1.
\end{align}

We now distinguish  two cases.

\paragraph{\bf Case 1.} Assume that $|t|_c^c > 2C_a\lambda$. Then we have $|t|_c^c\ge1$ and 
$|t|_c^c - \lambda \simeq |t|_c^c$. Using \eqref{eq:4} we see that  \eqref{red5} follows, since
\[
|t+\g|_c^c - \lambda\le
(|t|_c^c - \lambda) +\big||t+\g|_c^c - |t|_c^c\big|
\lesssim
(|t|_c^c - \lambda) + |t|_c^{c - 1}
\lesssim |t|_c^c - \lambda.
\]
\paragraph{\bf Case 2.}  Assume that $|t|_c^c \le2C_a\lambda$.
Using \eqref{eq:4}, as in the previous case, we may write
\[
\big| |t+\g|_c^c - \lambda \big| 
\le
\big| |t|_c^c - \lambda \big| + |t|_c^{c - 1}+1
\lesssim 
\big| |t|_c^c - \lambda \big| + \lambda^{1-1/c}+1.
\]
Thus the left-hand side of \eqref{red5} is controlled by
\begin{align*}
 \left( \lambda^{\kappa} | |t|_c^c - \lambda| \right)^{10} + 
\left( \lambda^{\kappa} (\lambda^{1-1/c}+1)  \right)^{10} 
  \lesssim 
\left(\lambda^{\kappa} | |t|_c^c - \lambda| \right)^{10} + 1,
\end{align*}
as required.
\end{proof}

We define
\begin{align}
\label{eq:5}
K_\lambda (x): =\lambda^{- 9/(4c)} \frac{1}{1 + \big( \lambda^{\kappa} | |x|_c^c - \lambda| \big)^{10}},
\qquad x \in \R^3, \quad \lambda \ge 1,
\end{align} 
and observe that 
\begin{align}
\label{eq:8}
\lambda^{- 3/c + 1} |\sigma_{\la}^{\mathfrak M} (x)| \lesssim K_\lambda (x), 
\qquad x \in \Z^3, \quad \lambda \ge 1.
\end{align}

For any function $f:\Z^3\to\mathbb{C}$ let us consider its extension $F_f:\R^3\to\mathbb{C}$ by
$$
F_f(x):=\sum_{y\in\Z^3} f(y)\ind_{[-1/2,1/2)^3 }(x-y),\qquad  x\in\R^3.
$$
If $f\in\ell^p(\Z^3)$ with $p\in[1, \infty]$ it is easy to see that
$F_f\in L^p(\R^3)$ and $\|f\|_{\ell^p(\Z^3)}=\|F_f\|_{L^p(\R^3)}$.
The next lemma is variant of a comparison principle that was recently
used in \cite[Theorem 1, p. 859]{BMSW1}, see also \cite{BMSW}. Lemma
\ref{norm} will transfer the problem from the set of integers $\Z^3$
to the continuous setting $\R^3$, where the properties of the maximal
 functions corresponding to the kernels $K_\lambda$ will be investigated.
\begin{lem}\label{norm}
Let $c \ge 1$ and $p\in[1, \infty]$ be fixed. Then for any $\mathbb{I}\subseteq [1,\infty)$ one has
\begin{align*}
    \big\| {\sup_{\la\in\mathbb{I}} \left|\Kl\ast f\r|} \big\|_{\ell^p(\Z^3)}
		\lesssim
		\big\|{\sup_{\la\in\mathbb{I}}  \Kl\ast |F_f |} \big\|_{L^p(\R^3)}.
\end{align*} 
\end{lem}

\begin{proof}
We assume that $p\in[1, \infty)$; the case of $p = \8$ can be handled in a similar way and thus is omitted.
Since $\Kl$ is nonnegative, it is clear that it suffices to prove that
\begin{equation*}
\Big(\sup_{\la\in\mathbb{I}} \Kl\ast |f|(x)\Big)^p
\lesssim
\int_{x + [-1/2,1/2)^3} \Big(\sup_{\la\in\mathbb{I}} \Kl\ast |F_f|(y)\Big)^p \, dy,
\qquad x\in\Z^3.
\end{equation*}
This, in turn, will follow once we show the estimate
\begin{align} \label{red6}
\Kl \ast |f| (x) \lesssim \Kl \ast |F_f| (y), 
\qquad x \in \Z^3, \quad y \in x + [-1/2,1/2)^3, \quad \lambda \ge 1.
\end{align}
Indeed, using Lemma~\ref{kernel} (applied with $t=x-z$, $\g=y-x+z-t$ and $a=1$) we obtain
for all $x,z \in \Z^3$ and  $y \in x + [-1/2,1/2)^3$ and  $t \in z + [-1/2,1/2)^3$ that
\begin{align*}
\Kl (x-z) \lesssim \Kl (y-t), \qquad  \lambda \ge 1.
\end{align*}
Consequently, uniformly in $y \in x + [-1/2,1/2)^3$ we obtain
\begin{align*}
\Kl\ast|f|(x)
=
\sum_{z\in\Z^3} |f(z)|\Kl(x-z)
\lesssim
\sum_{z\in\Z^3}  |f(z)| \int_{ z + [-1/2,1/2)^3 } \Kl(y-t) \, dt
= 
\Kl\ast|F_f|(y).
\end{align*}
This finishes the proof of Lemma \ref{norm}.
\end{proof}

In order to study $L^p(\R^3)$ bounds of the maximal functions
corresponding to the kernels $K_\lambda$ efficiently,  we will need a polar
decomposition with respect to the norm $|\cdot|_{c}$.  Proceeding as
in \cite[Theorem 2.49, p.  78]{folland} we can easily show that there
exists the unique Borel measure on $\mathbb S_c^2$ denoted by $\mu_c$
such that for every $f \ge 0$ or $f \in L^1(\R^3)$ we have the
identity
\begin{align} \label{polar}
\int_{\R^3} f(x) \,dx = \int_0^\infty r^2 \int_{\mathbb S_c^2} f(ry) \, d\mu_c(y) \, dr.
\end{align}
Further, using this formula one can easily show that for $f \ge 0$ or $f \in L^1(\mathbb S_c^2, d\mu_c)$ we have
\begin{align} \label{intS} 
 \int_{\mathbb S_c^2} f(x) \, d\mu_c(x)  
 =
\int_{\atop{y \in \R^2}{ |y|_c \le 1}} 
\Big( f\big(y , (1 - |y|_c^c)^{1/c} \big) + f\big(y , - (1 - |y|_c^c)^{1/c} \big)
\Big)
(1 - |y|_c^c)^{1/c - 1} \, dy.
\end{align}

\begin{rem}
\label{rem:10}
Using the identity \eqref{intS} with $f = 1$ and the well-known
relation between the Beta and Gamma functions we can easily compute
that $\mu_c(\mathbb S_c^2) = \frac{8 \Gamma(1/c)^3}{c^2\Gamma(3/c) }$.
\end{rem}

Now is not difficult to see that inequality \eqref{eq:8} and Lemma
\ref{norm} reduce the proof of Theorem \ref{thm:major} to Theorem
\ref{thm:kl} below. The key ingredient in the proof of Theorem \ref{thm:kl}
will be the Fourier transform estimates for the measures $\mu_c$
associated with the spheres $\mathbb S^2_c$. Namely, for every fixed
$c\in(1, 2)$ one has
\begin{align}\label{eq:fourier}
|\mathcal{F}_{\R^3} \mu_c(\xi)|+|\nabla \mathcal{F}_{\R^3}\mu_c (\xi)|\lesssim (1+|\xi|)^{-1}, \qquad \xi \in \R^3.
\end{align}
The proof of \eqref{eq:fourier} is the most technical part of this
paper and hence has been postponed to Section~\ref{sec:osc}.  We
assume momentarily that inequality \eqref{eq:fourier} has been proven and we use
it to establish Theorem \ref{thm:kl}.

\begin{thm}
\label{thm:kl}
Let $c \in (1, 2)$ be fixed, and $K_\lambda$ be the kernel defined in \eqref{eq:5}.
Then the following two inequalities hold
\begin{align} \label{eq:est16}
\big\| \sup_{1 \le \lambda < \8} | K_\lambda \ast f| \big\|_{L^p(\R^3)}
&\lesssim_p
\| f \|_{L^p(\R^3)}, \qquad f \in L^p(\R^3), \quad p > 3/2, \\ \label{eq:est16dyad}
\big\| \sup_{n\in\Z_+} | K_{\lambda_n} \ast f| \big\|_{L^p(\R^3)}
&\lesssim_p
\| f \|_{L^p(\R^3)}, \qquad f \in L^p(\R^3), \quad p > 1.
\end{align} 
\end{thm}

\begin{proof}
Our aim is to prove that for every $\lambda\ge 1$ the following estimates hold
\begin{align}
\label{eq:9}
\|K_\lambda\|_{L^1(\R^3)}&\simeq 1,\\
\label{eq:10}
\big|\mathcal{F}_{\R^3} K_\lambda (\xi)-\mathcal{F}_{\R^3} K_\lambda (0)\big|&\lesssim |\lambda^{1/c}\xi|,\\
\label{eq:11}
\big|\mathcal{F}_{\R^3} K_\lambda (\xi)\big|&\lesssim |\lambda^{1/c}\xi|^{-1},\\
\label{eq:12}
\Big|\frac{d}{dt}\mathcal{F}_{\R^3} K_t (\xi)\Big|\bigg|_{t=\lambda}&\lesssim \lambda^{-1},
\end{align}
with the implicit constants independent of $\lambda\ge1$ and $\xi\in\R^3$.

Once \eqref{eq:9}, \eqref{eq:10}, \eqref{eq:11} and \eqref{eq:12} are
established we may proceed much the same way as in \cite{MSZ1} to
deduce \eqref{eq:est16} and \eqref{eq:est16dyad}, we refer also to
\cite[(4.7) and (4.8), pp. 16--19]{BMSW} or \cite{BMSW1}. More
precisely, to prove \eqref{eq:est16dyad} it suffices to use
\eqref{eq:9}, \eqref{eq:10}, \eqref{eq:11} with the standard
Littlewood--Paley theory and appeal to \cite[Theorem 2.14,
p. 537]{MSZ1}. To prove \eqref{eq:est16} we use estimates
\eqref{eq:9}, \eqref{eq:10}, \eqref{eq:11}, \eqref{eq:12} with the
standard Littlewood--Paley theory and invoke \cite[Theorem 2.39,
p. 543]{MSZ1}.  This reduces the matter to proving \eqref{eq:9},
\eqref{eq:10}, \eqref{eq:11} and \eqref{eq:12}. We first show
\eqref{eq:9}. Using \eqref{polar} and changing variables $r=\lambda^{1/c}t$ we note
\begin{align*}  
\|K_\lambda\|_{L^1(\R^3)} &=
\lambda^{- 9/(4c)} \int_0^\infty \frac{\mu_c(\mathbb S_c^2)r^2 dr}{1 +  \big(\lambda^{\kappa} | r^c - \lambda|\big)^{10}}  \\
&=\lambda^{\kappa+1} \int_0^\infty \frac{\mu_c(\mathbb S_c^2)t^2 dt}{1 +  \big(\lambda^{\kappa+1} | t^c - 1|\big)^{10}}.
\end{align*} 
For every $c\ge1$ and $t\ge0$ we have $|t^c-1|\geq |t-1|$, thus \eqref{eq:9} follows, since
\begin{align}
\label{eq:13}
 \int_0^\infty \frac{\lambda^{\kappa+1}t^2 dt}{1 +  \big(\lambda^{\kappa+1} | t^c - 1|\big)^{10}}
 \lesssim  \int_0^\infty \frac{\lambda^{\kappa+1}dt}{1 +  \big(\lambda^{\kappa+1} | t - 1|\big)^{8}}
\lesssim  \int_{\R} \frac{dt}{1 +  t^{8}}
 \lesssim 1.
\end{align} 
Invoking \eqref{polar} and changing variable $r=\lambda^{1/c}t$, we obtain
\begin{align}
\label{eq:14}
\begin{split}
\mathcal{F}_{\R^3} K_\lambda (\xi) &=
\lambda^{- 9/(4c)} \int_0^\infty\frac{1}{1 +  \big(\lambda^{\kappa} | r^c - \lambda|\big)^{10}}
\int_{\mathbb S_c^2}  e(- rx \cdot\xi) \, d\mu_c(x) r^2 dr\\
&=\lambda^{\kappa+1} \int_0^\infty\frac{t^2}{1 +  \big(\lambda^{\kappa+1} | t^c - 1|\big)^{10}}\mathcal{F}_{\R^3} \mu_c(\la^{1/c}t\xi)  dt,
\end{split}
\end{align} 
since $\mathcal{F}_{\R^3} \mu_c(\xi)=\int_{\mathbb S_c^2}  e(- x \cdot\xi) \, d\mu_c(x)$.

We now prove \eqref{eq:10}. By \eqref{eq:14} we may conclude
\begin{align*}
\big|\mathcal{F}_{\R^3} K_\lambda (\xi)-\mathcal{F}_{\R^3} K_\lambda (0)\big|
\lesssim  |\lambda^{1/c}\xi|\int_0^\infty\frac{\lambda^{\kappa+1}t^3dt}{1 +  \big(\lambda^{\kappa+1} | t^c - 1|\big)^{10}},
\end{align*}
since
$|\mathcal{F}_{\R^3} \mu_c(\xi)-\mathcal{F}_{\R^3} \mu_c(0)|\lesssim |\xi|$,
hence arguing in a similar way as in \eqref{eq:13} we obtain the
claim. To prove \eqref{eq:11} we use \eqref{eq:14} and the estimate
$|\mathcal{F}_{\R^3} \mu_c(\xi)|\lesssim |\xi|^{-1}$ from
\eqref{eq:fourier} and we are done. Finally, to prove \eqref{eq:12} we
differentiate \eqref{eq:14} (this produces a factor $\lambda^{-1}$)
and then we use $|\nabla\mathcal{F}_{\R^3} \mu_c(\xi)|\lesssim |\xi|^{-1}$ from
\eqref{eq:fourier} and we are also done.
\end{proof}

%%%%%%%%%%%%%%%%%%%%%%%%%%%%%%%%%%%%%%%%%%%%%%%%%%%%%%%%%%%%%%%%%%%%%%%%%%%%%%%%%%%%%%
\section{Proof of inequality \eqref{eq:fourier}. Fourier transform estimates }\label{sec:osc}
%%%%%%%%%%%%%%%%%%%%%%%%%%%%%%%%%%%%%%%%%%%%%%%%%%%%%%%%%%%%%%%%%%%%%%%%%%%%%%%%%%%%%%

We start by introducing a (non-smooth) decomposition of the unity on $\mathbb S_c^2$.
It is a very useful construction which will allow us to estimate the oscillatory integrals from \eqref{eq:fourier}.
We define for any $j \in [3]$ the homeomorphisms
\begin{align*}
&\vp_j^{\pm} \colon \{\xi \in \mathbb S_c^2 \colon \pm\xi_j > 0 \} \to \{ (x_1,x_2)\in\R^2 : |x_1|^c + |x_2|^c < 1 \}, 
\end{align*}
where
\begin{align*}
\vp_j^{\pm} (\xi) := (\xi_1, \ldots, \xi_{j-1}, \xi_{j+1}, \ldots, \xi_3 ).
\end{align*}
\begin{lem} \label{lem:decom}
Let $c \in (1,2)$ be fixed. Then, for each $j \in [3]$ there exist  continuous functions $g_j^{\pm} \colon \mathbb S_c^2 \to [0,1]$ such that 
$\supp g_j^{\pm} = \big\{ \xi \in \mathbb S_c^2 : \pm\xi_j \ge  (14/48)^{1/c} \big\}$
and satisfy
\begin{align} \label{est44}
\sum_{j=1}^3 \big( g_j^+(\xi) + g_j^-(\xi)  \big) = 1, 
\qquad \xi \in \mathbb S_c^2.
\end{align}
Moreover, for each $j \in [3]$ there exist  smooth functions $M_j^{\pm}$ on
\begin{align*}
\mathbb{A}:= \Big\{ (x_1,x_2,y) \in \R^3 : |y| < 1 - \frac{14}{48} \,\, \textrm{or} \,\, \frac{14}{48} < |x_1|^c < 2 - \frac{14}{48} \,\, \textrm{or} \,\, \frac{14}{48} < |x_2|^c < 2 - \frac{14}{48} \Big\},
\end{align*}
such that the functions 
\begin{align*}
\{ (x_1,x_2) \in \R^2 : |x_1|^c + |x_2|^c < 1 \}
\ni (x_1,x_2) \mapsto M_j^{\pm} (x_1,x_2,|x_1|^c + |x_2|^c)
\end{align*}
are supported in 
$\{(x_1, x_2)\in\R^2: |x_1|^c + |x_2|^c \le 1 - 14/48 \}$ and we have
\begin{align} \label{iden11}
g_j^{\pm} \circ (\vp_j^{\pm})^{-1} (x_1,x_2)
=
M_j^{\pm} (x_1,x_2, |x_1|^c + |x_2|^c ), \qquad \textrm{for \quad $|x_1|^c + |x_2|^c < 1$}.
\end{align}
\end{lem}

\begin{proof}
Let $H \colon \R \to [0,1]$ be a compactly supported smooth  function such that $H(x) = 1$ for 
$x \in [-1 + 15/48, 1 - 15/48]$, $\supp H \subset [-1 + 14/48, 1 - 14/48]$ and 
$H(x) > 0$ for $x \in (-1 + 14/48, 1 - 14/48)$. Further, we let 
\[
h (x_1,x_2) := H (|x_1|^c + |x_2|^c), \qquad x_1,x_2 \in \R.
\]
Observe that $h$ is smooth everywhere except possibly at the points $(x_1,x_2)$ such that $x_1 = 0$ or $x_2 = 0$.
Furthermore, we consider 
\begin{align*}
h_j^{\pm} (\xi) 
:=
\begin{cases}
	h \circ \vp_j^{\pm} (\xi) , & \quad \pm\xi_j > (13/48)^{1/c},\\
	 0, & \quad \pm\xi_j < (14/48)^{1/c},
\end{cases} 
\quad \xi \in \mathbb S_c^2.
\end{align*}
Notice that $h_j^{\pm}$ are well-defined and continuous because $\supp H \subset [-1 + 14/48, 1 - 14/48]$. Moreover, we have
\begin{align*}
\supp h_j^{\pm} = \big\{ \xi \in \mathbb S_c^2 : \pm\xi_j \ge  (14/48)^{1/c} \big\},
\qquad j \in [3].
\end{align*}

Let us consider 
\[
g (\xi) := \sum_{j=1}^3 \big( h_j^+(\xi) + h_j^-(\xi)  \big), 
\qquad \xi \in \mathbb S_c^2,
\]
and observe that $g(\xi) > 0$ on $\mathbb S_c^2$. Finally, we define
\begin{align*}
g_j^{\pm} (\xi) := \frac{h_j^{\pm} (\xi) }{ g(\xi) },
\qquad \xi \in \mathbb S_c^2, \quad j \in [3].
\end{align*}
Notice that \eqref{est44} holds, thus it remains to prove the
existence of the functions $M_j^{\pm}$.  By symmetry it is enough to
focus on $M_3^+$.  Observe  that for
$|x_1|^c + |x_2|^c < 1$ we obtain
\begin{align*}
h_3^{+} \circ (\vp_3^{+})^{-1} (x_1,x_2) & = H( |x_1|^c + |x_2|^c ), \qquad
h_3^{-} \circ (\vp_3^{+})^{-1} (x_1,x_2) = 0, \\
(h_1^{+} + h_1^{-} ) \circ (\vp_3^{+})^{-1} (x_1,x_2) 
& = H( 1 - |x_1|^c ), \\
(h_2^{+} + h_2^{-} ) \circ (\vp_3^{+})^{-1} (x_1,x_2) 
& = H( 1 - |x_2|^c ),
\end{align*}
and consequently we can take 
\begin{align*}
M_3^+(x_1,x_2,y) := \frac{H(y)}{H(y) + H(1- |x_1|^c ) + H(1- |x_2|^c )},
\qquad (x_1,x_2,y) \in \mathbb{A}.
\end{align*}
Now the mapping properties of $M_3^+$ and \eqref{iden11} follow easily from the mapping properties of $H$ and the inclusion
\begin{align*} 
\big\{ (x_1,x_2, |x_1|^c + |x_2|^c ) \in \R^3 : |x_1|^c + |x_2|^c < 1 \big\}
\subseteq
\mathbb{A}.
\end{align*}
This finishes the proof of Lemma~\ref{lem:decom}.
\end{proof}

In order to prove \eqref{eq:fourier} we establish a slightly more general result, which will be used in the proof of  Theorem \ref{thm:erg}.
\begin{lem} \label{lem:FTest} Assume that $c \in (1,2)$ is fixed and
let
$f \in C(\mathbb{R}^3) \cap C^\infty((\mathbb{R}\setminus\{0\})^3)$ be
such that for every fixed $\beta \in \N^3$ we have the estimate
\begin{align} \label{35.2}
|\partial^{\beta} f(x)|
\lesssim
\prod_{j=1}^{3} |x_j|^{-\beta_j + (c-1) \ind_{\beta_j \ge 1} }, \qquad x \in (\mathbb{R}\setminus\{0\})^3.
\end{align}
Then the following estimate holds
\begin{align} \label{35.1}
\Big| \int_{\mathbb S_c^2} f(w) e(-w \cdot \xi) \, d\mu_c(w)  \Big|
\lesssim
(1 + |\xi|)^{-1}, \qquad \xi \in \R^3.
\end{align}
\end{lem}

It is clear that only the values of $f$ on $\mathbb S_c^2$ are
essential. Therefore we may think that $f$ has compact support and in
practice it suffices to verify the assumption \eqref{35.2} only for
$|x| \le C$, where $C>0$ is a large absolute constant.

\begin{proof}
By symmetry, taking into account Lemma~\ref{lem:decom} and \eqref{intS}, we see that our task 
is reduced to showing that
\begin{align} \label{red8}
\Big|
\int_{\R^2} e(- (x_1 \xi_1 + x_2 \xi_2 + \phi(x) \xi_3))
f(x, \phi(x)) M(x,|x|_c^c) \, dx
\Big|
\lesssim
(1 + |\xi|)^{-1}, \qquad \xi \in \R^3,
\end{align}
where $x=(x_1,x_2)$, $|x|_c^c = |x_1|^c + |x_2|^c$, 
$\phi(x) := (1 - |x_1|^c - |x_2|^c)^{1/c}$ and $M$ is a fixed smooth function on $\mathbb{A}$ such that 
$\supp M(\cdot, |\cdot|_c^c) \subset \{ x \in \R^2 : |x|_c^c \le 1 - 14/48 \}$.
Let us define
\begin{align*} 
F(x)
:=
f(x, \phi(x)) M(x,|x|_c^c), \qquad x \in \mathbb{R}^2.
\end{align*}
It is clear that
$F \in C_c(\R^2) \cap C^\infty((\mathbb{R}\setminus\{0\})^2)$ and
$\supp F \subseteq \{x \in \R^2 : |x|_c^c \le 1 - 14/48 \}$. Further,
one can show that for every fixed $\alpha \in \N^2$ we have
\begin{align} \label{38.1}
|\partial^{\alpha}_x F(x)|
\lesssim
\prod_{j=1}^{2} |x_j|^{-\alpha_j + (c-1) \ind_{\alpha_j \ge 1} }, \qquad x \in (\mathbb{R}\setminus\{0\})^2.
\end{align}
Indeed, taking into account the fact that the assumptions imposed
on $f$ are symmetric it suffices to check \eqref{38.1} for
$x_1, x_2 > 0$ satisfying $|x|_c^c \le 1 - 14/48$. In view of
the Leibniz rule it is  reduced to proving
\begin{align} \label{38.2}
\big|\partial^{\alpha}_x [ f(\cdot, \phi (\cdot) ) ] (x) \big|
& \lesssim
\prod_{j=1}^{2} x_j^{-\alpha_j + (c-1) \ind_{\alpha_j \ge 1} }, 
\qquad x_1, x_2 > 0, \quad |x|_c^c \le 1 - 14/48, \\ \nonumber
\big|\partial^{\alpha}_x [ M(\cdot, |\cdot|_c^c ) ] (x) \big|
& \lesssim
\prod_{j=1}^{2} x_j^{-\alpha_j + (c-1) \ind_{\alpha_j \ge 1} }, \qquad x_1, x_2 > 0, \quad |x|_c^c \le 1 - 14/48.
\end{align}
By a simple induction we obtain the formula
\begin{align*} 
& \partial^{\alpha}_x [ f(\cdot, \phi (\cdot) ) ] (x) \\
& \quad =
\sum_{\beta_1 + \gamma_1 + \delta_1 \le \alpha_1}
\sum_{\beta_2 + \gamma_2 + \delta_2 \le \alpha_2}
\sum_{\theta \le \alpha_1 + \alpha_2}
\partial^{(\beta_1 , \beta_2, \gamma_1 + \gamma_2)} f(x, \phi (x) )
P_{\alpha, \beta, \gamma, \delta, \theta} (x_1, x_1^c, x_2, x_2^c, \phi (x) ) \\
& \qquad \qquad \times
\Big( \prod_{j=1}^{2} x_j^{-\gamma_j - \delta_j + (c-1) \ind_{\gamma_j + \delta_j \ge 1} } \Big)
(1 - |x|_c^c)^{- \theta},
\end{align*}
for some polynomials $P_{\alpha, \beta, \gamma, \delta, \theta}$. This
immediately implies \eqref{38.2}, the second bound can be obtained
similarly.  Now we come back to the proof of \eqref{red8}.  Observe
that this estimate is trivial for $|\xi| \le 1$, therefore we may
assume that $|\xi| \ge 1$. We shall distinguish two cases.

\paragraph{\bf Case 1.} Assume that $|\xi_3| \le \tau |\xi'|$, where
$\xi' = (\xi_1,\xi_2)$ and $\tau>0$ is a constant such that
$2\tau|\nabla \phi (x)| \le 1$ for all $|x|_c^c < 3/4$.   Let
$\rho \colon \R \to [0,1]$ be a smooth function such that
$\rho(x) = 1$ for $x \in [-1,1]$ and $\supp \rho \subset [-2,2]$. To prove \eqref{red8} we take $\delta := |\xi|^{-1} \le 1$ and show that
\begin{align} \label{red9}
\Big|
\int_{\R^2} e\big( - (x \xi' + \phi(x) \xi_3) \big) 
\big(1 - \rho(x_1/\delta) \big)  \big(1 - \rho(x_2/\delta) \big)
F(x) \, dx
\Big|
\lesssim
|\xi'|^{-1}.
\end{align}
Let us define $\widetilde{\phi}(x,\xi) := \frac{x\xi'}{|\xi'|} + \frac{\xi_3}{|\xi'|} \phi(x)$ and observe that
$\nabla_{\!x} \widetilde{\phi}(x,\xi) = \frac{\xi'}{|\xi'|} + \frac{\xi_3}{|\xi'|} \nabla \phi(x)$. Consequently, this gives us 
$|\nabla_{\!x} \widetilde{\phi}(x,\xi) | \ge 1/2$.
Now integrating by parts 
with respect to the first order differential operator 
\begin{align*} 
D_x := \sum_{j=1}^2 \big( \partial_{x_j} \widetilde{\phi}(x,\xi) \big) \partial_{x_j}
\end{align*}
and using \eqref{38.1}
we see that the left-hand side of \eqref{red9} is controlled by
\begin{align*} 
& |\xi'|^{-1}
\int_{\R^2} \bigg( \sup_{1\le k,l \le 2} 
\frac{| \partial_{x_k} \partial_{x_l} \widetilde{\phi}(x,\xi) | }
{|\nabla_{\!x} \widetilde{\phi}(x,\xi) |^2}
\big(1 - \rho(x_1/\delta) \big)  \big(1 - \rho(x_2/\delta) \big)
|F(x)| \\
& \qquad \qquad \qquad +
\frac{1}{ |\nabla_{\!x} \widetilde{\phi}(x,\xi) | }
\Big| \nabla_{\!x} \Big( \big(1 - \rho(x_1/\delta) \big)  \big(1 - \rho(x_2/\delta) \big)
F(x) \Big) \Big|
\bigg) \, dx \\
& \quad \lesssim
|\xi'|^{-1}
\int_{ |x|_c^c < 3/4 } \Big( |x_1|^{c-2} + |x_2|^{c-2} 
+ \delta^{-1} (\ind_{\delta \le |x_1| \le 2 \delta} 
+ \ind_{\delta \le |x_2| \le 2 \delta})
\Big) \, dx \\
& \quad \lesssim
|\xi'|^{-1}.
\end{align*}
This proves \eqref{red9}.

\paragraph{\bf Case 2.} We now assume that $|\xi_3| > \tau |\xi'|$. To prove \eqref{red8} we need to show the estimate 
\begin{align} \label{red10}
\Big|
\int_{\R^2} e\big( - (x \xi' + \phi(x) \xi_3) \big) 
F(x) \, dx
\Big|
\lesssim
|\xi_3|^{-1}.
\end{align}
Using  $\psi_0 \in C_c^{\8}([-4,-1/2] \cup [1/2,4])^2$ we introduce a dyadic partition of unity
such that for
$0 < |x_1|, |x_2| \le 1$ we have 
\begin{align*} 
\sum_{j_1,j_2 \ge 0} \psi_0(2^j\circ x) = 1,
\end{align*}
where 
$2^j\circ x := (2^{j_1} x_1, 2^{j_2} x_2)$ for all $j\in\Z^2$. Now the left-hand
side of \eqref{red10} is dominated by
\begin{align*} 
& \sum_{j_1,j_2 \ge 0}
\Big|
\int_{\R^2} e\big( - (x \xi' + \phi(x) \xi_3) \big) 
F(x) \psi_0(2^j\circ  x) \, dx
\Big| \\
& \quad \le
\sum_{j_1,j_2 \ge 0} 2^{-j_1 - j_2}
\Big|
\int_{\R^2} 
e\big( - \big(x (2^{-j}\circ \xi') + \phi(2^{-j}\circ x) \xi_3 \big) \big) 
F(2^{-j} \circ x) \psi_0 (x) \, dx
\Big|.
\end{align*}
Since $|2^{-j}\circ \xi'| \le |\xi'|$, it is enough to prove for all $j=(j_1, j_2)\in\N^2$ that 
\begin{align} \label{red11}
\Big|
\int_{\R^2} 
e\big( - \big(x \xi' + \phi(2^{-j}\circ x) \xi_3 \big) \big)  
F(2^{-j} \circ x) \psi_0 (x) \, dx
\Big|
\lesssim
|\xi_3|^{-1} 2^{j_1 c/2 + j_2 c/2}.
\end{align}
It is enough to prove \eqref{red11} only for $j\in\N^2$ satisfying
$j_1 \vee j_2 \ge C_c$, where $C_c>0$ is a large absolute constant
depending only on $c\in(1, 2)$. A more precise value of $C_c$ will be specified later. In
the opposite case we have $j_1, j_2 \le C_c$, this means that we only
need to handle finitely many terms in which both the phase and the
amplitude functions are smooth and the determinant of the Hessian of
the phase is non-zero on the support of the amplitude. Therefore, for
 $j_1,j_2\le C_c$ (in fact with an arbitrary $C_c>0$) the estimate \eqref{red11} follows from the general
theory, see \cite[Lemma 4.15, p. 87]{MusSch}. So from now on we
assume that $j_1 \vee j_2 \ge C_c$. By the symmetry
we may assume that $\supp \psi_0 \subset [1/2, 4]^2$.
Let us define
\begin{align*} 
I_j(\xi) 
:= \int_{\R^2} e( - \xi_3 \Phi_j(x,\xi)) F(2^{-j} \circ x)
\psi_0 (x) \, dx,
\end{align*}
where
\begin{align*} 
\Phi_j(x,\xi) := \frac{x \xi'}{\xi_3} + \phi_j(x),\quad \textrm{with}\quad \phi_j(x) := \phi (2^{-j}\circ x).
\end{align*}
Observe that 
\begin{align*} 
|I_j(\xi)|^2 
= \int_{\R^2} J_j(u,\xi) \, du,
\end{align*}
where
\begin{align*} 
J_j(u,\xi) &:= 
\int_{\R^2} 
e\big( \xi_3 ( \Phi_j(x+u,\xi) - \Phi_j(x,\xi) ) \big)
\Psi_j (x,u) \, dx, \\
\Psi_j (x,u) &:= 
\psi_0 (x+u) \overline{F\big(2^{-j}\circ (x+u) \big)}
\psi_0 (x) F(2^{-j}\circ x).
\end{align*}
Our aim is to show that 
\begin{align} \label{red12}
|J_j(u,\xi)| 
\lesssim 
|\xi_3|^{-3} |2^{-jc}\circ u|^{-3}, \qquad j_1 \vee j_2 \ge C_c.
\end{align}
Notice that combining \eqref{red12} with the trivial bound 
$|J_j(u,\xi)| \lesssim 1$, we see that the square of the left-hand side in \eqref{red11} is controlled by
\begin{align*} 
|I_j(\xi)|^2 
& 
\lesssim
\int_{\R^2} \big( |\xi_3|^{-3} |2^{-jc}\circ u|^{-3} \big) \wedge 1 \, du\\
&=
|\xi_3|^{-2} 2^{j_1 c + j_2 c } 
\int_{\R^2} |u|^{-3} \wedge 1 \, du \\
& \lesssim
|\xi_3|^{-2} 2^{j_1 c + j_2 c }. 
\end{align*}

Therefore we have reduced the proof of Lemma \ref{lem:FTest} to showing \eqref{red12}.
Let us define
\begin{align*} 
a_j(x,u) := \frac{\nabla_{\!x} \Phi_j(x+u,\xi) - \nabla_{\!x} \Phi_j(x,\xi) }
{|\nabla_{\!x} \Phi_j(x+u,\xi) - \nabla_{\!x} \Phi_j(x,\xi)|^2}
= \frac{b_j(x,u)}{|b_j(x,u)|^2},
\end{align*}
where
\[
b_j(x,u):=\nabla_{\!x} \phi_j(x+u) - \nabla_{\!x} \phi_j(x). 
\]
We shall write  $a_{j}(x,u)=(a_{j,1}(x,u), a_{j,2}(x,u))$, $b_{j}(x,u)=(b_{j,1}(x,u),b_{j,2}(x,u))$.
Further, we consider the following differential operators
\begin{align*} 
L_j f(x) & := \frac{1}{\dpi \xi_3} (a_j(x,u) \nabla_{\!x} ) f(x)
=
\frac{1}{\dpi \xi_3} \sum_{k=1}^2 a_{j,k} (x,u) \partial_{x_k} f(x), \\
L_j^* f(x) & = 
- \frac{1}{\dpi \xi_3} 
\sum_{k=1}^2 \partial_{x_k} \big( a_{j,k} (x,u) f(x) \big).
\end{align*}
We shall write $L_{j,t}$ and $L_{j,t}^*$ to indicate that $L_{j}$ and $L_{j}^*$, respectively, act on a variable $t$.
Notice that $L_j^*$ is an adjoint operator to $L_j$ and we have the identity
\begin{align*} 
L_{j,x} \big( e \big( \xi_3 ( \Phi_j(x+u,\xi) - \Phi_j(x,\xi) ) \big)  \big)
=
e \big( \xi_3 ( \Phi_j(x+u,\xi) - \Phi_j(x,\xi) ) \big). 
\end{align*}
Therefore 
\begin{align} \label{est18}
|J_j(u,\xi)| \le 
\int_{\R^2} 
\big| (L_{j,x}^*)^3
\Psi_j (x,u) \big| \, dx.
\end{align}
We now show that 
\begin{align} \label{est19}
|b_j(x,u)| \simeq 
|2^{-jc}\circ u|, \qquad j_1 \vee j_2 \ge C_c, \quad (x,u) \in \supp \Psi_j.
\end{align}
Letting $\g(t) := 2^{-j}\circ x + t2^{-j}\circ u$ for $t \in [0,1]$, we see that for $k\in[2]$ we have
\begin{align*} 
b_{j,k} (x,u) 
& = 2^{-j_k} \big( (\partial_k \phi) (2^{-j}\circ (x+u)) - (\partial_k \phi) (2^{-j} \circ x) \big) \\
& = 2^{-j_k} \int_0^1 \Big( (\partial_k \partial_1 \phi) ( \g(t) ) \g'_1(t) 
+ (\partial_k \partial_2 \phi) ( \g(t) ) \g'_2(t)
\Big) \, dt.
\end{align*}
Therefore simple computations produce
\begin{align*} 
|b_{j,1} (x,u) |
& \simeq 2^{-j_1} \Big| \int_0^1 \big( 1- |\g(t)|_c^c \big)^{1/c - 2}
\Big( \g_1(t)^{c-2} (1 - \g_2(t)^c) 2^{-j_1} u_1 
+ (\g_1(t) \g_2(t))^{c-1} 2^{-j_2} u_2 
\Big) \, dt \Big|,  \\
|b_{j,2} (x,u) |
& \simeq 2^{-j_2} \Big| \int_0^1 \big( 1- |\g(t)|_c^c \big)^{1/c - 2}
\Big( 
(\g_1(t) \g_2(t))^{c-1} 2^{-j_1} u_1
+ \g_2(t)^{c-2} (1 - \g_1(t)^c) 2^{-j_2} u_2
\Big) \, dt \Big|.
\end{align*}
Since $(x,u) \in \supp \Psi_j$ we see that
$\g(t) \in 2^{-j} \circ [1/2,4]^2$ and $|\g(t)|_c^c < 3/4$, $t\in [0,1]$.
Using this we can easily obtain the estimate $(\lesssim)$ in \eqref{est19}. Indeed, we have
\begin{align*} 
|b_j (x,u) |
\lesssim 
\int_0^1 \Big( 2^{-j_1 c} |u_1| + 2^{-j_1 c} 2^{-j_2 c} |u_2| 
+ 2^{-j_1 c} 2^{-j_2 c} |u_1| + 2^{-j_2 c} |u_2|
\Big) \, dt
\simeq
|2^{-jc} \circ u|.
\end{align*}
Therefore, it suffices to prove $(\gtrsim)$ in \eqref{est19}.  Here is
the place where we adjust the value of the constant $C_c>0$.  Let $\kappa>0$ be a number which
will be chosen in a moment.  If
$2^{-j_1 c} |u_1| \le \kappa 2^{-j_2 c} |u_2|$, then using the fact
that $1 - \g_1(t)^c \ge 1/4$ one can check that
\begin{align}
\label{eq:22}
(\g_1(t) \g_2(t))^{c-1} 2^{-j_1} |u_1|
\le (\kappa C_c^1 2^{-cj_2}) 
\g_2(t)^{c-2} (1 - \g_1(t)^c) 2^{-j_2} |u_2|/2
\end{align}
for some absolute constant $C_c^1>0$. In the opposite case, if 
$2^{-j_1 c} |u_1| > \kappa 2^{-j_2 c} |u_2|$ we have a similar bound 
\begin{align} 
\label{eq:23}
(\g_1(t) \g_2(t))^{c-1} 2^{-j_2} |u_2|
\le(\kappa^{-1} C_c^2 2^{-cj_1}) 
\g_1(t)^{c-2} (1 - \g_2(t)^c) 2^{-j_1} |u_1|/2
\end{align}
for some absolute constant $C_c^2>0$. We now require
$\kappa^{-1} C_c^2 2^{-cj_1}\le 1$ and
$\kappa C_c^1 2^{-cj_2}\le 1$. It suffices to take $C_c>0$ such
that $\max\{C_c^1, C_c^2\}^2\le 2^{c C_c}$; and set
$\kappa:=\max\{C_c^1, C_c^2\}$ if $j_2\ge C_c$ or
$\kappa:=\max\{C_c^1, C_c^2\}^{-1}$ if $j_1\ge C_c$.  With this
choice of $\kappa$ and $C_c$ we see that \eqref{est19} is proved, since by \eqref{eq:22}, we obtain
\begin{align*} 
|b_{j,2} (x,u) |
\simeq 
2^{-j_2} |u_2| \int_0^1 
\g_2(t)^{c-2} (1 - \g_1(t)^c) 2^{-j_2} \, dt 
\simeq 
2^{-j_2 c} |u_2|
\simeq 
|2^{-j c} \circ u|,
\end{align*}
whereas by \eqref{eq:23} we conclude
\begin{align*} 
|b_{j,1} (x,u) |
\simeq 
2^{-j_1} |u_1| \int_0^1 
\g_1(t)^{c-2} (1 - \g_2(t)^c) 2^{-j_1} \, dt 
\simeq 
2^{-j_1 c} |u_1|
\simeq 
|2^{-j c}\circ u|.
\end{align*}

Observe that for any $\a\in \N^2$ and $k\in[2]$ we also have the identity
\begin{align*} 
\partial^{\a}_x b_{j,k} (x,u)  
=
2^{-j_1 \a_1 -j_2 \a_2 - j_k} 
\big( (\partial^{\a+e_k} \phi) (2^{-j}\circ (x+u) ) 
- (\partial^{\a+e_k} \phi) (2^{-j} \circ x )  \big).
\end{align*}
Proceeding analogously and using the estimate 
\begin{align*} 
|\partial_1^k \partial_2^l \phi (x)  |
\lesssim_{k,l}
x_1^{-k+c \ind_{k \ge 1}} x_2^{-l+c \ind_{l \ge 1}},
\qquad x_1,x_2 > 0, \quad x_1^c + x_2^c \le 3/4,
\end{align*}
we can easily show that for every  $\a \in \N^2$ we have
\begin{align} \label{est20}
| \partial^{\a}_x b_{j} (x,u)  |
\lesssim_{\a} |2^{-jc}\circ u|, 
\qquad j_1, j_2 \ge 0, \quad (x,u) \in \supp \Psi_j.
\end{align}
We now show that for any $\a \in \N^2$ the following estimate holds
\begin{align} \label{est21}
| \partial^{\a}_x a_{j} (x,u)  |
\lesssim_{\a} |2^{-jc} \circ u|^{-1}, 
\qquad j_1 \vee j_2 \ge C_c, \quad (x,u) \in \supp \Psi_j.
\end{align}
Using the identity
\begin{align*} 
\partial_x^{\a} a_{j} (x,u)
=
\sum_{\b \le \a} \binom{\a}{\b} \partial^{\a - \b}_x b_{j} (x,u)
\partial^{\b}_x \Big( \frac{1}{|b_{j} (x,u)|^2} \Big) 
\end{align*}
and estimate \eqref{est20} the task is reduced to showing that for a fixed $\a \in \N^2$ we have 
\begin{align*} 
\Big|
\partial^{\a}_x \Big( \frac{1}{|b_{j} (x,u)|^2} \Big) 
\Big|
\lesssim_{\a} |2^{-jc}\circ u|^{-2}, 
\qquad j_1 \vee j_2 \ge C_c, \quad (x,u) \in \supp \Psi_j.
\end{align*}
This, however, follows from estimates \eqref{est19}, \eqref{est20} and the fact that 
$\partial^{\a}_x ( |b_{j} (x,u)|^{-2} )$
is a linear combination of the terms of the form
\begin{align*} 
|b_{j} (x,u)|^{-2(1 + \sum_{\b \le \a } p_{\b} )}  
\prod_{\b \le \a } \big(\partial^{\b}_x |b_{j} (x,u)|^2 \big)^{p_{\b}},
\end{align*}
where the indices run over the set of all non-negative integers
$p_{\b}$, $\b \in \N^2$, $\b \le \a$, satisfying
$\sum_{\b \le \a } p_{\b} \le |\a|$.  The latter can be easily
justified by using the induction argument. Now by \eqref{est18} and
\eqref{est21} we see that \eqref{red12} will be proven if we show that
for every fixed $\a \in \N^2$ one has
\begin{align*} 
|\partial^{\a}_x \Psi_j(x,u)| 
\lesssim_{\a} 1,
\qquad j_1,j_2 \ge 0,
\end{align*}
which is a straightforward consequence of the Leibniz rule and the estimates \eqref{38.1}.
The proof of \eqref{red12} and consequently  Lemma~\ref{lem:FTest} is completed.
\end{proof}

As a simple consequence of the previous lemma we obtain 
Corollary \ref{cor:FTest}, which will be used in the proof of Theorem \ref{thm:erg}.

\begin{cor} \label{cor:FTest}
Assume that $c \in (1,2)$ is fixed and
let $\Upsilon_c \colon \R \to \R$ be defined as follows
\begin{align}  \label{def:Up}
\Upsilon_c (x) := (|x|^c)' := c \, \sgn(x) |x|^{c-1}, \qquad x \in \R.
\end{align} 
Further, for $\alpha \in \{0,1\}^3$ let us define $h_{c,\alpha} \colon \R^3 \to \R$ by setting
\begin{align}  \label{def:hca}
h_{c,\alpha} (x) := \prod_{j=1}^3 \Upsilon_c (x_j)^{\alpha_j}, \qquad x \in \R^3.
\end{align} 
Then for every  $h \in C^\infty (\R^3)$ the function
$\R^3 \ni x \mapsto h(x) h_{c,\alpha} (x)$ satisfies the assumptions
of Lemma~\ref{lem:FTest}, and consequently we have
\begin{align} \label{II.2}
\Big| \int_{\mathbb S_c^2} h(w) h_{c,\alpha} (w) e(-w \cdot \xi) \, d\mu_c(w)  \Big|
\lesssim
(1 + |\xi|)^{-1}, \qquad \xi \in \R^3.
\end{align}
\end{cor}

We now close this section by giving a proof of Theorem
\ref{thm:contsph}, where the $L^p(\R^3)$ bounds for $r$-variations
corresponding to the spherical averages $\mathcal A_t^c$ are studied.

\begin{proof}[Proof of Theorem \ref{thm:contsph}]
If $p=\infty$ there is nothing to prove and the estimates
\eqref{est17} and \eqref{eq:26} follow in this case.  We now prove \eqref{eq:24}
for $p\in (3/2, 4)$, and \eqref{eq:25} for $p\in(1, \infty)$ with any
$r>2$. These in turn will imply maximal estimates \eqref{est17} and \eqref{eq:26} in the same ranges of $p$,
since $r$-variations dominate the maximal functions. In order, to
obtain \eqref{est17} for all $p\in(3/2, \infty]$ we will simply interpolate maximal estimate \eqref{est17}  for $p\in(3/2, 4)$  with
\eqref{est17} at $p=\infty$. To prove \eqref{eq:24} and \eqref{eq:25} we observe that for all $\xi \in \R^3$ we have
\begin{align}
\label{eq:27}
|\mathcal{F}_{\R^3} \mu_c(\xi)-\mathcal{F}_{\R^3} \mu_c(0)|&\lesssim |\xi|,\\
\label{eq:28}
|\mathcal{F}_{\R^3} \mu_c(\xi)|&\lesssim (1+|\xi|)^{-1},\\
\label{eq:29}
|\langle\xi, \nabla \mathcal{F}_{\R^3}\mu_c (\xi)\rangle |&\lesssim 1,
\end{align}
where \eqref{eq:28} and \eqref{eq:29} immediately follow from Lemma
\ref{lem:FTest}, see also \eqref{eq:fourier}.  Now using positivity of
the operators $\mathcal A_t^c$ and combining \eqref{eq:27},
\eqref{eq:28} and \eqref{eq:29} with the standard
Littlewood--Paley theory and appealing to \cite[Theorem 2.14,
p. 537]{MSZ1} and \cite[Theorem 2.39, p. 543]{MSZ1} we obtain
\eqref{eq:24} for $p\in (3/2, 4)$ and $r>2$.  To prove \eqref{eq:25}
for $p\in (1, \infty)$ and $r>2$ we proceed much the same way as
before, but we only use \eqref{eq:27} and
\eqref{eq:28} with the standard
Littlewood--Paley theory and invoke \cite[Theorem 2.14, p. 537]{MSZ1}.

We now assume that $p\in[1, 3/2]$ and we show that estimate
\eqref{est17} does not hold in this range. Observe that 
\begin{equation}
\label{eq:equiv}
\sup_{t>0}\mathcal A_t^c f(x)= \mu_c(\mathbb S_c^2)\mathbf{A}_c f(x), \qquad x\in \R^3,\quad f\in C_c^\infty (\R^3), \quad f\ge0;
\end{equation}
where $\mathbf{A}_c$ is a maximal function over annuli defined as
\begin{equation*}
\mathbf{A}_c f(x):=\sup_{0<a<b}\frac{1}{|\mathbb A_{a,b}|}\Big|\int_{\mathbb A_{a,b}}f(x-y) dy\Big|,
\end{equation*}
with $\mathbb A_{a,b}:=\{y\in\R^3:a<|y|_c<b\}$.
To prove \eqref{eq:equiv} note first that if $f\in C_c^\infty (\R^3)$ and $f\ge0$, then for any $r>0$ and $x\in\R^3$ we have
\begin{align*}
\frac{1}{ \mu_c\left(\mathbb S_c^2\right)} \int_{\mathbb S_c^2} f(x -r y) \, d\mu_c(y)
=\lim_{h\rightarrow 0^+} \frac{1}{|\mathbb A_{r-h,r+h}|}\int_{\mathbb A_{r-h,r+h}}f(x-y) dy\le \mathbf{A}_c f(x),
\end{align*}
therefore taking supremum over $r>0$ we conclude $\frac{1}{ \mu_c\left(\mathbb S_c^2\right)} \sup_{t>0}\mathcal A_t^c f(x)\le \mathbf{A}_c f(x).$ 

On the other hand, using \eqref{polar} we see for any $0<a<b$ and $x\in\R^3$ that
\begin{align*}
\frac{1}{|\mathbb A_{a,b}|}\int_{\mathbb A_{a,b}}f(x-y) dy
&=\frac{1}{|\mathbb A_{a,b}|}\int_a^{b} r^2 \int_{\mathbb S_c^2}f(x - r y) \, d\mu_c(y) \, dr\\
&\le \frac{1}{|\mathbb A_{a,b}|}\int_a^{b} r^2 \, dr\, \sup_{t>0}\mathcal A_t^c f(x)= \frac{1}{ \mu_c\left(\mathbb S_c^2\right)} \sup_{t>0}\mathcal A_t^c f(x),
\end{align*}
and taking supremum over all annuli proves \eqref{eq:equiv}.

Consider $f\in C_c^\infty (\R^3)$ with the property $\mathbbm{1}_{\{y\in\R^3:|y|_c<1/2\}}(x)\le f(x)\le\mathbbm{1}_{\{y\in\R^3:|y|_c<1\}}(x)$. Then for $x\in\R^3$ such that $|x|_c>100$ we have
\begin{align*}
\mathbf{A}_c f(x)\ge\frac{1}{|\mathbb A_{|x|_c-1,|x|_c+1}|}\int_{\mathbb A_{|x|_c-1,|x|_c+1}}f(x-y) dy
\simeq \frac{|\{y\in\R^3: |y|_c<1/2\}|}{|\mathbb A_{|x|_c-1,|x|_c+1}|}\simeq |x|_c^{-2}.
\end{align*}
Therefore using $\eqref{polar}$ we get an estimate
\begin{align*}
\big\| \mathbf{A}_c f\big\|_{L^p(\R^3)}^p
\gtrsim \int_{\{y\in\R^3:|x|_c>100\}}|x|_c^{-2p} dx\simeq \int_{100}^\infty r^{-2p+2} dr
\end{align*}
and the last integral converges if and only if $p>3/2$.
\end{proof}

%%%%%%%%%%%%%%%%%%%%%%%%%%%%%%%%%%%%%%%%%%%%%%%%%%%%%%%%%%%%%%%%%%%%%%%%%%%%%%%%%%%
\section{Proof of Theorem \ref{thm:erg}. Ergodic theorems}\label{sec:erg}
%%%%%%%%%%%%%%%%%%%%%%%%%%%%%%%%%%%%%%%%%%%%%%%%%%%%%%%%%%%%%%%%%%%%%%%%%%%%%%%%%%%

We now establish pointwise convergence (ii) from Theorem
\ref{thm:erg}. We fix $c\in(1, 11/10)$ and we shall abbreviate
$A_{\lambda}^c$ to $A_{\lambda}$. We may assume that $f \in L^1(X) \cap
L^\infty(X)$, since from Theorem \ref{thm:L2} and the Calder{\'o}n transference principle \cite{B3} we deduce that for every $f\in
L^p(X)$ we have
\begin{align}
\label{eq:16}
\big\|\sup_{\lambda\in\Z_+}|A_{\lambda} f|\big\|_{L^p(X)}\lesssim \|f\|_{L^p(X)} \quad\text{ if } \quad (11-4c)/(11-7c) < p \le\8;
\end{align}
and
\begin{align}
\label{eq:17}
\big\|\sup_{n\in\N}|A_{2^n} f|\big\|_{L^p(X)}\lesssim \|f\|_{L^p(X)} \quad\text{ if } \quad 1 < p \le\8.
\end{align}
Using $\sigma_{\lambda}^{\mathfrak M}$ from \eqref{eq:18} and $\sigma_{\lambda}^{\mathfrak m}$ from \eqref{eq:19} we may write
$A_{\lambda}f(x):=A_{\lambda}^{\mathfrak M}f(x)+A_{\lambda}^{\mathfrak m}f(x)$, where
\begin{align*}
A_{\lambda}^{\mathfrak M}f(x):=\frac{1}{r_c(\lambda)}\sum_{n\in\Z^3}\sigma_{\lambda}^{\mathfrak M}(n)f(T^nx),\\
A_{\lambda}^{\mathfrak m}f(x):=\frac{1}{r_c(\lambda)}\sum_{n\in\Z^3}\sigma_{\lambda}^{\mathfrak m}(n)f(T^nx).
\end{align*}
We now show that both limits
$\lim_{\lambda\to\infty}A_{\lambda}^{\mathfrak M}f(x)$ and
$\lim_{\lambda\to\infty}A_{\lambda}^{\mathfrak m}f(x)$ exist almost
everywhere for any $f \in L^1(X) \cap L^\infty(X)$, which is dense in $L^p(X)$ for any $p\in[1, \infty)$. This in turn, in
view of \eqref{eq:16} and \eqref{eq:17}, by a standard density
argument establishes the desired claim in (ii) of Theorem
\ref{thm:erg}. The proof will be split into five steps.

\paragraph{{\bf Step 1.}}
We first show that $\lim_{\lambda\to\infty}A_{\lambda}^{\mathfrak m}f(x) = 0$ almost everywhere on $X$. Indeed, suppose for a contradiction that 
$\limsup_{\lambda \to \infty} |A_{\lambda}^{\mathfrak m} f(x)| > 0$ on a set of positive measure. Then there is $\varepsilon > 0$ such that 
\begin{align*} 
\nu \big(\{x \in X : \limsup_{\lambda \to \infty} |A_{\lambda}^{\mathfrak m} f(x)| > \varepsilon \}\big) > \varepsilon.
\end{align*} 
Since $\nu$ is $\sigma$-finite, by a standard argument, we can find a
rapidly increasing sequence $(\lambda_k)_{k \in \Z_+}$, satisfying
$\lambda_{k+1} \ge 2 \lambda_k$ and such that
\begin{align*} 
\nu \big(\{x \in X : \sup_{\lambda_k \le \lambda \le \lambda_{k+1}}|A_{\lambda}^{\mathfrak m}f (x)|  > \varepsilon \}\big) > \varepsilon, \qquad k \in \Z_+, 
\end{align*}
This implies that 
\begin{align}
\label{id:3}
K^{-1} \sum_{k=1}^K \| \sup_{\lambda_k \le \lambda \le \lambda_{k+1}}|A_{\lambda}^{\mathfrak m}f|\|_{L^2(X)}^2 > \varepsilon^3, 
\qquad K \in \Z_+.
\end{align} 
We now claim that there is $\delta > 0$ such that 
\begin{align} \label{id:2}
\| \sup_{\lambda \ge N} |A_{\lambda}^{\mathfrak m} f | \|_{L^2(X)}
\lesssim
N^{-\delta} \| f \|_{L^2(X)}, \qquad N \in \Z_+,
\end{align} 
which contradicts \eqref{id:3}, since $\lim_{K\to\infty}K^{-1} \sum_{k=1}^K\lambda_k^{-\delta}=0$. In order to justify \eqref{id:2} it suffices to prove that 
\begin{align*} 
\| \sup_{N \le \lambda \le 2N} |A_{\lambda}^{\mathfrak m} f | \|_{L^2(X)}
\lesssim
N^{-\delta} \| f \|_{L^2(X)}, \qquad N \in \Z_+.
\end{align*}
This, however, is a consequence of Theorem~\ref{thm:minor}, the fact that 
$r_c(\lambda) \simeq \lambda^{3/c - 1}$, upon invoking the Calder\'on transference principle \cite{B3}. 

\paragraph{{\bf Step 2.}}
Now we split the operator $A_{\lambda}^{\mathfrak M} f$. Let $\mathfrak n_c:=\big(\frac{c}{2}\big)^3\frac{\Gamma(3/c)}{\Gamma(1/c)^3 }$, and observe that
\begin{align*}
A_{\lambda}^{\mathfrak M} f(x)
:= 
B^1_{\lambda} f (x)
+
\mathfrak n_cB^2_{\lambda} f (x),
\end{align*}
where
\begin{align*}
B^j_{\lambda} f (x)
&:= 
\sum_{n \in \Z^3} K_{\lambda}^j (n)
f (T^n x), \qquad \lambda \in \Z_+, \quad j \in [2], 
\end{align*}
and for $\lambda \in \Z_+$, $n \in \Z^3$, we put
\begin{align*}
K_{\lambda}^1 (n)
&:= 
\lambda^{\kappa} \eta \Big( \frac{n}{\lambda^{1/c}} \Big) 
\bigg[ 
\frac{1}{r_c(\lambda)} \mathcal{F}^{-1}_{\R} \psi \big( \lambda^{\kappa} (Q (n) - \lambda) \big) \\
& \qquad \qquad \qquad \qquad \qquad \qquad 
-
\mathfrak n_c
\lambda^{-3/c + 1}
\mathcal{F}^{-1}_{\R} \psi \big( \lambda^{\kappa} (|n|_c^c - \lambda) \big)
\bigg],\\
K_{\lambda}^2 (n)
&:= 
\lambda^{-3/c + 1}
\lambda^{\kappa} \eta \Big( \frac{n}{\lambda^{1/c}} \Big) 
\mathcal{F}^{-1}_{\R} \psi \big( \lambda^{\kappa} ( |n|_c^c - \lambda) \big).
\end{align*}
We show that for each $j \in [2]$ the limit
$\lim_{\lambda\to\infty}B^j_{\lambda} f (x)$ exists almost everywhere
on $X$. For this purpose, we prove that for sufficiently large $r>2$ we have
\begin{align}
\label{eq:20}
\big\| V^r(B_{\lambda}^1 f : \lambda \in \Z_+) \big\|_{L^r(X)}\lesssim \|f\|_{L^r(X)},
\end{align}
and we also prove that for every $r>2$ we have
\begin{align}
\label{eq:21}
\big\| V^r(B_{\lambda}^2 f : \lambda \in \Z_+) \big\|_{L^2(X)}\lesssim \|f\|_{L^2(X)}.
\end{align}
Once \eqref{eq:20} and \eqref{eq:21} are established the proof of (ii)
in Theorem \ref{thm:erg} follows, since variational seminorms imply
pointwise almost everywhere convergence for the underlying sequences.
These estimates for \eqref{eq:20} and \eqref{eq:21} will be proved in
the next two steps.

\paragraph{{\bf Step 3.}}
To prove \eqref{eq:20}, by the Calder\'on transference principle \cite{B3}, it suffices to show that 
for sufficiently large $r$ we have 
\begin{align*} 
\big\| V^r(K_{\lambda}^1 * g : \lambda \in \Z_+) 
\big\|_{\ell^r(\Z^3)}
\lesssim
\| g \|_{\ell^r(\Z^3)}, \qquad g \in \ell^r(\Z^3).
\end{align*}
Since we have $V^r(K_{\lambda}^1 * g (m) : \lambda \in \Z_+) \lesssim
\| K_{\lambda}^1 * g (m) \|_{\ell^r(\lambda)}$, $m \in \Z^3$, we conclude
\begin{align*} 
\big\| V^r(K_{\lambda}^1 * g : \lambda \in \Z_+) 
\big\|_{\ell^r(\Z^3)}
& \lesssim
\Big(
\sum_{\lambda \in \Z_+}  \| K_{\lambda}^1 * g \|_{\ell^r(\Z^3)}^r
\Big)^{1/r} \\
& \lesssim 
\Big(
\sum_{\lambda \in \Z_+}  \| K_{\lambda}^1 \|_{\ell^1(\Z^3)}^r
\Big)^{1/r}
\| g \|_{\ell^r(\Z^3)}, \qquad g \in \ell^r(\Z^3).
\end{align*}
Therefore the task is reduced to showing that there is $\delta> 0$ such that
\begin{align} \label{id:5}
\| K_{\lambda}^1 \|_{\ell^1(\Z^3)}
\lesssim
\lambda^{-\delta}, \qquad \lambda \in \Z_+.
\end{align}
Using Corollary~\ref{cor:asym} and the mean value theorem  we may write
\begin{align} \label{id:6}
|K_{\lambda}^1 (n)|
\lesssim_{\eps}
\lambda^{\kappa}
\lambda^{-3/c + 1 - (9 - 8c)/(5c) + \eps}
\omega_{\lambda}(n)
+
\lambda^{2\kappa}
\lambda^{-3/c + 1}
\omega_{\lambda}(n),
\end{align}
for every $\eps > 0$, 
where $\omega_{\lambda}$ is defined in \eqref{def:omega}. In the above estimate we use the fact that $|Q(n) - |n|_c^c| \le 3$ and the estimate 
\begin{align*} 
1 + \lambda^{\kappa}|\theta - \lambda|
\simeq
1 + \lambda^{\kappa}||n|_c^c - \lambda|,
\end{align*}
uniformly in $\lambda \ge 1$, $n \in \Z^3$ for $\theta$ being a convex combination of $Q(n)$ and $|n|_c^c$.
Further, using Lemma~\ref{kernel}, polar decomposition \eqref{polar} and then changing the variable $r \mapsto \lambda^{1/c} t$ we obtain uniformly in $\lambda \ge 1$ that
\begin{align} \nonumber
\sum_{n \in \Z^3} \omega_{\lambda}(n)
& \lesssim
\int_{\R^3} \omega_{\lambda}(x) \, dx
\lesssim 
\int_0^\infty \frac{r^2 dr}{1 +  \big(\lambda^{\kappa} | r^c - \lambda|\big)^{10}} \\ \nonumber
& =
\lambda^{3/c} \int_0^\infty \frac{t^2 dt}{1 +  \big(\lambda^{\kappa+1} | t^c - 1|\big)^{10}} \\ \nonumber
& \lesssim
\lambda^{3/c} \int_0^\infty \frac{dt}{1 +  \big(\lambda^{\kappa+1} | t - 1|\big)^{8}} \\ \label{id:14}
& \lesssim 
\lambda^{3/c - \kappa-1}, 
\end{align}
Combining this with \eqref{id:6} we obtain \eqref{id:5} and the pointwise convergence for $B^1_{\lambda} f$ is justified.

\paragraph{{\bf Step 4.}}
As in the previous step, to prove \eqref{eq:21}, by the Calder\'on transference principle \cite{B3}, it suffices to show that for every $r>2$ we have 
\begin{align} \label{id:11}
\big\| V^r(K_{\lambda}^2 * g : \lambda \in \Z_+) 
\big\|_{\ell^2(\Z^3)}
\lesssim
\| g \|_{\ell^2(\Z^3)}, \qquad g \in \ell^2(\Z^3).
\end{align}
Let us define 
\begin{align*} 
K_{\lambda}^3 (x)
:= 
\lambda^{-3/c + 1}
\lambda^{\kappa} 
\mathcal{F}^{-1}_{\R} \psi \big( \lambda^{\kappa} ( |x|_c^c - \lambda) \big), \qquad x \in \R^3, \quad \lambda \ge 1,
\end{align*} 
and observe that to prove \eqref{id:11} it suffices to show that 
\begin{align} \label{id:12} 
\sum_{\lambda \in \Z_+}
\| K_{\lambda}^2 - K_{\lambda}^3 \|_{\ell^1(\Z^3)}
& < \infty, \\ \label{id:13}
\big\| V^r(K_{\lambda}^3 * g : \lambda \in \Z_+) 
\big\|_{\ell^2(\Z^3)}
& \lesssim
\| g \|_{\ell^2(\Z^3)}, \qquad g \in \ell^2(\Z^3).
\end{align} 
For  \eqref{id:12} notice that for every fixed  $N \in \Z_+$ we have uniformly in $\lambda \ge 2$ that
\begin{align*} 
\| K_{\lambda}^2 - K_{\lambda}^3 \|_{\ell^1(\Z^3)}
& \lesssim_N
\sum_{|n|_c \ge 4^{1/c}  \lambda^{1/c} } 
\frac{\lambda^{-3/c + 1}
\lambda^{\kappa} }{1 + \big( \lambda^{\kappa} | |n|_c^c - \lambda| \big)^{N}} \\
& \lesssim_N
\sum_{|n|_c \ge 4^{1/c}  \lambda^{1/c} } 
\frac{\lambda^{-3/c + 1}
\lambda^{\kappa} }{1 + \big( \lambda^{\kappa}  |n|_c^c \big)^{N}} \\
& \lesssim_N
\lambda^{-3/c + 1}
\lambda^{\kappa} \lambda^{-3N/(8c)}  \lambda^{-3\kappa/c}.
\end{align*}
Taking sufficiently large $N$ we obtain \eqref{id:12}. 
To prove \eqref{id:13} we will proceed in a similar way as in \cite{MSZ1}. 
For the reader convenience we will give the details. 
For this purpose, in this step we show that the following
estimates hold 
\begin{align} \label{id:7}
\| K_{\lambda}^3 \|_{\ell^1(\Z^3)}
& \lesssim 1, \qquad \lambda \in \Z_+, \\ \label{id:8}
\big| \mathcal{F}_{\Z^3}^{-1} K_{\lambda}^3 (\xi) \big|
& \lesssim (\lambda^{1/c} \| \xi \|)^{-1}, 
\qquad \lambda \in \Z_+, \quad \xi \in \T^3, \\ \label{id:9}
\big| \mathcal{F}_{\Z^3}^{-1} K_{\lambda}^3 (\xi) 
- \mathcal{F}_{\Z^3}^{-1} K_{\lambda}^3 (0)\big|
& \lesssim \lambda^{1/c} \| \xi \|, 
\qquad \lambda \in \Z_+, \quad \xi \in \T^3, \\ \label{id:10}
\big| \mathcal{F}_{\Z^3}^{-1} K_{\lambda + s}^3 (\xi) 
- \mathcal{F}_{\Z^3}^{-1} K_{\lambda}^3 (\xi) \big|
& \lesssim (s/\lambda)^{3/4}, 
\qquad 0 \le s \le \lambda, 
\quad \lambda \in \Z_+, \quad \xi \in \T^3, \\ \label{id:16}
V^2(\mathcal{F}_{\Z^3}^{-1} K_{\lambda}^3 (0) : \lambda \in \Z_+) 
& < \infty.
\end{align}
In fact the exponent $3/4$ in \eqref{id:10} can be replaced with any
$\delta\in(1/2, 1)$. 
The estimate
\eqref{id:7} is a direct consequence of \eqref{id:14}, since
\[
|\mathcal{F}^{-1}_{\R} \psi \big( \lambda^{\kappa} ( |n|_c^c - \lambda) \big)| \lesssim \omega_{\lambda}(n).
\]
In order to prove the remaining estimates  we may assume, without any loss of generality,
that $\T^3 = [-1/2,1/2)^3$. Further, we will use the
Poisson summation formula (its application is possible thanks to the
rapid decay of $K_{\lambda}^3 (x)$ in $x\in \R^3$ and the estimate
\eqref{I.1} below), which in our case says that
\begin{align} \label{I.0}
\mathcal{F}_{\Z^3}^{-1} K_{\lambda}^3 (\xi)
=
\sum_{n \in \Z^3} I_{\lambda} (n, \xi), 
\qquad \xi \in \T^3, \quad \lambda \ge 1,
\end{align} 
where
\begin{align*} 
I_{\lambda} (n, \xi) 
:=
\int_{\R^3} K_{\lambda}^3 (x) e\big( x \cdot (\xi - n) \big) \, dx, 
\qquad n \in \Z^3, \quad \xi \in \T^3, \quad \lambda \ge 1.
\end{align*} 
\paragraph{{\bf Estimate \eqref{id:8}}}
By \eqref{I.0} inequality \eqref{id:8} is a simple consequence of the following bound
\begin{align} \label{I.1}
|I_{\lambda} (n, \xi) |
\lesssim 
\Big( \prod_{j=1}^3 \frac{1}{1 + |n_j|}  \Big) 
\frac{1}{1 + \lambda^{1/c} |\xi - n|},
\qquad \lambda \ge 1, \quad \xi \in \T^3, \quad n \in \Z^3.
\end{align} 
To justify \eqref{I.1} we need to introduce the functions
$h_{c,\alpha}$ with $\alpha \in \{0,1\}^3$ as in \eqref{def:hca}. Notice
that $h_{c,\alpha}$ are homogeneous of order $(c-1)|\a|$, where
$|\alpha| : = \alpha_1 + \alpha_2 + \alpha_3$. More precisely, we have
\begin{align*} 
h_{c,\alpha} (rw) = r^{(c-1)|\a|} h_{c,\alpha} (w), 
\qquad r > 0, \quad w \in \R^3.
\end{align*} 
We define $\a = \a(n) \in \{0,1\}^3$ in a such way that for every
$j \in [3]$ we put $\a_j = 0$ if $n_j = 0$ and $\a_j = 1$ if
$n_j \ne 0$.  Further, for $\kappa$ defined in \eqref{eq:46} we notice
that $-1< \kappa < 0$ and $\kappa + (c-1)/c < 0$, which will be
frequently used below.

We now prove \eqref{I.1}. Integrating by parts we obtain
\begin{align*} 
I_{\lambda} (n, \xi)
=
 \frac{C_{\alpha}}{(\xi - n)^{\a}}
\int_{\R^3} \partial_x^{\a} K_{\lambda}^3 (x) e\big( x \cdot (\xi - n) \big) \, dx
\end{align*} 
for some constants $C_{\alpha}\in\C$. 
Further, taking $\psi_{\a} := (\mathcal{F}^{-1}_{\R} \psi)^{(|\a|)}$ observe that 
\begin{align} \label{II.1}
\partial_x^{\a} K_{\lambda}^3 (x) 
=
\lambda^{-3/c + 1 + \kappa + \kappa |\a|} 
\psi_{\a} \big( \lambda^{\kappa} ( |x|_c^c - \lambda) \big)
h_{c,\alpha} (x).
\end{align} 
Therefore using \eqref{polar} and changing the variable $r \mapsto \lambda^{1/c} t$ we obtain 
\begin{align} \nonumber
I_{\lambda} (n, \xi)
& =
 \frac{C_{\alpha}}{(\xi - n)^{\a}} \lambda^{\kappa + 1 + |\a| (c-1)/c  + \kappa |\a|}
\int_{0}^{\infty} t^{2 + (c-1)|\a|} \psi_{\a} \big( \lambda^{\kappa+1} ( t^c - 1) \big) \\ \label{II.3}
& \qquad \qquad \qquad \times
\int_{\mathbb S_c^2}
h_{c,\alpha} (w)
e\big( \lambda^{1/c} t w \cdot (\xi - n) \big) \, d\mu_c(w) \, dt.
\end{align} 
Therefore using \eqref{II.2} and the estimate
$|t - 1| \lesssim |t^c - 1|$ for $t > 0$, we see that for every
fixed $N \in \N$ we have
\begin{align*} 
|I_{\lambda} (n, \xi)|
& \lesssim_N
\Big( \prod_{j=1}^3 \frac{1}{1 + |n_j|}  \Big)
\lambda^{\kappa + 1}
\int_{0}^{\infty} 
\frac{t^{2 + (c-1)|\a|} }{1 + \big( \lambda^{\kappa+1} | t - 1| \big)^{N}} 
\frac{dt}{1 + \lambda^{1/c} t |\xi - n|}\\
& \lesssim_N
\Big( \prod_{j=1}^3 \frac{1}{1 + |n_j|}  \Big)
\lambda^{\kappa + 1}
\int_{0}^{\infty}
\frac{1}{1 + \big( \lambda^{\kappa+1} | t - 1| \big)^{N - 1 - (c-1)|\a|}} 
\frac{t\, dt}{1 + \lambda^{1/c} t |\xi - n|}.
\end{align*} 
Now splitting the above integral into the intervals $(0,1/2)$ and $(1/2,\infty)$ and using the relation
\begin{align*} 
\int_{0}^{1/2} \frac{t\, dt}{1 + A t } 
\simeq
\frac{1}{1 + A }, \qquad A > 0,
\end{align*} 
we obtain \eqref{I.1} and consequently \eqref{id:8}.

\paragraph{{\bf Estimate \eqref{id:9}}} Let
$J_{\lambda} (n, \xi) := I_{\lambda} (n, \xi) - I_{\lambda} (n, 0)$ and note
that \eqref{id:9} is reduced to proving
\begin{align} \label{V.1}
|J_{\lambda} (n, \xi)|
\lesssim
\Big( \prod_{j=1}^3 \frac{1}{1 + |n_j|}  \Big) 
\frac{\lambda^{1/c} \|\xi\|}{1 + | n|},
\qquad \lambda \ge 1, \quad \xi \in \T^3, \quad n \in \Z^3.
\end{align} 
Integrating by parts, using the Leibniz rule, \eqref{II.1}, \eqref{polar} and finally changing the variable $r \mapsto \lambda^{1/c} t$ we obtain
for some constants $C_{\alpha}, C_{\alpha, \beta}\in\C$ that
\begin{align*} 
J_{\lambda} (n, \xi)
& =
\frac{C_{\alpha} }{n^{\a}}
\int_{\R^3} e( - x \cdot n ) \partial_x^{\a} 
\big[ K_{\lambda} (x) \big( e\big( x \cdot \xi  \big) - 1 \big) \big] \, dx \\
& =
 \sum_{0 \le \beta \le \a} \frac{C_{\alpha, \beta}}{n^{\a}}
\int_{\R^3} e( - x \cdot n )
\partial_x^{\beta} K_{\lambda} (x) 
\partial_x^{\a - \beta} \big( e\big( x \cdot \xi  \big) - 1 \big) \, dx \\
& =
\sum_{0 \le \beta \le \a} \frac{C_{\alpha, \beta}}{n^{\a}} \lambda^{-3/c + 1 + \kappa + \kappa |\beta|} 
\int_{\R^3} e( - x \cdot n )
\psi_{\beta} \big( \lambda^{\kappa} ( |x|_c^c - \lambda) \big)
h_{c,\beta} (x) \xi^{\a - \beta} \\
& \qquad \qquad \qquad \qquad \qquad \times
\big( e( x \cdot \xi ) - \ind_{\a = \beta} \big) \, dx \\
& =
\sum_{0 \le \beta \le \a} \frac{C_{\alpha, \beta}}{n^{\a}}
\lambda^{1 + \kappa + \kappa |\beta| + |\beta| (c-1)/c} 
\int_0^{\infty} t^{2+ (c-1)|\beta|}
\psi_{\beta} \big( \lambda^{\kappa+1} ( t^c - 1) \big)
J_{\lambda, \a, \beta} (n, \xi, t) \, dt,
\end{align*} 
where
\begin{align*} 
J_{\lambda, \a, \beta} (n, \xi, t)
=
\xi^{\a - \beta}
\int_{\mathbb S_c^2}
h_{c,\beta} (w)
e( - \lambda^{1/c} t w \cdot n ) 
\big( e\big( \lambda^{1/c} t w \cdot \xi \big)  - \ind_{\a = \beta} \big) \, d\mu_c(w).
\end{align*}  
Next, we show that 
\begin{align} \label{id:15}
|J_{\lambda, \a, \beta} (n, \xi, t)|
\lesssim
\frac{(t+1)^2}{t(|n| + 1)} \lambda^{1/c} \|\xi\|, 
\qquad t > 0, \quad \lambda \ge 1, \quad n \in \Z^3, \quad \xi \in \T^3.
\end{align} 
Indeed, if $\a \ne \beta$ then using \eqref{II.2} we obtain
\begin{align*} 
|J_{\lambda, \a, \beta} (n, \xi, t)|
\lesssim
\frac{\|\xi\|}{1 + \lambda^{1/c} t |\xi - n|}
\lesssim
\frac{(t+1)\|\xi\|}{t(|n| + 1)} 
\end{align*} 
and \eqref{id:15} follows. In the opposite case when $\a = \beta$ we combine the identity 
\begin{align*} 
e(w \cdot \xi) - 1 = 2\pi i \int_0^1 (w \cdot \xi) e(u w \cdot \xi) \, du, 
\qquad w,\xi \in \R^3, 
\end{align*} 
with \eqref{II.2} and we get
\begin{align*} 
|J_{\lambda, \a, \beta} (n, \xi, t)|
& \lesssim
\Big| 
\int_0^1 \int_{\mathbb S_c^2} \Big( \sum_{j=1}^3 \lambda^{1/c} t w_j \xi_j \Big)
h_{c,\beta} (w)
e\big( \lambda^{1/c} t w \cdot (u\xi - n) \big) \, d\mu_c(w) \, du \Big| \\
& \lesssim
\int_0^1 
\frac{\lambda^{1/c} t \|\xi\|}{1 + \lambda^{1/c} t |u\xi - n|} \, du
\lesssim
\frac{t+1}{|n| + 1} \lambda^{1/c} \|\xi\|.
\end{align*} 
This directly leads to \eqref{id:15}. Next, applying \eqref{id:15} for every fixed $N \in \N$  we arrive at
\begin{align*} 
|J_{\lambda} (n, \xi)|
& \lesssim
\Big( \prod_{j=1}^3 \frac{1}{1 + |n_j|}  \Big) 
\sum_{0 \le \beta \le \a} 
\lambda^{1 + \kappa } 
\int_0^{\infty}
\frac{t^{2+ (c-1)|\beta|}}{1 + \big( \lambda^{\kappa+1} | t - 1| \big)^{N}} 
\frac{(t+1)^2}{t(|n| + 1)} \lambda^{1/c} \|\xi\| \, dt \\
& \lesssim
\Big( \prod_{j=1}^3 \frac{1}{1 + |n_j|}  \Big) 
\frac{1}{|n| + 1}  \lambda^{1/c} \|\xi\|
\sum_{0 \le \beta \le \a} 
\lambda^{1 + \kappa } 
\int_0^{\infty}
\frac{dt}{1 + \big( \lambda^{\kappa+1} | t - 1| \big)^{N- 4 - (c-1)|\beta|}},
\end{align*} 
which gives \eqref{V.1} and consequently \eqref{id:9} is justified.

\paragraph{{\bf Estimate \eqref{id:10}}} We begin with showing that
\begin{align} \label{V.2}
|\partial_{\lambda} I_{\lambda} (n, \xi) |
\lesssim
\lambda^{-1}
\Big( \prod_{j=1}^3 \frac{1}{1 + |n_j|}  \Big),
\qquad \lambda \ge 1, \quad \xi \in \T^3, \quad n \in \Z^3.
\end{align} 
Taking into account \eqref{II.3} we obtain
\begin{align*} 
\partial_{\lambda} I_{\lambda} (n, \xi)
& =
\frac{1}{(\xi - n)^{\a}} \lambda^{\kappa + 1 + |\a| (c-1)/c  + \kappa |\a|}
\int_{0}^{\infty} t^{2 + (c-1)|\a|} \int_{\mathbb S_c^2}
h_{c,\alpha} (w) e\big( \lambda^{1/c} t w \cdot (\xi - n) \big) \\ 
& \qquad \qquad \qquad \times
\Big[
C_{\alpha,1} \lambda^{-1} \psi_{\a} \big( \lambda^{\kappa+1} ( t^c - 1) \big)
+
C_{\alpha,2}
\psi'_{\a} \big( \lambda^{\kappa+1} ( t^c - 1) \big) \lambda^{\kappa} ( t^c - 1) \\
& \qquad \qquad \qquad \qquad \qquad +
C_{\alpha,3} \psi_{\a} \big( \lambda^{\kappa+1} ( t^c - 1) \big) 
\lambda^{1/c - 1} t w \cdot (\xi - n)
\Big]
 \, d\mu_c(w) \, dt,
\end{align*} 
for some constants $C_{\alpha,j}\in \C$ for $j \in [3]$.
Using \eqref{II.2} for every fixed $N \in \N$ we obtain
\begin{align*} 
|\partial_{\lambda} I_{\lambda} (n, \xi) |
& \lesssim
\Big( \prod_{j=1}^3 \frac{1}{1 + |n_j|}  \Big) \lambda^{\kappa + 1}
\int_{0}^{\infty} t^{2 + (c-1)|\a|} \\ 
& \qquad \qquad \qquad \times
\bigg(
\frac{\lambda^{-1}}{1 + \big( \lambda^{\kappa+1} | t - 1| \big)^{N}} 
+
\frac{\lambda^{\kappa} | t^c - 1|}{1 + \big( \lambda^{\kappa+1} | t^c - 1| \big)^{N}}  
\bigg) \, dt \\
& \lesssim
\Big( \prod_{j=1}^3 \frac{1}{1 + |n_j|}  \Big) \lambda^{\kappa}
\int_{0}^{\infty} 
\frac{t^{2 + (c-1)|\a|} }{1 + \big( \lambda^{\kappa+1} | t - 1| \big)^{N-1}} \, dt \\
& \lesssim
\Big( \prod_{j=1}^3 \frac{1}{1 + |n_j|}  \Big) \lambda^{\kappa}
\int_{0}^{\infty} 
\frac{dt}{1 + \big( \lambda^{\kappa+1} | t - 1| \big)^{N-3 - (c-1)|\a| }}.
\end{align*} 
Therefore we obtain \eqref{V.2} as claimed. Consequently, using \eqref{V.2} we infer that
\begin{align*} 
|I_{\lambda+s} (n, \xi) - I_{\lambda} (n, \xi)|
& =
\Big| \int_{\lambda}^{\lambda + s} \partial_{t} I_{t} (n, \xi) \, dt \Big| \\
& \lesssim
\Big( \prod_{j=1}^3 \frac{1}{1 + |n_j|}  \Big) s/\lambda, 
\qquad 0 \le s \le \lambda, 
\quad \lambda \ge 1, \quad \xi \in \T^3, \quad n \in \Z^3.
\end{align*} 
On the other hand, from \eqref{I.1} we conclude
\begin{align*}
|I_{\lambda+s} (n, \xi) - I_{\lambda} (n, \xi)|
\lesssim 
\Big( \prod_{j=1}^3 \frac{1}{1 + |n_j|}  \Big) 
\frac{1}{1 + |n|},
\qquad 0 \le s \le \lambda, \quad \lambda \ge 1, \quad \xi \in \T^3, \quad n \in \Z^3.
\end{align*} 
Therefore fixing $\delta \in [0,1]$ and combining these estimates  we have
\begin{align*}
|I_{\lambda+s} (n, \xi) - I_{\lambda} (n, \xi)|
\lesssim_{\delta}
\Big( \prod_{j=1}^3 \frac{1}{1 + |n_j|}  \Big) 
\Big(\frac{1}{1 + |n|} \Big)^{1-\delta} (s/\lambda)^\delta,
\end{align*} 
uniformly in $0 \le s \le \lambda$, $\lambda \ge 1$, $\xi \in \T^3$
and $n \in \Z^3$. This, in view of \eqref{I.0}, proves
\eqref{id:10}. 

\paragraph{{\bf Estimate \eqref{id:16}}}
Since $\sum_{\lambda \in \Z_+} \lambda^{-2(\kappa + 1)} < \infty$,
it is enough to prove that
\begin{align} \label{VIII.23}
\mathcal{F}_{\Z^3}^{-1} K_{\lambda}^3 (0)
=
\mu_c(\mathbb S_c^2)/c + O(\lambda^{-\kappa - 1}), \qquad \lambda \ge 1.
\end{align} 
Taking into account \eqref{I.0} and \eqref{I.1} our aim reduces to proving that
\begin{align} \label{id:17} 
I_{\lambda} (0,0) = \mu_c(\mathbb S_c^2)/c + O(\lambda^{-\kappa - 1}), \qquad \lambda \ge 1.
\end{align} 
Using \eqref{II.3} and changing the variable $t^c - 1 \mapsto s$ we get 
\begin{align*} 
I_{\lambda} (0, 0)
& =
\mu_c(\mathbb S_c^2) c^{-1}
\lambda^{\kappa + 1}
\int_{-1}^{\infty} (s+1)^{3/c - 1} 
\mathcal{F}^{-1}_{\R} \psi ( \lambda^{\kappa+1} s) 
\, ds \\ 
& =
\mu_c(\mathbb S_c^2) c^{-1}
\lambda^{\kappa + 1}
\int_{-1/2}^{1/2} (s+1)^{3/c - 1} 
\mathcal{F}^{-1}_{\R} \psi ( \lambda^{\kappa+1} s) 
\, ds + O(\lambda^{-100}) \\
& =
\mu_c(\mathbb S_c^2) c^{-1}
\lambda^{\kappa + 1}
\int_{-1/2}^{1/2}
\mathcal{F}^{-1}_{\R} \psi ( \lambda^{\kappa+1} s) 
\, ds 
+ O(\lambda^{-\kappa - 1}).
\end{align*} 
Since $\int_{\R} \mathcal{F}^{-1}_{\R} \psi(s) ds = \psi(0) = 1$, we conclude \eqref{id:17}.
This finishes the Step 4.

\paragraph{{\bf Step 5.}} Here using \eqref{id:7}, \eqref{id:8}, \eqref{id:9}, \eqref{id:10} and \eqref{id:16} we finally justify \eqref{id:13}. To proceed, let $(P_t)_{t>0}$ be the discrete Poisson semigroup defined by 
\begin{align*} 
\mathcal F_{\Z^3}^{-1} (P_t f) (\xi) := p_t (\xi) \mathcal F_{\Z^3}^{-1} f (\xi),
\qquad t>0, \quad \xi \in \T^3,
\end{align*} 
where 
\begin{align*} 
p_t (\xi): 
=
e^{-2\pi t |\xi|_{\rm sin}}, 
\qquad 
|\xi|_{\rm sin} := \Big( \sum_{j=1}^3 (\sin(\pi \xi_j) )^2 \Big)^{1/2},
\qquad t>0, \quad \xi \in \T^3.
\end{align*} 
Notice that $|\xi|_{\rm sin} \simeq \|\xi\|$. It was recently proved in \cite[Section 5.1, p. 893]{BMSW1} that
for every $1<p<\infty$ and $r>2$ we have the variational estimate
\begin{align*} 
\big\| V^r( P_t f : t > 0) 
\big\|_{\ell^p(\Z^3)}
& \lesssim
\| f \|_{\ell^p(\Z^3)}, \qquad f \in \ell^p(\Z^3).
\end{align*}
Using this, \eqref{id:16} and the fact that 
\begin{align*} 
V^r( a_N b_N  : N \in \Z_+)
\le 
\big( \sup_{N \in \Z_+} |a_N| \big) V^r( b_N  : N \in \Z_+)
+
\big( \sup_{N \in \Z_+} |b_N| \big) V^r( a_N  : N \in \Z_+),
\end{align*}
we see that our task reduces to showing that
\begin{align*} 
\big\| V^r(T_{\lambda} f : \lambda \in \Z_+) 
\big\|_{\ell^2(\Z^3)}
\lesssim_{r}
\| f \|_{\ell^2(\Z^3)}, \qquad f \in \ell^2(\Z^3), 
\qquad r > 2,
\end{align*}
where
\begin{align*} 
T_{\lambda} f 
:=
K_{\lambda}^3 * f - \mathcal{F}_{\Z^3}^{-1} K_{\lambda}^3 (0) P_{\lambda^{1/c}} f, 
\qquad \lambda \in \Z_+.
\end{align*}
Splitting into long and short variations we see that it suffices to prove that
\begin{align} \label{id:18}
\Big\| \Big( \sum_{k \in \Z_+} |T_{2^k} f|^2 \Big)^{1/2} 
\Big\|_{\ell^2(\Z^3)}
\lesssim
\| f \|_{\ell^2(\Z^3)}, \qquad f \in \ell^2(\Z^3), \\ \label{id:19}
\Big\| \Big( \sum_{k \in \Z_+} 
V^2(T_{\lambda} f : \lambda \in [2^k, 2^{k+1}])^{2} \Big)^{1/2}
\Big\|_{\ell^2(\Z^3)}
\lesssim
\| f \|_{\ell^2(\Z^3)}, \qquad f \in \ell^2(\Z^3).
\end{align}
Let $\mathcal{F}_{\Z^3}^{-1} T_\lambda$ denote the multiplier of the operator $T_\lambda$. 
Using \eqref{id:8} and \eqref{id:9} we see that
\begin{align} \label{id:20}
| \mathcal{F}_{\Z^3}^{-1} T_{\lambda} (\xi) |
\lesssim
(\lambda^{1/c} \| \xi \|)^{-1} \wedge (\lambda^{1/c} \| \xi \|),
\qquad \xi \in \T^3, \quad \lambda \in \Z_+,
\end{align}
and applying the Plancherel theorem we see that \eqref{id:18} follows. 
In order to justify \eqref{id:19} we use the following version of the Rademacher--Menshov inequality, see \cite[Lemma 2.5]{MSZ1},
\begin{align*} 
V^2(T_{\lambda} f : \lambda \in [2^k, 2^{k+1}])
\lesssim
\sum_{l=0}^k \Big( \sum_{m=0}^{2^l -1} 
| T_{2^k + 2^{k-l} (m+1)} f - T_{2^k + 2^{k-l} m} f |^2 \Big)^{1/2},
\end{align*}
which together with the triangle inequality shows that
\begin{align*} 
\text{LHS of } \eqref{id:19}
\lesssim
\sum_{l \ge 0} \Big( \sum_{k \ge l \vee 1 } \sum_{m=0}^{2^l -1} 
\| T_{2^k + 2^{k-l} (m+1)} f - T_{2^k + 2^{k-l} m} f \|_{\ell^2(\Z^3)}^2
\Big)^{1/2}.
\end{align*}
Combining \eqref{id:10} with \eqref{id:20} we see that uniformly in 
$0 \le l \le k$, $0 \le m \le 2^l -1$, $\xi \in \T^3$, we have
\begin{align*} 
\big| \mathcal{F}_{\Z^3}^{-1} T_{2^k + 2^{k-l} (m+1)} (\xi) 
- \mathcal{F}_{\Z^3}^{-1} T_{2^k + 2^{k-l} m} (\xi) \big|
\lesssim
(2^{k/c} \| \xi \|)^{-1} \wedge (2^{k/c} \| \xi \|) \wedge 2^{-3l/4 }, 
\end{align*}
and consequently using the Plancherel theorem we get \eqref{id:19}. 
This concludes the proof of \eqref{id:13}.
This finishes the Step 5 and the proof of
Theorem~\ref{thm:erg} is completed.
%%%%%%%%%%%%%%%%%%%%%%%%%%%%%%%%%%%%%%%%%%%%%%%%%%%%%%%%%%%%%%%%%%%%%%%%%%%%%%%%%%%%%%
\section{Proof of Theorem \ref{thm:disc}. The upper bound for the discrepancy}\label{sec:disc}
%%%%%%%%%%%%%%%%%%%%%%%%%%%%%%%%%%%%%%%%%%%%%%%%%%%%%%%%%%%%%%%%%%%%%%%%%%%%%%%%%%%%%%

To prove Theorem~\ref{thm:disc} we need some preparatory results, which will be gathered below. We first observe that 
\begin{align*}
D_c(\lambda,\xi,a)=\mathbf D_c(\lambda,\xi; \ind_{[a,100]}), \qquad \lambda \in \Z_+, \quad \xi \in \S^2, \quad a > 0,
\end{align*}
where for any function $f:\R\to\R$ we have
\begin{align*} 
\mathbf D_c(\lambda,\xi; f) := \sum_{m\in\mathbf S_c^3(\lambda)} 
f\Big( \frac{m \cdot \xi}{\lambda^{1/c}} \Big)
- r_c(\lambda) \int_{\mathbb S_c^2} 
f(x \cdot\xi) \, d\nu_{c} (x),
\end{align*}
and $\nu_{c}$ is a normalized measure $\mu_{c}$ to have mass one.

Proceeding as in \cite[Lemma 9]{Mag} we show that the function
$\mathbbm{1}_{[a,100]}$ in $\mathbf D_c(\lambda,\xi; \ind_{[a,100]})$ can be replaced by smooth
functions $\phi_{a,\delta}^{\pm}$. Indeed, let $\phi \in C^{\8}(\R)$
be such that $0 \le \phi (t) \le 1$, $\supp \phi \subset [-1,1]$ and
$\int \phi(x)dx = 1$. Further, let
$\phi_{a,\delta}^{\pm}:= \mathbbm{1}_{[a \pm \delta,100]} \ast \phi_{\delta}$,
where $\phi_{\delta}$ is an $L^1$ dilatation of $\phi$, i.e.
$\phi_{\delta} (t):= \delta^{-1} \phi (\delta^{-1} t)$. 
We now prove an analogue of \cite[Lemma 9]{Mag}.

\begin{lem} \label{lem:sm_disc}
Let $c \in (1,2)$ be fixed. Then the following bound holds
\begin{align*} 
D_c(\lambda,\xi) 
\le 
\sup_{0 < a \le 100} |\mathbf D_c(\lambda,\xi; \phi_{a,\delta}^{+})| 
\vee
\sup_{0 < a \le 100} |\mathbf D_c(\lambda,\xi; \phi_{a,\delta}^{-})|  
+ O \big(\delta^{1/2} r_c(\lambda) \big),
\end{align*}
uniformly in $\lambda \in\Z_+$, $\xi \in \S^2$ and $0 < \delta \le 1$.
\end{lem}

\begin{proof}
Since $D_c(\lambda,\xi,a) = 0$ for $a > 100$, we may assume that $0 < a \le 100$. Observe that 
\begin{align*}
\phi_{a,\delta}^{+} (t) 
\le \mathbbm{1}_{[a,100] } (t) \le \phi_{a,\delta}^{-} (t),
\qquad |t| \le 99, \quad 0 < a \le 100, \quad 0 < \delta \le 1.
\end{align*}
Consequently, we have
\begin{align*} 
\sum_{m\in\mathbf S_c^3(\lambda)}  
\phi_{a,\delta}^{+} \Big( \frac{m\cdot \xi}{\lambda^{1/c}} \Big)
&\le 
\sum_{m\in\mathbf S_c^3(\lambda)} 
\mathbbm{1}_{[a,100]} \Big( \frac{m\cdot \xi}{\lambda^{1/c}} \Big)
\le 
\sum_{m\in\mathbf S_c^3(\lambda)} 
\phi_{a,\delta}^{-} \Big( \frac{m\cdot \xi}{\lambda^{1/c}} \Big), \\
r_c(\lambda) \int_{\mathbb S_c^2} 
\phi_{a,\delta}^{-} (x\cdot \xi) \, d\nu_{c} (x) 
& \ge
r_c(\lambda) \int_{\mathbb S_c^2} 
\mathbbm{1}_{[a,100]} (x \cdot\xi) \, d\nu_{c} (x)
\ge 
r_c(\lambda) \int_{\mathbb S_c^2} 
\phi_{a,\delta}^{+} (x \cdot\xi) \, d\nu_{c} (x),
\end{align*}
for $\lambda \in\Z_+$, $\xi \in \S^2$, $0 < a \le 100$ and $0 < \delta \le 1$.
Subtracting the above inequalities we see that
\begin{align*} 
|D_c(\lambda,\xi,a)| 
\le 
|\mathbf D_c(\lambda,\xi; \phi_{a,\delta}^{+})| 
\vee
|\mathbf D_c(\lambda,\xi; \phi_{a,\delta}^{-})| 
+ 
r_c(\lambda) \int_{\mathbb S_c^2} 
| \phi_{a,\delta}^{-} (x\cdot \xi) - \phi_{a,\delta}^{+} (x\cdot \xi) | \, d\nu_{c} (x).
\end{align*} 
Since we have 
\begin{align*} 
\int_{\mathbb S_c^2} 
| \phi_{a,\delta}^{-} (x\cdot \xi) - \phi_{a,\delta}^{+} (x\cdot \xi) | \, d\nu_{c} (x)
\lesssim
\int_{\R} \phi_{\delta} (y) \Big( 
\int_{\mathbb S_c^2} \mathbbm{1}_{[a - \delta + y , a + \delta + y]} (x \cdot\xi)
\, d\mu_{c} (x)
\Big) \, dy,
\end{align*} 
the proof of Lemma~\ref{lem:sm_disc} will be completed if we show
\begin{align*} 
\int_{\mathbb S_c^2} \mathbbm{1}_{[\a , \a + \delta]} (x\cdot \xi)
\, d\mu_{c} (x)
\lesssim
\delta^{1/2}, \qquad \xi \in \S^2, \quad \a \in \R, \quad \delta > 0.
\end{align*} 

Using the decomposition of the
unity on $\mathbb S_c^2$ and \eqref{intS} we see that our problem
reduces to showing that
\begin{align} \label{est33}
\int_{|x_1|^c + |x_2|^c < 3/4} \mathbbm{1}_{[\a , \a + \delta]} 
\big(x_1 \xi_1 + x_2 \xi_2 + \big( 1 - |x_1|^c - |x_2|^c \big)^{1/c} \xi_3 \big)
\, dx
\lesssim
\delta^{1/2}, 
\end{align} 
uniformly in $\xi \in \S^2$, $\a \in \R$, $0 < \delta \le 1$.
By symmetry we may assume that 
$x_1, x_2 > 0$ and $\xi_3 \ge 0$. In what follows we distinguish two cases.

\paragraph{\bf Case 1} $\xi_3 \le 1/100$. Here we have
$|\xi_1| \ge 1/4$ or $|\xi_2| \ge 1/4$ and without any loss of
generality we may assume that $|\xi_2| \ge 1/4$. Then, we see that
\begin{align*} 
\partial_{x_2} \big(x_1 \xi_1 + x_2 \xi_2 + \big( 1 - x_1^c - x_2^c \big)^{1/c} \xi_3 \big)
=
\xi_2 - \xi_3 \big(g_{x_1}(x_2) \big)^{c-1}, 
\qquad x_1^c + x_2^c < 3/4,
\end{align*} 
where $g_b \colon \big( 0, (3/4 - b^c)^{1/c} \big) \to (0,\8)$, and $0 < b < (3/4)^{1/c}$, are defined by
\begin{align*} 
g_b (y) := \frac{y}{ \big( 1 - b^c - y^c \big)^{1/c} },
\qquad b^c + y^c < 3/4, \quad b,y > 0.
\end{align*} 
Since $g_{x_1} (x_2) \le 4$ for $x_1^c + x_2^c < 3/4$ we infer that
\begin{align*} 
\big| \partial_{x_2} \big(x_1 \xi_1 + x_2 \xi_2 + \big( 1 - x_1^c - x_2^c \big)^{1/c} \xi_3 \big) \big|
\simeq 1.
\end{align*} 
This forces that uniformly in $0 < x_1 < (3/4)^{1/c}$ we have
\begin{align*} 
\big| \big\{ 0 < x_2 < (3/4 - x_1^c)^{1/c} : x_1 \xi_1 + x_2 \xi_2 + \big( 1 - x_1^c - x_2^c \big)^{1/c} \xi_3 \in [\a, \a + \delta] \big\} \big|
\lesssim 
\delta
\end{align*} 
and the conclusion follows.

\paragraph{\bf Case 2} $\xi_3 \ge 1/100$. We see that the left-hand side of \eqref{est33} is equal to 
\begin{align*}
\int_{0 < x_1 < (3/4)^{1/c}}
\int_{ x_2^c < 3/4 - x_1^c} 
\mathbbm{1}_{[\a - x_1 \xi_1 , \a + \delta - x_1 \xi_1 ]} 
\big(x_2 \xi_2 + \big( 1 - x_1^c - x_2^c \big)^{1/c} \xi_3 \big)
\, dx_2 \, dx_1. 
\end{align*} 
It suffices to verify that 
\begin{align} \label{est34}
\int_{0 < y < (3/4 - b^c )^{1/c}}
\mathbbm{1}_{[\a , \a + \delta ]} 
\big(y \xi_2 + \big( 1 - b^c - y^c \big)^{1/c} \xi_3 \big)
\, dy
\lesssim
\delta^{1/2}, 
\end{align} 
uniformly in $0 < b < (3/4)^{1/c}$, $\a \in \R$, $0 < \delta \le 1$, $\xi_3 \ge 1/100$ such that $|\xi_2|^2 + |\xi_3|^2 \le 1$. 
Let
\begin{align*}
f(y) := y \xi_2 + \big( 1 - b^c - y^c \big)^{1/c} \xi_3, 
\qquad 0 < y < (3/4 - b^c)^{1/c},
\end{align*} 
and observe that
\begin{align*}
f'(y) = \xi_2 -  \xi_3 ( g_b (y) )^{c-1}, 
\qquad 0 < y < (3/4 - b^c)^{1/c}.
\end{align*} 
We consider two sets
\begin{align*}
I_b &:= \big\{
0 < y < (3/4 - b^c)^{1/c} : |f'(y)| \ge \delta^{1/2}
\big\},\\
J_b &:= \big\{
0 < y < (3/4 - b^c)^{1/c} : |f'(y)| \le \delta^{1/2}
\big\}.
\end{align*} 
Note that $J_b$ is an interval and $I_b$ is a sum of at most two
intervals, which follow from the fact that the function
$y \mapsto g_b (y)$ is increasing. More precisely, we have
\begin{align} \label{est35}
g_b' (y) =
\frac{1 - b^c}{ ( 1 - b^c - y^c )^{1/c + 1} } \simeq 1, 
\qquad 0 < y < (3/4 - b^c)^{1/c}, \quad 0 < b < (3/4)^{1/c}.
\end{align}
By the mean value theorem we conclude
\begin{align*}
| \{ y \in I_b : f(y) \in [\a, \a + \delta] \} | 
\lesssim
\delta^{1/2}, \qquad \a \in \R, \quad 0 < b < (3/4)^{1/c}, \quad 0 < \delta \le 1. 
\end{align*} 
Thus to prove \eqref{est34} it is enough to check that 
\begin{align} \label{est36}
| J_b  | 
\lesssim
\delta^{1/2}, \quad 0 < b < (3/4)^{1/c}, \quad 0 < \delta \le 1. 
\end{align} 
Observe that
\begin{align*}
J_b
=
\Big\{
0 < y < (3/4 - b^c)^{1/c} : 
0 \vee \Big( \frac{\xi_2 - \delta^{1/2} }{\xi_3} \Big)^{1/(c-1)}
\le g_b (y)
\le \Big( \frac{\xi_2 + \delta^{1/2} }{\xi_3} \Big)^{1/(c-1)}
\Big\}.
\end{align*} 
Using \eqref{est35} we see that for all $0 \le A \le B < \8$ and
$0 < b < (3/4)^{1/c}$ one has
\begin{align*}
\big|
\big\{
0 < y < (3/4 - b^c)^{1/c} : 
g_b (y) \in [A,B]
\big\} \big|
\lesssim
B-A,
\end{align*} 
which shows that
\begin{align*}
|J_b|
\lesssim
\big( \xi_2 + \delta^{1/2}  \big)^{1/(c-1)}
-
\big( \xi_2 - \delta^{1/2}  \big)^{1/(c-1)} \vee 0
\lesssim
\delta^{1/2}.
\end{align*} 
This gives \eqref{est36} and consequently leads to \eqref{est33}.
The proof of Lemma~\ref{lem:sm_disc} is completed.
\end{proof}

\begin{lem} \label{lem:FB}
Let $c \ge 1$, $C > 0$ and $M > 1$ be fixed. Then
\begin{align*} 
\int_{\R^3}
\mathbbm{1}_{|x| \le C \lambda^{1/c} } 
\frac{dx}{\big( 1 + A ||x|_c^c - \lambda| \big)^M}
\lesssim
\lambda^{3/c - 1} A^{-1},
\qquad A,\lambda > 0.
\end{align*}
\end{lem}

\begin{proof}
Let $C' > 0$ be a constant such that 
$\{ x\in\R^3 : |x| \le C \lambda^{1/c} \} \subseteq
\{ x\in\R^3 : |x|_c \le C' \lambda^{1/c} \}$. Then, using polar coordinates, see \eqref{polar}, and  changing the variable 
$r \mapsto \lambda^{1/c} t$ we see that the left-hand side in question is controlled by
\begin{align*} 
\int_{0}^{ C' \lambda^{1/c} } 
\frac{r^2 \, dr}{\big( 1 + A |r^c - \lambda| \big)^M}
&\lesssim
\lambda^{3/c}
\int_{0}^{ C' } 
\frac{ dt}{\big( 1 + A \lambda |t^c - 1| \big)^M} \\
&\lesssim
\lambda^{3/c}
\int_{\R} 
\frac{ dt}{\big( 1 + A \lambda |t| \big)^M} \\
&\lesssim
\lambda^{3/c - 1} A^{-1},
\end{align*}
uniformly in $A,\lambda > 0$, and we are done.
\end{proof}

\begin{lem} \label{lem:1}
Let $c \in (1,5/4)$ be fixed. Then for every $\eps > 0$ we have
\begin{align*} 
\mathcal F_{\Z^3}^{-1}\sigma_{\lambda} (\xi)
& =
\lambda^{\kappa} \int_{\R^3} e(x\cdot\xi) \eta \Big( \frac{x}{\lambda^{1/c}} \Big) 
\mathcal{F}^{-1}_{\R} \psi \big( \lambda^{\kappa} ( |x|_c^c - \lambda) \big) \, dx \\
& \quad +
O_\eps \big(\lambda^{3/c - 1 - (5-4c)/(3c) + \eps} \big)
+ \lambda^{3/c - 1} O\big(|\xi| + \lambda^{- 1/(4c)} \big),
\end{align*}
uniformly in $\lambda \ge 1$ and $\xi \in \R^3$.
\end{lem}

\begin{proof}
Using \eqref{iden31} we see that
\begin{align}  \nonumber
\mathcal F_{\Z^3}^{-1}\sigma_\lambda (\xi) & = 
\sum_{m \in \Z^3} e(m\cdot \xi) \sigma_\lambda (m) \\ \label{iden32}
& =
\lambda^{\kappa} \sum_{m \in \Z^3} e(m\cdot \xi)
\eta \Big( \frac{m}{\lambda^{1/c}} \Big) 
\mathcal{F}^{-1}_{\R} \psi \big( \lambda^{\kappa} (Q (m) - \lambda) \big) \\ \nonumber
& \quad +
\sum_{m \in \Z^3} e(m\cdot \xi) 
\eta \Big( \frac{m}{\lambda^{1/c}} \Big) 
\int_{-1/2}^{1/2} e((Q (m) - \lambda) t) 
 \widetilde{\psi}_\lambda (t)  \, dt.
\end{align} 
We first show that 
\begin{align}
\label{estA}
\begin{split}
& \Big| \sum_{m \in \Z^3} e(m\cdot \xi) 
\eta \Big( \frac{m}{\lambda^{1/c}} \Big) 
\int_{-1/2}^{1/2} e((Q (m) - \lambda) t) 
 \widetilde{\psi}_\lambda (t)  \, dt \Big| \\ 
& \qquad \qquad \qquad \qquad \qquad \qquad \lesssim
\lambda^{3/c - 1 - (5-4c)/(3c) } \log(\lambda + 1),
\qquad \lambda \in \Z_+, \quad \xi \in \R^3,
\end{split}
\end{align}
and 
\begin{align}
\label{estB}
\begin{split}
 \sum_{m \in \Z^3} \mathbbm{1}_{|m| \lesssim \lambda^{1/c} } 
\big| 
\mathcal{F}^{-1}_{\R} \psi \big( \lambda^{\kappa} (Q (m) - \lambda) \big)
&-
\mathcal{F}^{-1}_{\R} \psi \big( \lambda^{\kappa} (|m|_c^c - \lambda) \big) \big| \\ 
& \lesssim
\lambda^{3/c - 1}, \qquad \lambda \in \Z_+.
\end{split}
\end{align} 
We first justify \eqref{estA}. Setting
$\lambda_c := (2c)^{-1} (2 \lambda)^{\kappa}$ and using
respectively Lemma~\ref{lem:L2} (with $g = \eta_1$), the
Cauchy--Schwarz inequality and the Plancherel theorem we see that the
left-hand side of \eqref{estA} is controlled by
\begin{align*}
\int_{\lambda_c \le |t| \le 1/2} 
\big| \prod_{j=1}^3 \Pi_{t,\lambda}^{\eta_j} (\xi_j) \big| \, dt
& \lesssim
\lambda^{1/3 + 1/(3c)} \log(\lambda + 1)
\int_{|t| \le 1/2} 
\big| \prod_{j=2}^3 \Pi_{t,\lambda}^{\eta_j} (\xi_j) \big| \, dt \\
& \lesssim
\lambda^{1/3 + 1/(3c)} \log(\lambda + 1)
\prod_{j=2}^3 \Big( \int_{|t| \le 1/2} 
\big| \Pi_{t,\lambda}^{\eta_j} (\xi_j) \big|^2 \, dt \Big)^{1/2} \\
& \lesssim
\lambda^{3/c - 1 - (5-4c)/(3c) } \log(\lambda + 1),
\end{align*}
uniformly in $\lambda \in \Z_+$ and $\xi \in \R^3$ as desired.  Next,
we verify \eqref{estB}. By the mean value theorem and the fact that
\begin{align*}
\frac{1}{1 + \big( \lambda^{\kappa} |\theta - \lambda| \big)^{10}}
\simeq
\frac{1}{1 + \big( \lambda^{\kappa} ||m|_c^c - \lambda| \big)^{10}},
\qquad \lambda \in \Z_+, \quad  m \in \Z^3,
\end{align*}
where $\theta$ is a convex combination of $|m|_c^c$ and $Q(m)$, we see that the left-hand side of \eqref{estB} is dominated by
\begin{align*}
\sum_{m \in \Z^3} 
\frac{\mathbbm{1}_{|m| \lesssim \lambda^{1/c} } \lambda^{\kappa}}{1 + \big( \lambda^{\kappa} ||m|_c^c - \lambda| \big)^{10}},
\end{align*}
which, in turn, by Lemma~\ref{kernel} and 
Lemma~\ref{lem:FB} (with $A = \lambda^{\kappa}$) is controlled by
\begin{align*}
\int_{\R^3} 
\frac{\mathbbm{1}_{|x| \lesssim \lambda^{1/c} } \lambda^{\kappa}\,dx}{1 + \big( \lambda^{\kappa} ||x|_c^c - \lambda| \big)^{10}}
\lesssim
\lambda^{3/c - 1},
\end{align*}
giving the claim in \eqref{estB}.

Combining \eqref{iden32} with \eqref{estA} and \eqref{estB} we see
that  the proof of Lemma~\ref{lem:1} will be completed if we show, uniformly in $\lambda \in\Z_+$ and $\xi \in \R^3$, that
\begin{align}
\label{estC}
\begin{split}
& \lambda^{\kappa} \Big| \sum_{m \in \Z^3} e(m \cdot\xi)
\eta \Big( \frac{m}{\lambda^{1/c}} \Big) 
\mathcal{F}^{-1}_{\R} \psi \big( \lambda^{\kappa} (|m|_c^c - \lambda) \big) \\ 
& \qquad \qquad \qquad \qquad \qquad 
- \int_{\R^3} e(x \cdot\xi)
\eta \Big( \frac{x}{\lambda^{1/c}} \Big) 
\mathcal{F}^{-1}_{\R} \psi \big( \lambda^{\kappa} (|x|_c^c - \lambda) \big) \, dx
\Big| \\
&\qquad \qquad \qquad \qquad \qquad \qquad \qquad  \qquad \lesssim 
\lambda^{3/c - 1} \big( |\xi| + \lambda^{- 1/(4c)} \big).
\end{split}
\end{align}  
For $x \in m + [0,1]^3$, $\lambda \in\Z_+$, and $\xi \in \R^3$, we observe that
\begin{align*}
\Big| e(m\cdot \xi) 
\eta \Big( \frac{m}{\lambda^{1/c}} \Big)  - e(x\cdot \xi) 
\eta \Big( \frac{x}{\lambda^{1/c}} \Big) \Big|
\lesssim
\big( |\xi| + \lambda^{- 1/c} \big) \mathbbm{1}_{|x| \lesssim \lambda^{1/c} }.
\end{align*}
Applying the mean value theorem we obtain
\begin{align*}
& \eta \Big( \frac{x}{\lambda^{1/c}} \Big) 
\Big|
\mathcal{F}^{-1}_{\R} \psi \big( \lambda^{\kappa} (|m|_c^c - \lambda) \big)
-
\mathcal{F}^{-1}_{\R} \psi \big( \lambda^{\kappa} (|x|_c^c - \lambda) \big)
\Big| \\
& \qquad \qquad \qquad \qquad \lesssim 
\frac{\mathbbm{1}_{|x| \lesssim \lambda^{1/c} }
\lambda^{\kappa} \lambda^{(c-1)/c}}{1 + \big( \lambda^{\kappa} |\theta_{m,x} - \lambda| \big)^{10}} \\
& \qquad \qquad \qquad \qquad \simeq
\frac{\mathbbm{1}_{|x| \lesssim \lambda^{1/c} }
\lambda^{-1/(4c)} }{1 + \big( \lambda^{\kappa} ||x|_c^c - \lambda| \big)^{10}},
\qquad x \in m + [0,1]^3, \quad \lambda \in\Z_+, 
\end{align*}
where $\theta_{m,x}$ is a convex combination of $|m|_c^c$ and $|x|_c^c$;
here we have used the fact that 
\[
||x|_c^c - |m|_c^c| \lesssim \lambda^{(c-1)/c} \le \lambda^{-\kappa}
\]
uniformly in $|m| \lesssim \lambda^{1/c}$, $x \in m + [0,1]^3$,
$\lambda \in\Z_+$.  Applying these estimates together with
Lemma~\ref{lem:FB} (with $A = \lambda^{\kappa}$) we see
that the left-hand side of \eqref{estC} is controlled by
\begin{align*}
\lambda^{\kappa}
\int_{\R^3} \mathbbm{1}_{|x| \lesssim \lambda^{1/c} } 
\frac{\big( |\xi| + \lambda^{- 1/(4c)} \big) dx}{1 + \big( \lambda^{\kappa} ||x|_c^c - \lambda| \big)^{10}}
\lesssim
\lambda^{3/c - 1} \big( |\xi| + \lambda^{- 1/(4c)} \big),
\end{align*}
as desired.
This finishes the proof of Lemma~\ref{lem:1}.
\end{proof}

\begin{lem} \label{lem:FA}
We have the following estimates
\begin{align}
\label{eq:30}
\int_{\R} | \mathcal{F}_{\R} \phi_{a,\delta}^{\pm} (t) | \, dt
\lesssim
\log(1 + \delta^{-1}), \qquad 0 < a \le 100, \quad 0 < \delta \le 1/2,
\end{align}
and
\begin{align}
\label{eq:31}
\int_{\R} (1 + |t|) | \mathcal{F}_{\R} \phi_{a,\delta}^{\pm} (t) | \, dt
\lesssim
\delta^{-1}, \qquad 0 < a \le 100, \quad 0 < \delta \le 1/2.
\end{align}
\end{lem}
\begin{proof}
Since $\phi_{a,\delta}^{\pm} = \mathbbm{1}_{[a \pm \delta,100]} \ast \phi_{\delta}$, we have
\begin{align*}
\mathcal{F}_{\R} \phi_{a,\delta}^{\pm} (t) 
=
\mathcal{F}_{\R} (\mathbbm{1}_{[a \pm \delta,100]}) (t) 
\mathcal{F}_{\R} (\phi_{\delta}) (t)
=
\mathcal{F}_{\R} (\mathbbm{1}_{[a \pm \delta,100]}) (t) 
\mathcal{F}_{\R} (\phi) (\delta t).
\end{align*}
Taking into account the estimate 
\begin{align*}
| \mathcal{F}_{\R} (\mathbbm{1}_{[a , b]}) (t) |
\lesssim
(1 + |t|)^{-1}, \qquad t \in \R, \quad - \infty < a < b \le 100,
\end{align*}
we see that for every $M \in \Z_+$ we have
\begin{align*}
| \mathcal{F}_{\R} \phi_{a,\delta}^{\pm} (t)  |
\lesssim_M
(1 + |t|)^{-1} (1 + \delta |t|)^{-M}, 
\qquad t \in \R, \quad 0 < a \le 100, \quad \delta > 0.
\end{align*}
To prove the estimate in \eqref{eq:30} it suffices to use the above
estimate and split the integral in question into two pieces
$|t| \le \delta^{-1}$ and $|t| > \delta^{-1}$. Then, the conclusion
easily follows.  The proof of \eqref{eq:31} is even simpler. It is
enough to make the change of variable $\delta t \mapsto t$.
\end{proof}

Now we are ready to prove Theorem~\ref{thm:disc}.

\begin{proof}[{Proof of Theorem~\ref{thm:disc}}]
We set $\delta := \lambda^{-1/(2c)}$. Using the inverse
Fourier transform formula and applying Lemma~\ref{lem:1} and
Lemma~\ref{lem:FA} we infer that
\begin{align}
\label{iden33}
\begin{split}
\sum_{m\in\mathbf S_c^3 (\lambda)} 
\phi_{a,\delta}^{\pm} \Big( \frac{m \cdot \xi}{\lambda^{1/c}} \Big)
& =
\int_{\R} \mathcal{F}_{\R} \phi_{a,\delta}^{\pm} (t) 
\mathcal F_{\Z^3}^{-1}\sigma_{\lambda} \Big( \frac{t \xi}{\lambda^{1/c}} \Big) \, dt \\ 
 =I_\lambda(\xi)& +
O_\eps \big(\lambda^{3/c - 1 - (5-4c)/(3c) + \eps} \big)
+ \lambda^{3/c - 1} O\big(\lambda^{- 1/(4c)} \log(\lambda + 1) \big),
\end{split}
\end{align}
uniformly in $\lambda \ge 1$, $\xi \in \S^2$, $0 < a \le 100$, where
\begin{align*}
I_\lambda(\xi):=
\lambda^{\kappa}
\int_{\R} \mathcal{F}_{\R} \phi_{a,\delta}^{\pm} (t)
\int_{\R^3} e(\lambda^{ - 1/c}x\cdot t\xi) \eta \Big( \frac{x}{\lambda^{1/c}} \Big) 
\mathcal{F}^{-1}_{\R} \psi \big( \lambda^{\kappa} (|x|_c^c - \lambda) \big) 
\, dx \, dt.
\end{align*}
Changing the variable $x \mapsto \lambda^{1/c} y$ and using
the polar decomposition, see \eqref{polar}, we obtain
\begin{align*} 
I_{\lambda}(\xi)
& =
\lambda^{3/c+\kappa}
\int_{\R} \mathcal{F}_{\R} \phi_{a,\delta}^{\pm} (t)
\int_{\R^3} e(y \cdot t\xi) \eta (y) 
\mathcal{F}^{-1}_{\R} \psi \big( \lambda^{\kappa + 1} (|y|_c^c - 1) \big) 
\, dy \, dt \\
& =
\lambda^{3/c+\kappa}
\int_{\R} \mathcal{F}_{\R} \phi_{a,\delta}^{\pm} (t)
\int_{0}^{\8} r^2 \int_{\mathbb S_c^2}
 e(r w \cdot t\xi) \eta (rw) 
\mathcal{F}^{-1}_{\R} \psi \big( \lambda^{\kappa + 1} (r^c - 1) \big) 
\, d\mu_c(w) \, dr \, dt.
\end{align*}
We now show that
\begin{align}
\label{est31}
\begin{split}
& \int_{\R} | \mathcal{F}_{\R} \phi_{a,\delta}^{\pm} (t) |
\Big| \int_{0}^{\8} r^2
 e(r w \cdot t\xi) \eta (rw) 
\mathcal{F}^{-1}_{\R} \psi \big( \lambda^{\kappa + 1} (r^c - 1) \big) \, dr \\ 
& \qquad \qquad \qquad \qquad \qquad \qquad- 
\int_{0}^{\8} e(w \cdot t\xi) 
\mathcal{F}^{-1}_{\R} \psi \big( \lambda^{\kappa + 1} (r^c - 1) \big) \, dr 
\Big| \, dt
\lesssim
\lambda^{-1/c},
\end{split}
\end{align}
uniformly in $\lambda \in \Z_+$, $\xi \in \S^2$, $0 < a \le 100$ and 
$w \in \mathbb S_c^2$. Indeed, by definition $\eta (w) = 1$ for 
$w \in \mathbb S_c^2$ and we see
\begin{align*} 
| r^2 e(r w \cdot t\xi) \eta (rw) 
- e(w \cdot t\xi) |
\lesssim
|r-1| (1 + |t|), \qquad t \in \R, \quad r > 0, \quad \xi \in \S^2, \quad w \in \mathbb S_c^2.
\end{align*}
This together with \eqref{eq:31} shows that the left-hand side of \eqref{est31} is bounded by
\begin{align*} 
\lambda^{1/(2c)}
\int_0^{\8} 
\frac{|r - 1| \, dr}{1 + \big| \lambda^{ \kappa+1}  ( r^c - 1) \big|^3 }
\lesssim
\lambda^{1/(2c)}
\int_{\R} 
\frac{|r| \, dr}{1 + \big| \lambda^{ \kappa+1}  r \big|^3}
\lesssim 
\lambda^{-1/c},
\end{align*}
which is the asserted estimate.
Let
\begin{align*} 
J_{\lambda}(\xi)
:=
\lambda^{3/c+\kappa}
\int_{\R} \mathcal{F}_{\R} \phi_{a,\delta}^{\pm} (t)
\int_{0}^{\8} \int_{\mathbb S_c^2}
 e(w \cdot t\xi)
\mathcal{F}^{-1}_{\R} \psi \big( \lambda^{\kappa + 1} (r^c - 1) \big) 
\, d\mu_c(w) \, dr \, dt.
\end{align*}
Then estimate \eqref{est31} yields
\begin{align} \label{est32}
| I_{\lambda}(\xi) - J_{\lambda}(\xi) |
\lesssim
\lambda^{3/c -1 - 1/(4c) },
\qquad \lambda \in\Z_+, \quad \xi \in \S^2, \quad 0 < a \le 100.
\end{align}
Changing the variable $r^c \mapsto s + 1$ we infer that for every $M \in \Z_+$  we have
\begin{align} \nonumber
J_{\lambda}(\xi)
& =
c^{-1} \lambda^{3/c+\kappa}
\int_{\R} \mathcal{F}_{\R} \phi_{a,\delta}^{\pm} (t)
\int_{-1/2}^{1/2} \int_{\mathbb S_c^2}
 e(w \cdot t\xi)
\mathcal{F}^{-1}_{\R} \psi \big( \lambda^{\kappa + 1} s \big) 
\, d\mu_c(w) (s+1)^{1/c - 1} \, ds \, dt \\ \nonumber
& \quad + 
O_M(\lambda^{-M}) \\ \label{iden34}
& =
c^{-1} \lambda^{3/c+\kappa}
\int_{\R} \mathcal{F}_{\R} \phi_{a,\delta}^{\pm} (t)
\int_{-1/2}^{1/2} \int_{\mathbb S_c^2}
e( w \cdot t\xi)
\mathcal{F}^{-1}_{\R} \psi \big( \lambda^{\kappa + 1} s \big) 
\, d\mu_c(w) \, ds \, dt \\ \nonumber
& \quad + 
O\big( \lambda^{3/c - 1 - 1/(4c) } \big),
\end{align} 
uniformly in $\lambda \in\Z_+$, $\xi \in \S^2$ and $0 < a \le 100$.

On the other hand, using the formula
\begin{align*} 
\int_{\mathbb S_c^2} 
\phi_{a,\delta}^{\pm} (x \cdot \xi) \, d\nu_{c} (x)
=
\int_{\R} \mathcal{F}_{\R} \phi_{a,\delta}^{\pm} (t)
\mathcal{F}^{-1}_{\R^3} (\nu_{c}) (t \xi) \, dt,
\end{align*}
and Corollary~\ref{cor:asym} we see that for every $\eps > 0$ we have
\begin{align*} 
r_c(\lambda) \int_{\mathbb S_c^2} 
\phi_{a,\delta}^{\pm} (x \cdot \xi) \, d\nu_{c} (x)
=
c^{-1} \lambda^{3/c- 1}
\int_{\R} \mathcal{F}_{\R} \phi_{a,\delta}^{\pm} (t)
\mathcal{F}^{-1}_{\R^3} (\mu_{c}) (t \xi) \, dt
+
O_{\eps} \big(  \lambda^{3/c - 1 - (9 - 8c)/(5c) + \eps} \big),
\end{align*}
uniformly in $\lambda \in \Z_+$, $\xi \in \S^2$ and $0 < a \le 100$.
Combining this with \eqref{iden33}, \eqref{est32} and \eqref{iden34},
and taking into account Lemma~\ref{lem:sm_disc} it suffices to verify
that
\begin{align*} 
\Big|
\lambda^{\kappa + 1}
\int_{\R} \mathcal{F}_{\R} \phi_{a,\delta}^{\pm} (t)
\int_{-1/2}^{1/2} \mathcal{F}^{-1}_{\R^3} (\mu_c) (t \xi)
\mathcal{F}^{-1}_{\R} \psi \big( \lambda^{\kappa + 1} s \big) 
\, ds \, dt 
-
\int_{\R} \mathcal{F}_{\R} \phi_{a,\delta}^{\pm} (t)
\mathcal{F}^{-1}_{\R^3} (\mu_c) (t \xi) \, dt \Big|
\lesssim
\lambda^{- 1 },
\end{align*} 
uniformly in $\lambda \in\Z_+$, $\xi \in \S^2$ and $0 < a \le 100$.
Changing the variable $\lambda^{\kappa + 1} s \mapsto u$ in the first
integral above, using the estimate
$|\mathcal{F}^{-1}_{\R} (\mu_c) (\zeta)| \lesssim 1$ for
$\zeta \in \R^3$, and Lemma~\ref{lem:FA} (a), we see that the above
estimate will follow once we show that
\begin{align*} 
\Big|
\int_{- \lambda^{\kappa + 1}/2 }^{ \lambda^{\kappa + 1}/2 } 
\mathcal{F}^{-1}_{\R} \psi (u) 
\, du
- 1 \Big|
\lesssim
\lambda^{- 2 },
\qquad \lambda \ge 1.
\end{align*} 
This, however, is a direct consequence of the fact that 
$\int_{\R} \mathcal{F}^{-1}_{\R} \psi(t)dt = \psi(0) = 1$.
This completes the proof of Theorem~\ref{thm:disc}.
\end{proof}

Finally we prove our equidistribution theorem.

\begin{proof}[Proof of Theorem \ref{thm:equi}]
Here we will proceed in a similar way as in the proof of
Theorem~\ref{thm:disc}. For the  convenience of the reader we give a sketch of
the proof. Since only  the values of $\phi$ on a neighborhood of
$\mathbb S_c^2$ play a role, without any loss of generality we may
assume that $\phi \in C_c (\mathbb{R}^3)$.  By
Corollary~\ref{cor:asym} one can replace $r_c(\lambda)$ by
$\big(\frac{2}{c}\big)^3\frac{\Gamma(1/c)^3 }{\Gamma(3/c)} \lambda^{3/c - 1}$
on the left-hand side of \eqref{eq:47}.  Now using the inverse Fourier transform
formula we obtain
\begin{align*} 
\sum_{x \in \mathbf P_{c}^3(\lambda) } \phi (x)
& =
\int_{\R^3} \mathcal{F}_{\R^3} \phi (\xi) 
\mathcal F_{\Z^3}^{-1}\sigma_{\lambda} \Big( \frac{\xi}{\lambda^{1/c}} \Big) \, d\xi, \\
\int_{\mathbb S_c^2} 
\phi (w) \, d\nu_{c} (w)
& =
\int_{\R^3} \mathcal{F}_{\R^3} \phi (\xi)
\mathcal{F}_{\R^3}^{-1} (\nu_{c}) (\xi) \, d\xi.
\end{align*}
By Remark \ref{rem:10} we have
$\nu_{c} = \frac{c^2 \Gamma(3/c)}{8 \Gamma(1/c)^3} \mu_{c}$, we see
that our task is reduced to showing that
\begin{align*} 
c \lambda^{-3/c + 1}
\int_{\R^3} \mathcal{F}_{\R^3} \phi (\xi) 
\mathcal F_{\Z^3}^{-1}\sigma_{\lambda} \Big( \frac{\xi}{\lambda^{1/c}} \Big) \, d\xi
\xrightarrow[\lambda \to\8]{} 
\int_{\R^3} \mathcal{F}_{\R^3} \phi (\xi)
\mathcal{F}_{\R^3}^{-1} (\mu_{c}) (\xi) \, d\xi.
\end{align*} 
Now applying Lemma~\ref{lem:1} we see that we can further reduce our problem to proving that
\begin{align} \label{red122}
I_{\lambda} 
\xrightarrow[\lambda \to\8]{} 
\int_{\R^3} \mathcal{F}_{\R^3} \phi (\xi)
\mathcal{F}_{\R^3}^{-1} (\mu_{c}) (\xi) \, d\xi,
\end{align} 
where
\begin{align*} 
I_{\lambda} :
=
c \lambda^{-3/c + 1 + \kappa}
\int_{\R^3} \mathcal{F}_{\R^3} \phi (\xi) 
\int_{\R^3} e\Big( \frac{x \cdot \xi}{\lambda^{1/c}} \Big) 
\eta \Big( \frac{x}{\lambda^{1/c}} \Big) 
\mathcal{F}^{-1}_{\R} \psi \big( \lambda^{\kappa} (|x|_c^c - \lambda) \big) 
\, dx \, d\xi.
\end{align*} 
Changing the variable $x \mapsto \lambda^{1/c} y$ and using
the polar decomposition \eqref{polar}, we obtain
\begin{align*} 
I_{\lambda} 
& =
c \lambda^{1 + \kappa}
\int_{\R^3} \mathcal{F}_{\R^3} \phi (\xi) 
\int_{\R^3} e( y \cdot \xi) 
\eta (y) 
\mathcal{F}^{-1}_{\R} \psi \big( \lambda^{\kappa+1} (|y|_c^c - 1) \big) 
\, dy \, d\xi \\
& =
c \lambda^{1 + \kappa}
\int_{\R^3} \mathcal{F}_{\R^3} \phi (\xi) 
\int_{0}^{\8} r^2 \int_{\mathbb S_c^2} e( rw \cdot \xi) 
\eta (rw) 
\mathcal{F}^{-1}_{\R} \psi \big( \lambda^{\kappa+1} (r^c - 1) \big) 
\, d\mu_c(w) \, dr \, d\xi.
\end{align*} 
Applying the bound
\begin{align*} 
| r^2 e(r w \cdot \xi) \eta (rw) 
- e(w \cdot \xi) |
\lesssim
|r-1| (|\xi| + 1), \qquad r > 0, \quad \xi \in \R^3, \quad w \in \mathbb S_c^2,
\end{align*}
and changing the variable $r^c \mapsto s + 1$
we see that
\begin{align*} 
I_{\lambda} 
& =
c \lambda^{1 + \kappa}
\int_{\R^3} \mathcal{F}_{\R^3} \phi (\xi) \mathcal{F}_{\R^3}^{-1} (\mu_{c}) (\xi)
\int_{0}^{\8}  
\mathcal{F}^{-1}_{\R} \psi \big( \lambda^{\kappa+1} (r^c - 1) \big) 
 \, dr \, d\xi
+ O(\lambda^{- (1 + \kappa)}) \\
& =
\lambda^{1 + \kappa}
\int_{\R^3} \mathcal{F}_{\R^3} \phi (\xi) \mathcal{F}_{\R^3}^{-1} (\mu_{c}) (\xi)
\int_{-1}^{\8}  
\mathcal{F}^{-1}_{\R} \psi \big( \lambda^{\kappa+1} s \big) 
(s+1)^{1/c - 1} \, ds \, d\xi
+ O(\lambda^{- (1 + \kappa)}) \\
& =
\lambda^{1 + \kappa}
\int_{\R^3} \mathcal{F}_{\R^3} \phi (\xi) \mathcal{F}_{\R^3}^{-1} (\mu_{c}) (\xi)
\int_{-1/2}^{1/2}  
\mathcal{F}^{-1}_{\R} \psi \big( \lambda^{\kappa+1} s \big) 
\, ds \, d\xi
+ O(\lambda^{- (1 + \kappa)}).
\end{align*}
Since $\psi(0)=1$ we have
\begin{align*} 
\lim_{\lambda \to \infty} 
\lambda^{1 + \kappa}
\int_{-1/2}^{1/2}  
\mathcal{F}^{-1}_{\R} \psi \big( \lambda^{\kappa+1} s \big) 
\, ds
&=
\lim_{\lambda \to \infty}
\int_{- \lambda^{\kappa + 1}/2 }^{ \lambda^{\kappa + 1}/2 } 
\mathcal{F}^{-1}_{\R} \psi (u) 
\, du\\
&=
\int_{\R} \mathcal{F}^{-1}_{\R} \psi(u) \, du \\
&= \psi(0),
\end{align*} 
using the dominated convergence theorem we see that 
the identity \eqref{red122} follows.
\end{proof}

%%%%%%%%%%%%%%%%%%%%%%%%%%%%%%%%%%%%%%%%%%%%%%%%%

\end{document}